%% file: ActuatorMoving_arx.tex
\documentclass[reqno]{amsart}

\usepackage{amscd,amsthm,amsmath,amssymb,mathtools}
\usepackage{verbatim}
\usepackage{cite}
\usepackage{xcolor}
\usepackage[colorlinks=true]{hyperref}
\usepackage{url}
\usepackage{enumerate,enumitem}
\usepackage{graphicx}
 \usepackage{algorithm}
 \usepackage{algorithmic}

\usepackage{epstopdf,epsfig,subfigure}
\usepackage{curve2e}

  \usepackage[olditem,oldenum]{paralist}

\theoremstyle{plain}

\newtheorem*{mainresult}{Main Result}
\newtheorem{theorem}{Theorem}[section]
\newtheorem{proposition}[theorem]{Proposition}
\newtheorem{lemma}[theorem]{Lemma}
\newtheorem{corollary}[theorem]{Corollary}

\newtheorem{assumption}[theorem]{Assumption}

\theoremstyle{definition}

\newtheorem{remark}[theorem]{Remark}
\newtheorem{example}[theorem]{Example}

\numberwithin{equation}{section}

\input{Mathcommands}

\begin{document}
\title{Stabilization of nonautonomous parabolic equations by a single moving actuator}
\author{Behzad Azmi$^{\tt1}$}
\author{Karl Kunisch$^{\tt1,\tt2}$}
\author{S\'ergio S.~Rodrigues$^{\tt1}$}

 \thanks{
\vspace{-1em}\newline\noindent
\newline\noindent
$^{\tt1}$ Johann Radon Institute for Computational and Applied Mathematics,
  Altenbergerstr. 69, 4040 Linz, Austria.
  \newline\noindent
 $^{\tt2}$ Institute for Mathematics and Scientific Computing, University of Graz, Heinrichstr. 36, 8010 Graz, Austria.
    \newline\noindent
  {\sc Emails}:
 ({\tt behzad.azmi,sergio.rodrigues)@ricam.oeaw.ac.at},\quad{\tt  karl.kunisch@uni-graz.at}.
}

%{\normalfont\ss}
\begin{abstract}
It is shown that an internal control based on a moving indicator function
is able to stabilize the state of parabolic equations evolving in rectangular
domains. For proving the stabilizability result, we start with a control obtained
from an oblique projection feedback based on a finite number of static actuators,
then we used the continuity of the state when the control varies in relaxation metric to construct a switching control where
at each given instant of time only one of the static actuators is active, finally we construct the moving control by
traveling between the static actuators.

Numerical computations are performed by a concatenation
procedure following a receding horizon control approach. They confirm the stabilizing performance of the moving control. 
\end{abstract}

\maketitle

%
% {\scriptsize \vspace*{-1.75cm}
% \tableofcontents
% }

\pagestyle{myheadings} \thispagestyle{plain} \markboth{\sc B. Azmi, K. Kunisch, and S. S.
Rodrigues}{\sc A single moving stabilizing actuator}

\section{Introduction}

Stabilizability of controlled parabolic-like equations of the form
 \begin{align}\label{sys-y-Intro}
 \dot y+Ay+A_{\rm rc}(t)y=u(t)\Phi(t),\qquad y(0)=y_0,\qquad t>0,
\end{align}
where the state evolves in a Hilbert space~$H$, that is,~$y(t)\in H$ for all~$t\ge0$ is investigated.
The pair~$(u,\Phi)$, with
$u(t)\in\bbR$ and~$\Phi(t)\in H$, with $\norm{\Phi(t)}{H}=1$, is at our disposal.
We shall look for a continuous function~$\Phi\colon[0,+\infty)\to H$, where ~$\Phi(t)$
represents the actuator moving on a compact subset of the unit sphere~$\fkS_H$ in~$H$.

Under suitable assumptions on the operators~$A$ and~$A_{\rm rc}$, to be specified later, and under a suitable stabilizability assumption by means of a finite (possibly large) number of static/fixed actuators the main result of this manuscript is the following. 
\begin{mainresult}
There exist a (signed) magnitude control function~$u=u(t)$ and a continuous
moving actuator~$\Phi=\Phi(t)$ satisfying
\begin{align}
 &u\in L^2((0,+\infty),\bbR),\qquad \Phi(t)\in\fkS_H\quad\mbox{for}\quad t\ge0,\notag\\
& \dot\Phi\in L^\infty((0,+\infty),H){\,\textstyle\bigcap\,}\clC([0,+\infty),H),
 \quad \ddot\Phi\in L^\infty((0,+\infty),H),\notag
\end{align}
and constants~$C\ge1$ and $\mu>0$, such
 that the solution of the system~\eqref{sys-y-Intro}
 satisfies
\begin{align}
 &\norm{y(t)}{H}\le C\ex^{-\mu t}\norm{y_0}{H},\quad\mbox{for all}\quad t\ge0,\label{goal-Intro}
 \intertext{and the mapping~$y_0\mapsto u(y_0)$ is continuous,}
 &\norm{u}{\clC( H,L^2(\bbR_0,\bbR))}<+\infty.\label{goal-Intro-L2}
\end{align}
 \end{mainresult}

Note that, in particular, the  actuator moves in a regular way, with continuous  ``velocity''~$\dot\Phi$, which is meaningful from the applications/physical point of view.

The precise statement of Main Result is given in Corollary~\ref{C:main1}.

\subsection{Example}
As an illustration we consider a parabolic equation whose state evolves in~$H=L^2(\Omega)$, with~$\Omega\subset\bbR^d$, $d\in\{1,2,3\}$,
 a regular bounded domain.
\begin{align}\label{sys-y-parab-Intro}
 \dot y-\nu\Delta y+ay+b\cdot\nabla y=u
 \widehat\indf_{\omega(c)},\quad\clG y\rest{\Gamma}=0,\quad y(0,\Bigcdot)=y_0,
\end{align}
where ~$y=y(t,x)\in\bbR$, $y(t,\Bigcdot)\in L^2(\Omega)$,
$a=a(t,x)\in\bbR$, $b=b(t,x)\in\bbR^d$, and~$\clG$ denotes either Dirichlet or Neumann conditions
on the boundary~$\Gamma$ of~$\Omega$, i.e.  $\clG y\rest{\Gamma}=y(t,\bar x)$ or~$\clG y\rest{\Gamma}=\bfn(\bar x)\cdot\nabla y(t,\bar x)$,
where~$\bfn(\bar x)$ stands for the unit outward vector normal at~$\overline x\in\Gamma$.

We shall apply the abstract Main Result to the more concrete system~\ref{sys-y-parab-Intro}, after writing the later in the form~\eqref{sys-y-Intro}. For this purpose it will be enough to take the operators~$A=-\nu\Delta+\Id$ and~$A_{\rm rc}(t)=(a(t,\Bigcdot)-1)\Id+b(t,\Bigcdot)\cdot\nabla$, and the  actuator chosen  as ~$\Phi(t)=\widehat\indf_{\omega(c(t))}$, where~
$\widehat\indf_{\omega(c(t))}$ denotes  the normalized indicator function whose support is the
rectangle~$\overline{\omega(c(t))}$. This rectangle $\omega(c(t))\coloneqq c(t)+\omega_0\subset\Omega$ is the  translation of
a rectangular reference domain~$\omega_0\subset\bbR^d$, with~$0\in\omega_0$, and~$\dnorm{\omega_0}{}\coloneqq\int_{\omega_0}1\,\ed\bbR^d$. Then
\[
 \widehat\indf_{\omega(c(t))}(x)\coloneqq\begin{cases}
                                          \dnorm{\omega_0}{}^{-\frac12},\;&\mbox{if}\quad x\in\omega(c(t)),\\
                                          0,\;&\mbox{if}\quad x\notin\overline{\omega(c(t))},
                                         \end{cases}\qquad\mbox{and}\qquad\norm{\widehat\indf_{\omega(c(t))}}{L^2(\Omega)}=1.
\]

To simplify the exposition let us also assume that~$0$ is the center of mass of~$\omega_0$, so that we can simply say that
$0$ is the center of~$\omega_0$. Since~$c(t)\in\omega(c(t))$, this  justifies to call ~$c(t)\in\bbR^d$  the {\em center} of the actuator.
Hence the motion of the actuator~$\Phi(t)=\widehat\indf_{\omega(c(t))}$ is  described  by the center of~$\omega(c(t))$.
See Figure~\ref{fig.movingAct}, where we have taken~$\omega_0\subset\bbR^d$ as a small rectangular domain.

\setlength{\unitlength}{.003\textwidth}
\newsavebox{\Rectfw}%
\savebox{\Rectfw}(0,0){%
\linethickness{1.5pt}
{\color{black}\polygon(0,0)(120,0)(120,80)(0,80)(0,0)}%
}%
\newsavebox{\Rectfg}%
\savebox{\Rectfg}(0,0){%
{\color{lightgray}\polygon*(0,0)(120,0)(120,80)(0,80)(0,0)}%
}%

\newsavebox{\Rectref}%
\savebox{\Rectref}(0,0){%
{\color{white}\polygon*(0,0)(120,0)(120,80)(0,80)(0,0)}%
{\color{lightgray}\polygon*(45,30)(75,30)(75,50)(45,50)(45,30)}%
}%

\newsavebox{\Circref}%
\savebox{\Circref}(0,0){%
{\color{lightgray}\circle*{1}}%
}
%%%%%%%%%%%%%%%%%%%%%%%%%%

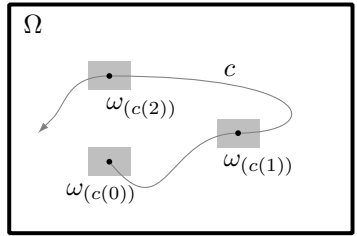
\begin{figure}[th!]
\begin{center}
\begin{picture}(500,100)%(0,0)
% %Rectref
 \put(100,0){\usebox{\Rectfw}}%
% %
 \put(105,5){\scalebox{.5}{\usebox{\Rectref}}}
 \put(150,15){\scalebox{.5}{\usebox{\Rectref}}}
 \put(105,35){\scalebox{.5}{\usebox{\Rectref}}}

 {\color{gray}
 \moveto(110,15)%
 \cbezier(135,25)(155,0)(160,35)(180,35)%
 \cbezier(180,35)(210,35)(200,55)(135,55)%
 \cbezier(135,55)(120,55)(120,45)(115,40)%
  \put(115,40){\vector(-1,-1){5}}
 }
 \put(135,25){\circle*{2}}
 \put(180,35){\circle*{2}}
 \put(135,55){\circle*{2}}

 \put(120,13.5){$\omega_{(c(0))}$}
 \put(175,23.5){$\omega_{(c(1))}$}
 \put(134,43.5){$\omega_{(c(2))}$}

 \put(175,55){$c$}

 \put(105,70){$\Omega$}
\end{picture}
\end{center}
\caption{An internal moving actuator with support~$\overline{\omega(c(t))}\subset\overline{\Omega}$.} \label{fig.movingAct}
\end{figure}

The main result of this paper, when applied to~\eqref{sys-y-parab-Intro},
implies the following Theorem~\ref{T:main-Intro}, concerning parabolic equations evolving in the bounded rectangular domain
\begin{subequations}\label{subsetsOmega-Intro}
\begin{align}
 \Omega&\coloneqq\bigtimes_{n=1}^d(0,L_n)\subset\bbR^d,\qquad L\coloneqq (L_1,L_2,\dots,L_d)\in(0,+\infty)^d\subset\bbR^d.
 \label{subsetsOmega-Intro.0}
\intertext{For any given~$r\in[0,1]$ we further  define the subsets}
 r\Omega&\coloneqq\bigtimes_{n=1}^d(0,rL_n)\subset\Omega,
 \qquad(1-r)\Omega+\tfrac{r}2L\coloneqq \bigtimes_{n=1}^d(\tfrac{r}{2}L_n,L_n-\tfrac{r}{2}L_n)
 \subset\Omega,\\
 \omega_0&\coloneqq r\Omega-\tfrac{r}{2}L=\bigtimes_{n=1}^d(-\tfrac{r}{2}L_n,\tfrac{r}{2}L_n).
\end{align}
\end{subequations}
 Observe that~$c+\omega_0\subset\Omega$ if, and only if, $c\in(1-r)\Omega+\tfrac{r}2L$.

\begin{theorem}\label{T:main-Intro}
 Let~$\Omega$ be a bounded
 rectangular domain  as in~\eqref{subsetsOmega-Intro.0}, and let
 $a\in L^\infty((0,+\infty)\times\Omega,\bbR)$ and~$b\in
 L^\infty((0,+\infty)\times\Omega,\bbR^d)$. Then for each  sufficiently small~$r\in(0,1)$, and
 for each  initial state~$y_0\in L^2(\Omega)$, each initial
 actuator position~$c(0)=c_0\in(1-r)\Omega+\tfrac{r}2L$ with initial actuator velocity~$\dot c(0)=0\in\bbR^d$,
 there exists an actuator motion function~$c$ and a magnitude control function~$u$, with
 \[
  \black    c(t)\in(1-r)\Omega+\tfrac{r}2L,\qquad\mbox{and}\qquad u(t)\in\bbR
 \]
such that the solution of~\eqref{sys-y-parab-Intro}, with
 \[
  \omega(c(t))\coloneqq r\Omega+c(t)-\tfrac{r}{2}L=\bigtimes_{n=1}^d(c_n(t)-\tfrac{r}{2}L_n,c_n(t)+\tfrac{r}{2}L_n)
 \]
 satisfies
\begin{subequations}\label{goal.parab-Intro}
 \begin{align}\label{goal.parab-Intro-exp}
 &\norm{ y(t,\Bigcdot)}{L^2(\Omega)}\le C\ex^{-\mu t}\norm{y_0}{L^2(\Omega)},\qquad\mbox{for all}\qquad t\ge0,
 \intertext{with}
 &u\in L^2((0,+\infty),\bbR),\quad\dot c\in L^\infty((0,+\infty),\bbR^d),\quad\mbox{and}\quad
 \ddot c\in L^\infty((0,+\infty),\bbR^d).\label{goal.parab-Intro-L2}
\end{align}
\end{subequations}
Furthermore, the mapping~$y_0\mapsto u(y_0)$ is continuous from~$L^2(\Omega)$
into~$L^2((0,+\infty),\bbR)$ with  $\norm{u(y_0)}{L^2((0,+\infty),\bbR)}\le C_u\norm{y_0}{H}$ and
$\norm{\dot{c}}{L^\infty(\bbR_0,\bbR^M)}+\norm{\ddot{c}}{L^\infty(\bbR_0,\bbR^M)}\le C_c$. Above
the constants $C$, $C_u$, $C_c$, and  $\mu>0$ are independent  of $y_0$ and $c_0$.
\end{theorem}

\medskip
Besides the theoretical result we also discuss the numerical computation and implementation of a stabilizing
control input based on a moving indicator function.
Note that the control input~$u(t)\widehat\indf_{\omega(c(t))}$
depends nonlinearly on the control functions~$(u,c)$. In order to realize the geometrical constraint $\omega(c(t))\subset \Omega$,
which can be obtained through   constraints on the velocity~$\dot c$ and acceleration~$\ddot c$, it will be convenient to introduce
a new auxiliary function
\[\eta=\ddot c+\varsigma\dot c+\epsilon c,\qquad\mbox{for given}\quad \epsilon\ge0,\quad\varsigma\ge 0.\]
We shall consequently  consider system~\eqref{sys-y-parab-Intro} in the extended
form
\begin{subequations}\label{sys-y-parab-Intro-Ext}
 \begin{align}
 \dot y-\nu\Delta y+ay+b\cdot\nabla y&=u\widehat\indf_{\omega(c)},\quad\clG y\rest{\Gamma}=0,&\quad& y(0,\Bigcdot)=y_0,\\
 \ddot c+\varsigma\dot c+\epsilon c&=\eta,&\quad& c(0)=c_0,\quad \dot c(0)= 0,
 \end{align}
\end{subequations}
with proper constraints on the newly introduced additional control $\eta$, in order to force the actuator to move in an appropriate way.
 Note that looking for~$c$ is equivalent to looking for~$\eta$,
 as soon as the initial actuator position~$c_0$ is given.
An analogous extension argument is used in~\cite{PhanRod18-mcss,Rod18,Badra09-cocv}, with a first order~{\sc ode},
$\dot c+\epsilon c=\eta$
in order to deal with boundary controls problems.

Observe that system~\eqref{sys-y-parab-Intro} is linear in the state variable~$y$ and nonlinear in the control variable~$(u,c)$.
Instead system~\eqref{sys-y-parab-Intro-Ext} is linear in the control variable~$(u,\eta)$ and nonlinear
in the state variable~$(y,c)$,
because~$c(t)\mapsto\indf_{\omega(c(t))}$ is nonlinear from~$\bbR^d$ into~$L^2(\Omega)$.

In order to compute  the pairs ~$(u,\eta)$, the stabilization
problem will be  formulated as an infinite horizon optimal control  problem (see~\eqref{cost-funct}--\eqref{IHOP}) whose solution
will be   a stabilizing  pair of  ~$(u,\eta)$.
To deal with  the resulting infinite-horizon problem  a receding horizon control framework will be employed.
In this framework, a stabilizing moving control is constructed
through the concatenation of solutions of open-loop
problems defined on overlapping temporal domains covering~$[0,\infty)$.

\subsection{Related literature}
Moving controls have been considered, for example, in~\cite{CastroZuazua05,Khapalov01} where suitable
moving Dirac delta functions are taken as actuators. In~\cite{Khapalov01}, both approximate controllability
and exact null controllability results
are proven for a semilinear 1D parabolic equation by means of  two  moving Dirac functions.
Both Dirac delta functions and indicator functions are typical actuators in applications,
see for instance ~\cite{Khapalov01} (cf.~\cite[Eqs.~(1.2) and~(1.3)]{Khapalov01}).
Such actuators  lead to {\em lumped} controls, which are essentially characterized by the
temporal behavior only. Concerning again the terminology,
in~\cite{Khapalov01} the
Dirac delta functions based controls are called {\em point} controls, and the indicator functions based controls are
called {\em average} controls or {\em zone} controls.
In~\cite{CastroZuazua05} approximate controllability results for higher dimensional linear autonomous parabolic equations,
by means of moving point controls and, more generally, with controls moving in a lower-dimensional submanifold, are presented.
For semilinear 1D parabolic equations evolving in the spatial interval~$(0,1)\in\bbR$,
approximate controllability results have been derived in~\cite{Khapalov99} by means of a single static average
control~$u(t)\indf_{\widehat\omega}$, $\widehat\omega=(l_1,l_2)\subset(0,1)$. The results
are obtained under the condition that $l_1\pm l_2$ are irrational numbers.

Concerning partial differential equations which are not of parabolic type we refer
to~\cite{MartinRosierRouchon13}, where controllability properties
for 1D damped wave equations, under periodic boundary conditions, $\Omega=\bbT$, are derived by means of a control based on a single
moving point actuator~$u(t)\Phi(t)$. The actuator is
either a Dirac delta~$\Phi(t)=\delta_{c(t)}$, see~\cite[Thm.~1.4]{MartinRosierRouchon13}, or a single
moving function~$\Phi(t)=\phi_{c(t)}\in L^2(\bbT)$, see~\cite[Thm.~1.1]{MartinRosierRouchon13} where we can also see that the
function $\phi_{c(t)}$ is required to have zero mean.
{ We refer also to~\cite{CasCinMun14,ChSilvaRosierZuazua14,KunischSouza18} where a
moving average control is considered, but where the magnitude control function~$u=u(t,x)$
depends on both time and space variables. By means of such a moving control, in~\cite{ChSilvaRosierZuazua14} the
approximate controllability of higher dimensional
damped wave equations is derived and, in~\cite{CasCinMun14} the  inner null controllability of
the one-dimensional wave equation is investigated theoretically and numerically.} In~\cite{KunischSouza18} the null controllability
is derived for a 1D coupled {\sc pde-ode} system of FitzHugh--Nagumo type, again with a
magnitude control function~$u=u(t,x)$ depending on both time and space variables.
We recall that such  systems are not null controllable
by means of static average controls.

It is well known that observability properties and null controllability properties are related.  In this respect we
refer  to the observability results in~\cite{Khapalov94} for the autonomous
higher dimensional case with point observations.
We recall that often the tools used to derive controllability/observability results for autonomous
systems are not appropriate or are not valid to
deal with the nonautonomous case. See for example the solution representation in~\cite[Eq.~(2.1)]{Khapalov94},
and the discussion in ~\cite[Sect.~6, \S1]{CastroZuazua05}.

Our result in Theorem~\ref{T:main-Intro} is of different nature, when compared to the ones mentioned above.
Approximate  and null controllability are properties concerning the state~$y(T)$ at a given time~$T$.
Instead, our goal in ~\eqref{goal.parab-Intro} is concerned with the asymptotic behavior of the state (as time goes to~$+\infty$).
Of course, if we have  a control driving the state to~$y(T)=0$ at time~$t=T$,
then by switching the control off,
for $t>T$, results in  a stabilizing
control.  Thus exact controllability is a stronger property than stabilizability.

On the other hand for practical considerations controls driving the system to~$0$ at time~$T$
may not be enough for applications, since, due to noise or computational error, the control may
not drive the state exactly to the origin. If the latter is unstable and the control is nonetheless
switched off then the state may diverge as time tends to infinity.
Therefore,  a control is still needed which  stabilizes the
state once it is close to the origin, or which   keeps it in a small neighborhood of~$0$ which is proportional
to the magnitude of noise and disturbances.

Moving indicator functions have also been considered in~\cite{Demetriou10}, where the goal is not the stabilizability of a given unstable free dynamics (as in this manuscript), but rather  to speed up the stabilization and/or counteract the effect of external disturbances (sources). Though the nominal systems under study in~\cite{Demetriou10} are stable parabolic equations, the proposed control design is interesting for applications.

\subsection{On  (lack of) stabilizability with a single static actuator}
In this section we provide examples  where a single static actuator
is not sufficient to stabilize the system, no matter what its shape or placement in the spatial domain is.
This negative result can be
seen as a motivation for our work in this manuscript,
where we show that we can still stabilize the
system if we are allowed to dynamically move a given indicator function as actuator.

Here  we consider only the particular case of controlled autonomous diffusion-reaction
systems of the form
\begin{subequations}\label{sys-y-axy}
 \begin{align}
 &\tfrac{\p}{\p t} y(t,x)-\nu\Delta y(t,x)+a(x)y(t,x)=u(t)\Psi(x),\qquad t>0,\\
 & y(0,\Bigcdot)=y_0,\qquad \clG y\rest\Gamma=0,
 \end{align}
 \end{subequations}
evolving in a regular enough bounded domain~$\Omega\subset\bbR^d$, $d\in\bbN_0$, and with
\begin{align}
 &\nu>0,\qquad a_0\in\bbR,\qquad\Psi\in L^2(\Omega),\\
 &u\in L^2_{\rm loc}((0,+\infty),\bbR),\qquad  y_0\in L^2(\Omega),\qquad\mbox{and}\qquad
a\in L^\infty(\Omega).
 \end{align}
In~\eqref{sys-y-axy} above $\clG$ stands for either the Dirichlet or the  Neumann trace operator.

 Let~$\{\widetilde e_i\mid i=1,2,\dots\}$ be a countable complete linearly independent
 system of eigenfunctions of the operator
 \[
 \clA\coloneqq-\nu\Delta +a(x)\Id\colon\rmD(\clA)\to L^2(\Omega),
 \]
with domain~$\rmD(\clA)
 =\{z\in H^2(\Omega)\mid \clG z\rest\Gamma=0\}$. Let~$\widetilde \alpha_i$ be the corresponding eigenvalues
\[
 \clA\widetilde e_i=\widetilde \alpha_i\widetilde e_i,
 \qquad \widetilde \alpha_1\le\widetilde \alpha_2\le\widetilde \alpha_3\le....,
 \qquad \lim_{i\to+\infty}\widetilde \alpha_i=+\infty.
\]
The following result implies that system~\eqref{sys-y-axy} is not exponentially stabilizable, for any given
static actuator~$\Psi\in L^2(\Omega)$.
\begin{proposition}\label{P:NOTstabil1Act}
 If there exists a nonsimple nonpositive eigenvalue~$\widetilde \alpha_j$,
then for each~$\Psi\in L^2(\Omega)$ we can find~$y_0\in L^2(\Omega)$ such that
$\clA y_0=\widetilde \alpha_j y_0$ and~$(y_0,\Psi)_{L^2(\Omega)}=0$. In particular,
 for
all~$u\in L^2_{\rm loc}((0,+\infty),\bbR)$ the weak solution~$y$ of~\eqref{sys-y-axy}
satisfies
$\norm{y(t,\Bigcdot)}{L^2(\Omega)}\ge\rme^{\widetilde \alpha_j t}\norm{y_0}{L^2(\Omega)}$, with~$\widetilde \alpha_j\ge0$
\end{proposition}

Next, for the sake of completeness, we present/recall also a positive result for stabilization with an appropriate
single actuator~$\Psi$.
\begin{proposition}\label{P:stabil1Act}
If all the nonpositive eigenvalues of~$\clA$ are simple, and if
none of the corresponding eigenfunctions is orthogonal to~$\Psi$, then
system~\eqref{sys-y-axy} is exponentially stabilizable.
\end{proposition}

The proofs of Propositions~\ref{P:NOTstabil1Act} and~\ref{P:stabil1Act} are given in the Appendix,
Section~\ref{Apx:proofP:NOTYESstabil1Act}.

\subsection{Contents and notation}
In Section~\ref{S:assumptions} we present the assumptions we require  for the operators~$A$ and~$A_{\rm rc}$
in~\eqref{sys-y-Intro}.
Our main
exponential stabilization result is proved in Section~\ref{S:stability}. In Section~\ref{S:proofT:main-Intro}
this is applied to the  concrete parabolic
equations as~\eqref{sys-y-parab-Intro} and Theorem~\ref{T:main-Intro} is proved.
Section~\ref{S:num_impl} is devoted to the numerical computation of a moving control
based on the receding horizon framework  which shows the exponentially stabilizing performance.

\smallskip

Concerning  notation, we write~$\bbR$ and~$\bbN$ for the sets of real numbers and nonnegative
integers, respectively, and  we set $\bbR_r\coloneqq(r,+\infty)$ with $r\in\bbR$, whose closure is
denoted by~$\overline\bbR_r\coloneqq[r,+\infty)$. Finally, we set~$\bbN_0\coloneqq\mathbb N\setminus\{0\}$.

Given two Banach spaces~$X$ and~$Y$, if the inclusion
$X\subseteq Y$ is continuous, we write $X\xhookrightarrow{} Y$. We write
$X\xhookrightarrow{\rm d} Y$, respectively $X\xhookrightarrow{\rm c} Y$, if the inclusion is also dense, respectively compact.

Let $X\subseteq Z$ and~$Y\subseteq Z$ be continuous inclusions, where~$Z$ is a Hausdorff topological space.
Then we  define the Banach spaces $X\times Y$, $X\cap Y$, and $X+Y$,
endowed with the  norms
$|(h,g)|_{X\times Y}:=\bigl(|h|_{X}^2+|g|_{Y}^2\bigr)^{\frac{1}{2}}$,
$|\hat h|_{X\cap Y}:=|(\hat h,\hat h)|_{X\times Y}$, and
$|\tilde h|_{X+Y}:=\inf\limits_{(h,g)\in X\times Y}\bigl\{|(h,g)|_{X\times Y}\mid \tilde h=h+g\bigr\}$,
respectively.
In case we know that $X\cap Y=\{0\}$, we say that $X+Y$ is a direct sum and we write $X\oplus Y$ instead.

For a given interval~$I\subset\bbR$, we denote~$W(I,X,Y)\coloneqq\{f\in L^2(I,X)\mid \dot f\in L^2(I,Y)\}$,
endowed withe the norm~$\norm{f}{W(I,X,Y)}\coloneqq \norm{(f,\dot f)}{L^2(I,X)\times L^2(I,Y)}$.

The space of continuous linear mappings from~$X$ into~$Y$ is denoted by~$\clL(X,Y)$. In case~$X=Y$ we
write~$\clL(X)\coloneqq\clL(X,X)$.

The continuous dual of~$X$ is denoted~$X'\coloneqq\clL(X,\bbR)$. The adjoint of an operator $L\in\clL(X,Y)$ will be denoted $L^*\in\clL(Y',X')$.

The space of continuous functions from~$X$ into~$Y$ is denoted by~$\clC(X,Y)$.

The orthogonal complement to a given subset~$B\subset H$ of a Hilbert space~$H$, with scalar product~$(\Bigcdot,\Bigcdot)_H$,  is
denoted by~$B^\perp\coloneqq\{h\in H\mid (h,s)_H=0\mbox{ for all }s\in B\}$.

Given two closed subspaces~$F\subseteq H$ and~$G\subseteq H$ of the Hilbert space given by~$H=F\oplus G$,
we denote by~$P_F^G\in\clL(H,F)$
the oblique projection in~$H$ onto~$F$ along~$G$. That is, writing $h\in H$ as $h=h_F+h_G$ with~$(h_F,h_G)\in F\times G$,
we have~$P_F^Gh\coloneqq h_F$.
The orthogonal projection in~$H$ onto~$F$ is denoted by~$P_F\in\clL(H,F)$. Notice that~$P_F= P_F^{F^\perp}$.

By
$\overline C_{\left[a_1,\dots,a_n\right]}$ we denote a nonnegative function that
increases in each of its nonnegative arguments~$a_i\ge0$, $1\le i\le n$.

Finally, $C,\,C_i$, $i=0,\,1,\,\dots$, stand for unessential positive constants.

\section{Assumptions}\label{S:assumptions}
The results will follow under general assumptions on the plant dynamics
operators~$A$ and~$A_{\rm rc}$, and on a particular stabilizability
assumption of~\eqref{sys-y-Intro} by means of controls
based on a large enough finite number~$M$ of suitable static actuators.

The Hilbert space~$H$, in which system~\eqref{sys-y-Intro}
is evolving in,
will be set as a pivot space, that is, we identify,~$H'=H$.
Let~$V$ be another Hilbert space
with~$V\subset H$.
\begin{assumption}\label{A:A0sp}
 $A\in\clL(V,V')$ is symmetric and $(y,z)\mapsto\langle Ay,z\rangle_{V',V}$ is a complete scalar product in~$V.$
\end{assumption}

From now on, we suppose that~$V$ is endowed with the scalar product~$(y,z)_V\coloneqq\langle Ay,z\rangle_{V',V}$,
which still makes~$V$ a Hilbert space.
Necessarily, $A\colon V\to V'$ is an isometry.
\begin{assumption}\label{A:A0cdc}
The inclusion $V\subseteq H$ is dense, continuous, and compact.
\end{assumption}

Necessarily, we have that
\[
 \langle y,z\rangle_{V',V}=(y,z)_{H},\quad\mbox{for all }(y,z)\in H\times V,
\]
and also that the operator $A$ is densely defined in~$H$, with domain $\rmD(A)$ satisfying
\[
\rmD(A)\xhookrightarrow{\rm d,\,c} V\xhookrightarrow{\rm d,\,c} H\xhookrightarrow{\rm d,\,c} V'\xhookrightarrow{\rm d,\,c}\rmD(A)'.
\]
Further,~$A$ has a compact inverse~$A^{-1}\colon H\to \rmD(A)$, and we can find a nondecreasing
system of (repeated accordingly to their multiplicity) eigenvalues $(\alpha_n)_{n\in\bbN_0}$ and a corresponding complete basis of
eigenfunctions $(e_n)_{n\in\bbN_0}$:
\begin{equation}\label{eigfeigv}
0<\alpha_1\le\alpha_2\le\dots\le\alpha_n\le\alpha_{n+1}\to+\infty \quad\mbox{and}\quad Ae_n=\alpha_n e_n.
\end{equation}

We can define, for every $\zeta\in\bbR$, the fractional powers~$A^\zeta$, of $A$, by
\[
 y=\sum_{n=1}^{+\infty}y_ne_n,\quad A^\zeta y=A^\zeta \sum_{n=1}^{+\infty}y_ne_n\coloneqq\sum_{n=1}^{+\infty}\alpha_n^\zeta y_n e_n,
\]
and the corresponding domains~$\rmD(A^{|\zeta|})\coloneqq\{y\in H\mid A^{|\zeta|} y\in H\}$, and
$\rmD(A^{-|\zeta|})\coloneqq \rmD(A^{|\zeta|})'$.
We have that~$\rmD(A^{\zeta})\xhookrightarrow{\rm d,\,c}\rmD(A^{\zeta_1})$, for all $\zeta>\zeta_1$,
and we can see that~$\rmD(A^{0})=H$, $\rmD(A^{1})=\rmD(A)$, $\rmD(A^{\frac{1}{2}})=V$.

For the time-dependent operator 
 we assume the following:
\begin{assumption}\label{A:A1}
For almost every~$t>0$ we have~$A_{\rm rc}(t)\in\clL(V, H)$,
and we have a uniform bound, that is, $\norm{A_{\rm rc}}{L^\infty(\bbR_0,\clL(V, H))}\eqqcolon C_{\rm rc}<+\infty.$
\end{assumption}

Finally, we will need the following norm squeezing property, by means of controls based on static actuators.
\begin{assumption}\label{A:MstaticAct}
There exist:
\begin{itemize}
 \item a positive integer~$M$, and positive real numbers~$T>0$ and~$\theta\in(0,1)$,
 \item a linearly independent family~$\{\widehat\Phi_j\mid j\in\{1,2,\dots,M\}\}\subset H$ with~$\norm{\widehat\Phi_j}{H}=1$,
\item a family of
 functions~$\{v_k\in\clL(V,L^\infty((kT,kT+T),\bbR^M))\mid k\in\bbN\}$, with
 $\sup\limits_{k\in\bbN}\norm{v_k}{\clL(V,L^\infty((kT,kT+T),\bbR^M))}\le\fkK$,
 \end{itemize}
 such that:  for all $k\in\bbN$, the solution of
 \begin{align}\label{sys-y-static}
 \dot y+Ay+A_{\rm rc}(t)y=\textstyle\sum\limits_{j=1}^Mv_{k,j}(\fkv)(t)\widehat\Phi_j,\qquad y(kT)=\fkv,\qquad t\in(kT,kT+T),
\end{align}
satisfies
\begin{align}\label{goal-static}
 &\norm{y(kT+T)}{V}\le \theta\norm{\fkv}{V},\qquad\mbox{for all}\quad\fkv\in V.
\end{align}
\end{assumption}

\medskip
\begin{remark}
Assumptions~\ref{A:A0sp}--\ref{A:MstaticAct} are satisfiable
for parabolic equations as~\eqref{sys-y-parab-Intro}
evolving in bounded rectangular domains~$\Omega\subset\bbR^d$. The satisfiability of such assumptions shall be
revisited/proven later on, in Section~\ref{S:proofT:main-Intro}, where we give the proof of Theorem~\ref{T:main-Intro}, concerning
standard parabolic equations.
\end{remark}
\begin{remark}
 Alternatively, in Assumption~\ref{A:A1} we can take a reaction-convection
 term~$A_{\rm rc}(t)\in L^\infty(\bbR_0,\clL(H, V'))$. The proof will however involve
 slightly different steps. Motivations and further details are given later in Section~\ref{sS:RemarkArc}.
\end{remark}

\section{Existence of a moving stabilizing control}\label{S:stability}
Hereafter~$\fkS_H$ denote the unit sphere in~$H$,
\[
 \fkS_H\coloneqq\{h\in H\mid \norm{h}{H}= 1\}.
\]
We prove our main result, which is the following.
\begin{theorem}\label{T:main0}
Under Assumptions~\ref{A:A0sp}--\ref{A:MstaticAct}, there exist a magnitude control function~$u$ and a continuous
moving actuator~$\Phi$ satisfying
\begin{align}
 &u\in L^2(\bbR_0,\bbR),\quad\dot\Phi\in L^\infty(\bbR_0,H),\quad \ddot\Phi\in L^\infty(\bbR_0,H),\notag\\
 &\Phi(0)=\widehat\Phi_1,\quad\dot\Phi(0)=0,\quad\Phi(t)\in\fkS_H\quad\mbox{for}\quad t\ge0,\notag
\end{align}
and constants~$C\ge1$ and $\mu>0$, such
 that the solution of the system~\eqref{sys-y-Intro},
  \begin{align}\label{sys-y0}
  \dot y+Ay+A_{\rm rc}(t)y=u(t)\Phi(t),\qquad y(0)=y_0\in V,\qquad t>0,
 \end{align}
 satisfies~\eqref{goal-Intro},
\begin{subequations}\label{goal0}
\begin{align}
 &\norm{y(t)}{V}\le C\ex^{-\mu t}\norm{y_0}{V},\quad\mbox{for all}\quad t\ge0,\label{goal0-exp}
 \intertext{and the mapping~$y_0\mapsto u(y_0)$ is continuous,}
 &\norm{u}{\clC( V,L^2(\bbR_0,\bbR))}\eqqcolon\fkN_0<+\infty.\label{goal0-L2}
\end{align}
\end{subequations}
 Furthermore, $\norm{\dot\Phi}{L^\infty(\bbR_0,H)}+\norm{\ddot\Phi}{L^\infty(\bbR_0,H)}\le C_\Phi$ with
$C_\Phi$ independent of~$y_0$.
 \end{theorem}

Note that Theorem~\ref{T:main0} gives us stabilizability in the $V$-norm. The stabilizability in $H$-norm as stated in~\eqref{goal-Intro}
follows as a consequence.
\begin{corollary}\label{C:main1}
Let $\Phi^*\in \clC^1([0,1],H)$ satisfy
\[
 \ddot\Phi^*\in L^\infty((0,1),H),\quad \Phi^*(1)=\widehat\Phi_1,\quad\dot\Phi^*(1)=0,\quad\Phi^*(t)\in\fkS_H,\quad\mbox{for}\quad t\in[0,1].
\]
Under Assumptions~\ref{A:A0sp}--\ref{A:MstaticAct}, there exist a magnitude control function~$u^{\rm e}$ and a continuous
moving actuator~$\Phi^{\rm e}$ satisfying
\begin{align}
 &u^{\rm e}\in L^2(\bbR_0,\bbR),\quad
 \dot\Phi^{\rm e}\in L^\infty(\bbR_0,H),
 \quad \ddot\Phi^{\rm e}\in L^\infty(\bbR_0,H),\notag\\
 &\Phi^e\rest{[0,1]}=\Phi^*,\quad \Phi^{\rm e}(t)\in\fkS_H\quad\mbox{for}\quad t\ge0,\notag
\end{align}
and constants~$C^{\rm e}\ge1$ and $\mu>0$, such
 that the solution of the system~\eqref{sys-y-Intro},
  \begin{align}\label{sys-y1}
  \dot y+Ay+A_{\rm rc}(t)y=u^{\rm e}(t)\Phi^{\rm e}(t),\qquad y(0)=y_0\in H,\qquad t>0,
 \end{align}
 satisfies~\eqref{goal-Intro},
\begin{subequations}\label{goal1}
\begin{align}
 &\norm{y(t)}{H}\le C^{\rm e}\ex^{-\mu t}\norm{y_0}{H},\quad\mbox{for all}\quad t\ge0,\label{goal1-exp}
 \intertext{and the mapping~$y_0\mapsto u(y_0)$ is continuous,}
 &\norm{u}{\clC( H,L^2(\bbR_0,\bbR))}\eqqcolon\fkN_0^{\rm e}<+\infty.\label{goal1-L2}
\end{align}
\end{subequations}
Moreover, $\norm{\dot\Phi^e}{L^\infty(\bbR_0,H)}+\norm{\ddot\Phi^e}{L^\infty(\bbR_0,H)}
\le \max\left\{C_\Phi,\norm{\dot\Phi^*}{L^\infty((0,1),H)}+\norm{\ddot\Phi^*}{L^\infty((0,1),H)}\right\}$ with
$C_\Phi$ independent of~$(y_0,\Phi^*(0))$.
 \end{corollary}
\begin{proof}
For $t \in [0,1]$ we choose the control $u(t)\Phi^*(t)$ with $u =0$.
Using the smoothing property of parabolic-like equations
(cf.~\cite[Lem.~2.4]{BreKunRod17}),
we arrive at a state~$y(1)\eqqcolon y_1\in V$, with
 \begin{equation}\label{smooth}
  \norm{y(1)}{V}\le \ovlineC{C_{\rm rc}}\norm{y(0)}{H}.
 \end{equation}

In the time interval~$\bbR_1$,
we can find a control $u^1\in L^2(\bbR_1,\bbR)$ and a moving actuator~$\Phi^1$ as in Theorem~\ref{T:main0}, with~$\Phi^1(1)=\widehat\Phi_1$
and~$\dot\Phi^1(1)=0$, giving us
\begin{equation}\label{stabR1}
   \norm{y(t)}{V}\le \ovlineC{C_{\rm rc}}\ex^{-\mu (t-1)}\norm{y(1)}{V},\qquad t\ge1.
 \end{equation}
Indeed it is enough to consider a shift in time variable and use Theorem~\ref{T:main0} to the function
$w(\tau)=y(1+\tau)$, which solves the system
\begin{align}
 \tfrac{\ed}{\ed\tau} w+A w+\widetilde A_{\rm rc}w= u(\tau)\Phi(\tau),
 \quad w(0)=y_1,\quad \tau>0,\notag
\end{align}
with~$\widetilde A_{\rm rc}(\tau)= A_{\rm rc}(1+\tau)$. Hence obtaining
\begin{align}
 \norm{w(\tau)}{V}\le \ovlineC{C_{\rm rc,1}}\ex^{-\mu \tau}\norm{w(0)}{V}, \qquad \tau\ge0,\notag
\end{align}
with~$C_{\rm rc,1}=\norm{\widetilde A_{\rm rc}}{L^\infty(\bbR_0,\clL(V,H))}
= \norm{ A_{\rm rc}}{L^\infty(\bbR_1,\clL(V,H))}\le C_{\rm rc}$ which implies~\eqref{stabR1}, by taking for~$t\ge1$,
$u^1(t)=u(t-1)$ and~$\Phi^1(t)=\Phi(t-1)$.

Next, defining
\begin{align}
u^e(t)&=0&&\quad\mbox{and}\quad&&\Phi^e(t)=\Phi^*(t),&&\quad\mbox{for}\quad t\in[0,1),\notag\\
u^e(t)&=u^1(t) &&\quad\mbox{and}\quad &&\Phi^e(t)=\Phi^1(t),&&\quad\mbox{for}\quad t\ge1,\notag
\end{align}
we obtain, using~\eqref{stabR1} and~\eqref{smooth},
\begin{align}
  \norm{y(t)}{H}&\le\norm{\Id}{\clL(V,H)}
  \norm{y(t)}{V}\le \ovlineC{C_{\rm rc},\norm{\Id}{\clL(V,H)}}\ex^{-\mu (t-1)}\norm{y(1)}{V}\notag\\
  &\le\ovlineC{C_{\rm rc},\norm{\Id}{\clL(V,H)},\mu}\ex^{-\mu t}\norm{y(0)}{H},\qquad t\ge1.\notag
 \end{align}
 and (cf.~\cite[Lem.~2.2]{BreKunRod17})
\begin{align}
  \norm{y(t)}{H}&\le\ovlineC{C_{\rm rc}}\norm{y(0)}{H}\le\ovlineC{C_{\rm rc}}\ex^{\mu}\ex^{-\mu t}\norm{y(0)}{V}\notag\\
  &\le\ovlineC{C_{\rm rc},\mu}\ex^{-\mu t}\norm{y(0)}{H},\qquad t\in[0,1).\notag
 \end{align}
We can see that we can take~$C^{\rm e}$ of the form~$\ovlineC{C_{\rm rc},\norm{\Id}{\clL(V,H)},\mu}$ in~\eqref{goal1-exp}.

Using $\Phi^*(1)=\Phi^1(1)=\widehat\Phi_1$ and~$\dot\Phi^*(1)=\dot\Phi^1(1)=0$, we can conclude
that~$\Phi^e\in\clC^1([0,+\infty),H)$. Finally,
 by Theorem~\ref{T:main0} we have that~$\norm{\dot\Phi^1}{L^\infty(\bbR_1,H)}+\norm{\ddot\Phi^1}{L^\infty(\bbR_1,H)}\le C_\Phi$ with
$C_\Phi$ independent of~$y_1$ (and of~$\Phi^*$).
\end{proof}

We are going to use Assumption~\ref{A:MstaticAct} together with a concatenation argument, and will prove that
Theorem~\ref{T:main0} is a corollary of the following result concerning the restriction of our system to the intervals
\begin{equation}\label{Ik}
 I_k\coloneqq(kT,kT+T),\quad \overline{I_k}\coloneqq[kT,kT+T]\qquad k\in\bbN.
\end{equation}

\begin{theorem}\label{T:main}
 Under Assumptions~\ref{A:A0sp}--\ref{A:MstaticAct}, there exist a magnitude control function~$u_k$ and a continuous
moving actuator~$\Phi_k$ satisfying
\begin{align}
 &u_k\in L^2(I_k,\bbR),\quad\dot\Phi_k\in L^\infty(I_k,H),\quad \ddot\Phi_k\in L^\infty(I_k,H),\quad\Phi_k(t)\in\fkS_H\quad\mbox{for}\quad t\in \overline{I_k},\notag\\
&\Phi_k(kT)=\Phi_k(kT+T)=\widehat\Phi_1,\quad \dot\Phi_k(kT)=\dot\Phi_k(kT+T)=0,\notag
\end{align}
such that the solution of the system
 \begin{align}\label{sys-y-Ik}
  \dot y+Ay+A_{\rm rc}(t)y&=u_k(t)\Phi_k(t),&\qquad y(kT)&=\fkv\in V,
  \end{align}
satisfies
\begin{subequations}\label{goal2}
\begin{align}
 &\norm{y(kT+T)}{V}\le \tfrac{\theta+1}{2}\norm{\fkv}{V},\label{goal2-exp}
 \intertext{and, the mapping~$\fkv\mapsto u_k(\fkv)$ is continuous,}
 &\norm{u_k}{\clC(V,L^2(I_k,\bbR))}\eqqcolon\fkN_1<+\infty.\label{goal2-L2}
\end{align}
with~$\fkN_1$ independent of~$k\in\bbN$.
\end{subequations}
 \end{theorem}

\begin{proof}[Proof of Theorem~\ref{T:main0}]
We consider the concatenation of controls~$u_k$ given by Theorem~\ref{T:main} as follows
 \[
 u\Phi(y_0)=u_k\Phi_k(y(kT)),\quad\mbox{if}\quad t\in I_k,
   \]
where the construction of~$u$ is to be understood in a sequential manner: first we take~$u\Phi(y_0)\rest{I_0}=u_0\Phi_0(y(0T))=u_0\Phi_0(y_0)$,
then we consider the corresponding state~$y(T)$ at final time~$t=T$, which we then use to define~$u\Phi(y_0)\rest{I_1}=u_1\Phi_1(y(1T))$,
in this way, by concatenation, we have constructed a control on the interval~$I_0\bigcup I_1=(0,2T)$. Once we have constructed
the control~$u\Phi(y_0)\rest{(0,kT)}$ on~$(0,kT)$, we take~$u\Phi(y_0)\rest{I_k}=u_k\Phi_k(y(kT))$
and have a control defined for time~$t\in (0,(k+1)T)$. Eventually we will have~$u\Phi(y_0)$ defined in the entire time
interval~$\bbR_0$.

By~\eqref{goal2-exp},  we find that the solution associated to~$u\Phi(y_0)$ satisfies
\begin{equation}
\label{factor.dec.v}
\norm{y(kT+T)}{V}\le \tfrac{\theta+1}{2}\norm{y(kT)}{V}\le (\tfrac{\theta+1}{2})^{k+1}\norm{y(0)}{V},
\end{equation}
that is, since~$0<\theta<1$,
\begin{equation}\label{exp.dec.nat-t}
\norm{y(kT+T)}{H}\le \ex^{-\overline \mu(k+1)T}\norm{y(0)}{H},\quad\mbox{with}\quad
\overline \mu\coloneqq-\tfrac1{T}\log(\tfrac{\theta+1}2)>0.
\end{equation}
By a standard continuity argument (e.g., see~\cite[Lem.~2.3]{BreKunRod17},
recalling that~$\clC(\overline{I_k},V)\xhookrightarrow{}W(I_k,\rmD(A),H)$), we find that for~$t\ge0$,
\begin{align}
 \norm{y(t)}{V}^2&\le\ovlineC{T,C_{\rm rc}}\left(\norm{y(k T)}{V}^2
 +\norm{u_{k}\Phi}{L^2(I_{k},H)}^2\right),\quad\mbox{with}\quad
 k=\lfloor\tfrac{t}{T}\rfloor,
 \end{align}
 where~$\lfloor r\rfloor$ denotes the integer satisfying~$k\le r<k+1$, $r\in\bbR$.
 Since~$\norm{\Phi(t)}{H}=1$, it follows that, by  using~\eqref{factor.dec.v},
\begin{align}
 \norm{y(t)}{V}^2&\le\ovlineC{T,C_{\rm rc}, M}\left(\norm{y(k T)}{V}^2
 +\norm{u_{k}}{L^2(I_{k},\bbR^M)}^2\right)
 \le\ovlineC{T,C_{\rm rc}, M,\fkN_1}\norm{y(k T)}{V}^2\notag\\
 &\le\ovlineC{T,C_{\rm rc}, M,\fkN_1}\ex^{-\overline\mu k T}\norm{y(0)}{V}^2
 =\ovlineC{T,C_{\rm rc}, M,\fkN_1}\ex^{\overline\mu (t-k T)}
 \ex^{-\overline\mu t}\norm{y(0)}{V}^2\notag\\
 & \le\ovlineC{T,C_{\rm rc}, M,\fkN_1,\overline\mu}
 \ex^{-\overline\mu t}\norm{y(0)}{V}^2,\notag
\end{align}
because~$\ex^{\overline\mu (t-k T)}\le\ex^{\overline\mu  T}=\ovlineC{T,\overline\mu}$. Therefore,~\eqref{goal0-exp} holds true.

It remains to show~\eqref{goal0-L2}. Due to~\eqref{goal2-L2}  and~\eqref{factor.dec.v} we find
\begin{align}
 &\norm{u(y_0)}{L^2(\bbR_0,\bbR)}
 =\textstyle\sum\limits_{k=0}^\infty\norm{u_k(y(kT))}{L^2(I_k,\bbR)}\notag\\
 &\hspace*{2em}\le\fkN_1\textstyle\sum\limits_{k=0}^\infty\norm{y(kT)}{V}
 \le\fkN_1\norm{y(0)}{V} \textstyle\sum\limits_{k=0}^\infty\ex^{-\frac{\overline\mu}{2} kT}
 =\fkN_1\tfrac{1}{1-\ex^{-\frac{\overline\mu}{2}T }}\norm{y_0}{V},\notag
\end{align}
which gives us~\eqref{goal0-L2}, with~$\fkN_0=\fkN_1\tfrac{1}{1-\ex^{-\frac{\overline\mu}{2}T }}$.
\end{proof}

\begin{proof}[Proof of Theorem~\ref{T:main}] Let us fix an arbitrary~$k\in\bbN$.
By Assumption~\ref{A:MstaticAct}, we have that~$\clV_k(\fkv)(t)=\textstyle\sum\limits_{j=1}^Mv_{k,j}(\fkv)(t)\widehat\Phi_j$,
is a control function driving  system~\eqref{sys-y-static} from~$\fkv\in V$ at time~$t=kT$ to a state~$y(kT+T)$
at time~$t=kT+T$, with
a norm squeezed by a factor~$\theta\in(0,1)$.
The proof will follow by successive approximations of such control,
hence we start by denoting~$\clV_k^0=\clV_k$, where the superscript underlines that~$\clV_k^0$ is our starting control. Since~$k$
has been fixed, for simplicity we will omit the subscript~$k$ in the control, $\clV^0=\clV_k^0=\clV_k$.
\begin{equation}\label{v0}
 \clV^0(\fkv)(t)=\textstyle\sum\limits_{j=1}^Mv_j^0(\fkv)(t)\widehat\Phi_j,\qquad v_j^0(\fkv)(t)\coloneqq v_{k,j}(\fkv)(t),
 \qquad (\fkv,t)\in V\times I_k.
\end{equation}

Let us consider our dynamical system~\eqref{sys-y-static} with a general external forcing~$f$ as follows,
\begin{equation}\label{sys-y-fIk}
 \dot y+Ay+A_{\rm rc}y=f,\qquad y(kT)=\fkv\in V,\qquad t\in I_k,
\end{equation}
Denoting by~$y=\fkY_k(\fkv,f)$, $y(t)=\fkY_k(\fkv,f)(t)$ the solution of~\eqref{sys-y-fIk}. We can write
\begin{subequations}\label{prop.v0}
\begin{equation}
 \norm{\fkY_k(\fkv,\clV^0(\fkv))(kT+T)}{V}\le \theta\norm{\fkv}{V}
\end{equation}
where~$0\le\theta<1$. We see that~$\clV^0(\fkv)$ is a control based on the static actuators~$\widehat\Phi_j$, and recall that
\begin{align}
&\clV^0(\fkv)(t)\in\clS_{\widehat\Phi}\coloneqq\linspan\{\widehat\Phi_j\mid 1\le j\le M\}\subset H,\\
&v^0=(v^0_1,v^0_2,\dots,v^0_M)\in\clL(V,L^\infty(I_k,\bbR^M)),
\end{align}
\end{subequations}

The proof is completed into~$\ref{st-movH}$ main steps in Sections~\ref{sS:mainproof.staV}--\ref{sS:mainproof.movH},
where we construct suitable approximations of~$\clV^0$: $\clV^0\approx \clV^1\approx \clV^2\approx \clV^3\approx \clV^4\approx \clV^5$,
arriving at a moving control~$\clV^5=\clV^5(\fkv)$, taking values~$\clV^5(\fkv)(t)$ in~$H$, with
\[
 \norm{\fkY_k(\fkv,\clV^5(\fkv))(kT+T)}{V}\le \tfrac{1+\theta}{2}\norm{\fkv}{V}.
\]
That is, ~$\clV^5(\fkv)$ drives the system
from~$\fkv\in H$ at initial time~$t=kT$ to a state~$y(kT+T)$ at final time~$t=kT+T$
 with a norm squeezed by a factor~$\frac{1+\theta}{2}\in(\theta,1)$.
\end{proof}

 In each of remaining steps of the proof of Theorem~\ref{T:main} we will use a continuity argument for system~\eqref{sys-y-fIk}.
The main contents, in each step, are as follows.
\begin{enumerate}[noitemsep,topsep=5pt,parsep=5pt,partopsep=0pt,leftmargin=0em]%
\renewcommand{\theenumi}{{\sf\arabic{enumi}}} %for appearing in future ref calling
 \renewcommand{\labelenumi}{} %for appearing inside enumerate env, trick to avoid indentation
 \item \textcircled{\bf s}~Step~\theenumi:\label{st-staDA}
 {\em Taking auxiliary static actuators in~$\rmD(A)$.}
In Section~\ref{sS:mainproof.staV}, we replace (i.e., approximate) our static
actuators~$\widehat\Phi_j\in H$ by suitable static actuators~$\widetilde\Phi_j\in \rmD(A)\subset H$.
In this way we obtain a control~$\clV^1(\fkv)(t)=\textstyle\sum\limits_{j=1}^Mv_j^0(\fkv)(t)\widetilde\Phi_j$,
taking values in~$\clS_{\widetilde\Phi}\coloneqq\linspan\{\widetilde\Phi_j\mid 1\le j\le M\}\subset \rmD(A)$.
Taking actuators in~$\rmD(A)$ is needed for
technical reasons, which play a role in Step~\ref{st-pccDA} of the proof.
\item \textcircled{\bf s}~Step~\theenumi:\label{st-pccDA}
 {\em Piecewise constant static control in~$\rmD(A)$.}
In Section~\ref{sS:mainproof.pccV}, we approximate the control~$\clV^1(\fkv)$, by a right-continuous piecewise
constant control~$\clV^2(\fkv)$ taking values~$\clV^2(\fkv)(t)$ in the
set~$\{s\widetilde\Phi_j\mid 1\le j\le M,\;-K\le s\le K\}\subset \rmD(A)$
for a suitable constant~$K>0$, for all $t\in I_k$.
\item \textcircled{\bf s}~Step~\theenumi:\label{st-pccH}
 {\em Piecewise constant static control in~$H$. Back to original actuators.}
In Section~\ref{sS:mainproof.pccH}, we replace back the~$\widetilde\Phi_j$s by the~$\widehat\Phi_j$s. In this way
we arrive at a piecewise constant control~$\clV^3(\fkv)$ defined as~$\clV^3(\fkv)(t)\coloneqq s\widehat\Phi_j$ if~$\clV^2(\fkv)(t)
=s\widetilde\Phi_j$
taking values in the set~$\{s\widehat\Phi_j\mid 1\le j\le M,\;-K\le s\le K\}\subset H$.
\item \textcircled{\bf s}~Step~\theenumi:\label{st-nondeg}
 {\em A piecewise constant static control with nondegenerate intervals of constancy.}
 In Section~\ref{sS:mainproof.nondeg} we construct a piecewise constant control~$\clV^4(\fkv)$ taking values~$\clV^4(\fkv)(t)$ in~$H$, where the lengths
 of the intervals of constancy are all larger than a suitable positive constant.
\item \textcircled{\bf s}~Step~\theenumi:\label{st-movH}
 {\em A moving control in~$H$.}
In Section~\ref{sS:mainproof.movH}, we construct a moving control~$\clV^5=\clV^5(\fkv)=u(t)\Phi(t)$ which
visits (several times) the positions of the static actuators~$\widehat\Phi_j$,
spending a suitable amount of time at those positions,
and travels, in~$H$, fast enough between
those positions. In this way we obtain a moving control~$\clV^5=\clV^5(\fkv)$, taking values~$\clV^5(\fkv)(t)$ in $H$.
\end{enumerate}

Steps~\ref{st-staDA} and~\ref{st-pccH} are needed only if some of our static actuators
are in~$H\setminus \rmD(A)$. This will, in general, be the case for indicator functions~$\indf_\omega$, $\omega\subset\Omega$,
for scalar parabolic equations evolving in bounded domains~$\Omega\subset\bbR^d$.

The continuity arguments in Steps~\ref{st-staDA},~\ref{st-pccH},~\ref{st-nondeg}, and~\ref{st-movH} are standard, namely the continuity
of the solution of system~\eqref{sys-y-fIk} on the right-hand side
as~$y=\fkY_k(\fkv,\Bigcdot)\in\clC(L^2(I_k,H),\clC(\overline{I_k},V))$.
The continuity argument in Step~\ref{st-pccDA} is less standard, involving
the continuity of the solution when the right-hand side
varies in the so called relaxation metric, details will be given in Section~\ref{sS:mainproof.pccV}.

\subsection{A static control taking values in $\rmD(A)$}\label{sS:mainproof.staV}
Observe that $\frac{1+\theta}{2}-\theta=\frac{1-\theta}{2}>0$.

Recall that (cf.~\cite[Lem.~2.3]{BreKunRod17},
recalling that~$\clC(\overline{I_k},V)\xhookrightarrow{}W(I_k,\rmD(A),H)$) the
solution of system~\eqref{sys-y-fIk}
satisfies
\begin{equation}\label{cont.YL2}
\norm{\fkY_k(\fkv,f)(t)}{V}^2\le D_Y\left(\norm{\fkv}{V}^2+\norm{f}{L^2(I_k,H)}^2\right),\quad t\in I_k=(kT,kT+T),
\end{equation}
with~$D_Y=\ovlineC{T,C_{\rm rc}}$ independent of~$k$, where~$C_{\rm rc}$ is defined in Assumption~\ref{A:A1}.

By  Assumption~\ref{A:MstaticAct},
the actuators~$\widehat \Phi_j$ are in the unit sphere~$\fkS_H$ of~$H$. Then, from~$\rmD(A)\xhookrightarrow{\rm d}H$ we can choose
a family~$\{\widetilde\Phi_j\mid 1\le j\le M\}$ such that
\begin{subequations}\label{choice-tildePhi}
 \begin{align}
 &\widetilde\Phi_j\in \rmD(A)\textstyle\bigcap \fkS_H,\\
 &\norm{\widetilde\Phi_j-\widehat\Phi_j}{H}
 \le\tfrac{1-\theta}{10}D_Y^{-\frac{1}{2}}T^{-\frac{1}{2}} \fkK^{-1}M^{-1},
 \qquad 1\le j\le M.
\end{align}
\end{subequations}

The fact that the~$\widetilde\Phi_j$ can be taken in the unit sphere is a corollary of the following result,
whose proof is given in
Section~\ref{sS:proofP:denseBall}.
\begin{proposition}\label{P:denseBall}
Let~$X\subset H$ be a vector space. Then, the density of~$X\subset H$ implies the density of~$X\bigcap \fkS_H\subset \fkS_H$.
\end{proposition}

Now we recall the control~$\clV^0$ in~\eqref{v0}, and define a new control as
\begin{equation}\label{v1}
 \clV^1(\fkv)(t)=\textstyle\sum\limits_{j=1}^Mv_j^1(\fkv)(t)\widetilde\Phi_j,\qquad v_j^1(\fkv)(t)\coloneqq v_j^0(\fkv)(t),
 \qquad (\fkv,t)\in V\times I_k.
\end{equation}
where we replace each actuator~$\widehat\Phi_j\in H$ by the auxiliary actuator~$\widetilde\Phi_j\in \rmD(A)\subset H$.

Note that, by Assumption~\ref{A:MstaticAct}, for~$v^1\coloneqq(v^1_1,v^1_2,\dots,v^1_{M})$, we find
\begin{equation}\label{supV1}
  \sup_{t\in I_k}\norm{v^1(\fkv)(t)}{\bbR^M}=\sup_{t\in I_k}\norm{v^0(\fkv)(t)}{\bbR^M}
 \le\fkK\norm{\fkv}{V}.
 \end{equation}

Let us denote~$d^1\coloneqq\fkY_k(\fkv,\clV^1(\fkv))-\fkY_k(\fkv,\clV^0(\fkv))$, which satisfies
\begin{equation}\label{sys-d1}
 \dot d^1+Ad^1+A_{\rm rc}d^1=\clV^1(\fkv)-\clV^0(\fkv),\quad\mbox{for}\quad t\in I_k,\qquad d^1(kT)=0.
\end{equation}
By~\eqref{cont.YL2} and~\eqref{choice-tildePhi}, we find that, since~$M\ge1$,
\begin{align}
 \norm{d^1(t)}{V}&\le D_Y^\frac{1}{2}\norm{\clV^1(\fkv)-\clV^0(\fkv)}{L^2(I_k,H)}
 \le \tfrac{1-\theta}{10}T^{-\frac{1}{2}}
  \fkK^{-1}M^{-1}\norm{v^0(\fkv)}{L^2(I_k,\bbR^M)}\notag\\
 &\le
 \tfrac{1-\theta}{10}T^{-\frac{1}{2}}\fkK^{-1}\norm{v^0(\fkv)}{L^2(I_k,\bbR^M)}
 \le \tfrac{1-\theta}{10}\fkK^{-1}\norm{v^0(\fkv)}{L^\infty(I_k,\bbR^M)}.\notag
\end{align}
Thus, using Assumption~\ref{A:MstaticAct},
\begin{align}\label{diff_V10}
 \norm{\fkY_k(\fkv,\clV^1(\fkv))(t)-\fkY_k(\fkv,\clV^0(\fkv))(t)}{V}&\le \tfrac{1-\theta}{10}\norm{\fkv}{V},
 \quad\mbox{for all}\quad t\in\overline{I_k}.
\end{align}

\subsection{A piecewise constant static control taking values in $\rmD(A)$}\label{sS:mainproof.pccV}
Let us denote the closed unit ball in~$\rmD(A)$ by~$\overline\fkB_{\rmD(A)}$.
Recall the control~$\clV^1(\fkv)(t)$,  defined in~\eqref{v1}, taking values
in~$\clS_{\widetilde\Phi}=\linspan\{\widetilde\Phi_j\mid 1\le j\le M\}$.
We will prove that the solution of~\eqref{sys-y-fIk} varies continuously in~$\clC(\overline{I_k},V)$
when the external forcing varies
continuously in the so called (weak) relaxation metric (cf.~\cite[Ch.~3]{Gamk78})
\begin{equation}\label{relax.metric}
 \fkD_{I_k}^{\rm wrx}(f,g)
 \coloneqq\sup_{t\in\overline{I_k}}\norm{\int_{kT}^t f(s)-g(s)\,\ed s}{\rmD(A)} ,\qquad \{f,g\}\subset
 L^\infty(I_k,\clS_{\widetilde\Phi}\textstyle\bigcap K\overline\fkB_{\rmD(A)}),
\end{equation}
for a given~$K>0$. Hence we will approximate~$\clV^1(\fkv)$ by a piecewise constant control~$\clV^2(\fkv)$ in such a metric.
We underline
here that~$f$ and~$g$ above are functions taking their values in the {\em bounded}
subset~$\clS_{\widetilde\Phi}\bigcap K\overline\fkB_{\rmD(A)}$ of the {\em finite dimensional}
subspace~$\clS_{\widetilde\Phi}\subset {\rmD(A)}$.

As the reference~\cite{Gamk78} shows, such continuity is known in control theory of
ordinary differential equations. It has also been used to derive
(approximate) controllability results for partial
differential equations, see for example~\cite[Sect.~12.3]{AgraSary05}, \cite[Sect.~6.3]{AgraSary06}, \cite[Sect.~9]{Rod06},
\cite[Sect.~3.2.2]{Rod-Thesis08}.

We  follow a variation
of the procedure in~\cite[Ch.~3]{Gamk78}, which allows us to construct a piecewise constant control
taking values in~$\{s\widetilde\Phi_j\mid 1\le j\le M,\; -K\le s\le K\}$, for a suitable fixed~$K>0$, see~\eqref{v2}.
The fact that the control takes its values in a subset of the cone~$\{r\widetilde\Phi_j\mid 1\le j\le M,\;r\in\bbR\}$ will be important
in Section~\ref{sS:mainproof.movH}. With respect to this, we would like to refer also
to~\cite[Lemma~3.5]{Shirikyan07}, for a different approximation
involving piecewise constant controls, but where the control is allowed to take
values which are not necessarily in the cone above.

In order to construct a piecewise constant control, we start with a partition of the time interval~$I_k$
into~$N$ subintervals of constant size~$\frac{T}{N}$,
\[
 I_{k,n}\coloneqq (kT+(n-1)\tfrac{T}{N}, kT+n\tfrac{T}{N}),
 \qquad  1\le n\le N.
\]
and we denote~$\overline{I_{k,n}}\coloneqq [kT+(n-1)\tfrac{T}{N}, kT+n\tfrac{T}{N}]$.
We are going to construct a piecewise constant control
on each of the subintervals~$I_{k,n}$ with exactly~$2M$
subintervals of constancy (possibly with vanishing length) where each of the~$M$ actuators~$\widetilde\Phi_j$,
$1\le j\le M$ will be active in exactly two of such intervals.

We start by defining the nonnegative constant
\begin{equation}\label{Sigma}
\Sigma_{n}(\fkv)\coloneqq\frac{N}{T}\sum\limits_{m=1}^M\norm{\int_{I_{k,n}}v^1_m(\fkv)(t)\,\ed t}{\bbR}.
\end{equation}
Observe that
 \begin{equation}\notag
\Sigma_{n}(\fkv)\le \frac{N}{T}\int_{I_{k,n}}\sum\limits_{m=1}^M\norm{v^1_m(\fkv)(t)}{\bbR}\,\ed t
\le\frac{N}{T}M\int_{I_{k,n}}\norm{v^1(\fkv)(t)}{\bbR^M}\,\ed t
 \end{equation}
 and, by~\eqref{supV1}, it follows that
 \begin{equation}\label{Sigma.bound.h}
 \Sigma_{n}(\fkv)\le M\fkK \norm{\fkv}{V}.
 \end{equation}

To simplify the exposition we denote
\begin{equation}\label{tildePhi-ext}
 \left\{\begin{array}{l}\widetilde\Phi_{2M+1-j}\coloneqq\widetilde\Phi_j,\\
v^1_{2M+1-j}(\fkv)(t)\coloneqq v^1_{j}(\fkv)(t),
\end{array}\right.\qquad 1\le j\le M.
\end{equation}

Next, we consider the cases~$\Sigma_{n}(\fkv)\ne0$ and~$\Sigma_{n}(\fkv)=0$ separately.

\medskip\noindent
{$\bullet$ \sc The case~$\Sigma_{n}(\fkv)\ne0$.}
We rewrite our control~$\clV^1(\fkv)$, as
\begin{align}
 \clV^1(\fkv)(t)
 &=\textstyle\sum\limits_{j=1}^Mv_j^1(\fkv)(t)\widetilde\Phi_j
   = \textstyle\sum\limits_{j=1}^M \frac{v_j^1(\fkv)(t)}{\Sigma_{n}(\fkv)}\Sigma_{n}(\fkv)\widetilde\Phi_j = \textstyle\sum\limits_{j=1}^{2M} \frac{v_j^1(\fkv)(t)}{2\Sigma_{n}(\fkv)}\Sigma_{n}(\fkv)\widetilde\Phi_j.\notag
 \end{align}

\begin{subequations}\label{length-IC}
Let us denote
\begin{equation}\label{lknj}
 l_{k,n,j}=l_{k,n,2M+1-j}\coloneqq\textstyle\frac{1}{2\Sigma_{n}(\fkv)}\int_{I_{k,n}} v^1_j(\fkv)(t)\,\ed t,
 \qquad 1\le j\le M.
\end{equation}
We define a piecewise constant control in each interval~$I_{k,n}\subset I_k$, where the lengths of the intervals of constancy
are given by
\begin{equation}\label{lengths-intc}
 \norm{l_{k,n,j}}{}\coloneqq\norm{l_{k,n,j}}{\bbR}.
\end{equation}
Observe that
\begin{equation}
 \textstyle\sum\limits_{j=1}^{2M}\norm{l_{k,n,j}}{}
 =2\textstyle\frac{1}{2\Sigma_{n}(\fkv)}\sum\limits_{j=1}^{M}\norm{\int_{I_{k,n}}v^1_j(\fkv)(t)\,\ed t}{\bbR}=\tfrac{T}{N}.
\end{equation}
 \end{subequations}
Note also that some of the lengths may vanish.

Next, we denote the switching time instants~$t_{k,n,j}=t_{k,n,j}^{[N]}$ as follows
\begin{subequations}\label{switching.n0}
 \begin{align}
t_{k,n,0}&\coloneqq  kT+(n-1)\tfrac{T}{N},\\
t_{k,n,j}&\coloneqq  kT+(n-1)\tfrac{T}{N}
+\textstyle\sum\limits_{m=1}^j\norm{l_{k,n,m}}{},
\qquad 1\le j\le 2M.
\end{align}
\end{subequations}
In particular, we have~$t_{k,n,2M}=kT+n\tfrac{T}{N}$.

We define 
\begin{subequations}\label{pcc-VN.Sig.n0}
\begin{align}
&\hspace*{0em}I_{k,n,j}\coloneqq [t_{k,n,j-1},t_{k,n,j}),\qquad &1\le j\le 2M.\\
&\hspace*{0em}\clV_{[N]}(\fkv)(t)=\sign(l_{k,n,j})\Sigma_{n}(\fkv)\widetilde\Phi_j,
\quad\mbox{if}\quad t\in I_{k,n,j},\quad &1\le j\le 2M.
\end{align}
 \end{subequations}

\medskip\noindent
{$\bullet$ \sc The case~$\Sigma_{n}(\fkv)=0$.}
We define
$\clV_{[N]}(\fkv)(t)= 0$, for all~$t\in I_{k,n}$. Which we can still rewrite as a piecewise constant control as follows.

Firstly we define
\begin{equation}\label{switching.0}
 t_{k,n,0}\coloneqq  kT+(n-1)\tfrac{T}{N},\quad t_{k,n,j}\coloneqq  kT+(n-1)\tfrac{T}{N}+j\tfrac{T}{2MN},
\quad 1\le j\le 2M.
\end{equation}
and, then we set analogously to~\eqref{pcc-VN.Sig.n0},
\begin{subequations}\label{pcc-VN.Sig.0}
\begin{align}
&\hspace*{0em}I_{k,n,j}\coloneqq [t_{k,n,j-1},t_{k,n,j})= [kT+(n-1)\tfrac{T}{N}+j\tfrac{T}{2MN}),\qquad &1\le j\le 2M.\\
&\hspace*{0em}\clV_{[N]}(\fkv)(t)=\sign(l_{k,n,j})\Sigma_{n}(\fkv)\widetilde\Phi_j=0\widetilde\Phi_j,
\quad\mbox{if}\quad t\in I_{k,n,j},\quad &1\le j\le 2M.
\end{align}
 \end{subequations}

\bigskip
In either case we obtain a piecewise constant control in the entire interval~$I_k$. Observe that
$\clV_{[N]}(\fkv)(t)$ tells us that we activate the actuators~$\widetilde\Phi_j$ in each interval~$I_{k,n}$ in the order
\begin{equation}\notag
 \widetilde\Phi_1\to\widetilde\Phi_2\to\dots\to\widetilde\Phi_{M-1}\to\widetilde\Phi_M
 \to\widetilde\Phi_{M+1}\to\widetilde\Phi_{M+2}\to\dots\to\widetilde\Phi_{2M-1}\to\widetilde\Phi_{2M},
\end{equation}
which is the same, by~\eqref{tildePhi-ext}, as the cycle
\begin{equation}\label{cycle}
 \widetilde\Phi_1\to\widetilde\Phi_2\to\dots\to\widetilde\Phi_{M-1}\to\widetilde\Phi_M
 \to\widetilde\Phi_M\to\widetilde\Phi_{M-1}\to\dots\to\widetilde\Phi_2\to\widetilde\Phi_1.
\end{equation}
Some actuators may be active in degenerate intervals of length zero.
The actuators are activated with the same input of constant  magnitude~$\sign(l_{k,n,j})\Sigma_{n}(\fkv)$.

Next,  we show that~$\clV_{[N]}(\fkv)(t)$ approaches~$\clV^1(\fkv)(t)$ in the relaxation metric~\eqref{relax.metric}.
We set
\begin{align}
\clI_{[N]}(t)\coloneqq\int_{kT}^t\left(\clV_{[N]}(\fkv)(s)-\clV^1(\fkv)(s)\right)\ed s.\notag
\end{align}
Then
\begin{align}
 \fkD_{I_k}^{\rm wrx}(\clV_{[N]}(\fkv),\clV^1(\fkv))
 =\sup_{t\in\overline{I_k}}\norm{\clI_{[N]}(t)}{\rmD(A)}.\notag
\end{align}

We show now that~$\clI_{[N]}$ vanishes at the extrema of the intervals~$I_{k,n}$.
Clearly
\begin{equation}\label{INvan.ext.base}
 \clI_{[N]}(kT)=0.
\end{equation}
Further, if we assume that
$\clI_{[N]}(kT+n\frac{T}{N})=0$ for a given~$0\le n\le N-1$, then:

\noindent
$\bullet$ if~$\Sigma_n(\fkv)\ne0$ we obtain
\begin{align}
\clI_{[N]}(kT+(n+1)\tfrac{T}{N})&= \int_{I_{k,n}}\left(\clV_{[N]}(\fkv)(s)-\clV^1(\fkv)(s)\right)\,\ed s\notag\\
&= {\textstyle\sum\limits_{j=1}^{2M}}\norm{l_{k,n,j}}{}\sign(l_{k,n,j})
\Sigma_{n}(\fkv)\widetilde\Phi_j-\int_{I_{k,n}}\clV^1(\fkv)(s)\,\ed s\notag\\
&= {\textstyle\sum\limits_{j=1}^{2M}}{\frac{1}{2} }\int_{I_{k,n}}v_j^1(\fkv)(s)\,\ed s
\widetilde\Phi_j-\int_{I_{k,n}}\clV^1(\fkv)(s)\,\ed s=0.\notag
\end{align}

\noindent
$\bullet$ if~$\Sigma_n(\fkv)=0$ we obtain
\begin{align}
\clI_{[N]}(kT+(n+1)\tfrac{T}{N})&= \int_{I_{k,n}}\left(0-\clV^1(\fkv)(s)\right)\,\ed s=0.\notag
\end{align}

Therefore, in either case we have that
\begin{equation}\label{INvan.ext.ind}
 \clI_{[N]}(kT+n\tfrac{T}{N})=0\quad\Longrightarrow\quad\clI_{[N]}(kT+(n+1)\tfrac{T}{N})=0,\qquad 0\le n\le N-1.
\end{equation}

From~\eqref{INvan.ext.base} and~\eqref{INvan.ext.ind}, by induction we can conclude that
\begin{equation}\label{INvan.ext}
\clI_{[N]}(kT+n\tfrac{T}{N})=0,\quad\mbox{for all}\quad n\in\{0,1,2,\dots,N\}.
\end{equation}

Now for an arbitrary~$t\in I_{k,n}$ we find
\begin{align}
 \norm{\clI_{[N]}(t)}{\rmD(A)}
 &\le\tfrac{T}{N}\left(\sup\limits_{s\in I_{k,n}}\norm{\clV_{[N]}(\fkv)(s)}{\rmD(A)}
 +\sup\limits_{s\in I_{k,n}}\norm{\clV^1(\fkv)(s)}{\rmD(A)} \right)\notag\\
 &\le\tfrac{T}{N}\left(\Sigma_{n}(\fkv)\sup\limits_{1\le j\le M}\norm{\widetilde\Phi_j}{\rmD(A)}
 +\sup\limits_{s\in I_{k}}\norm{v^1(\fkv)(s)}{\bbR^M}\sup\limits_{1\le j\le M}\norm{\widetilde\Phi_j}{\rmD(A)} \right)\notag
  \end{align}
  and by~\eqref{Sigma.bound.h} and Assumption~\ref{A:MstaticAct},
 \begin{align}
 &\norm{\clI_{[N]}(t)}{\rmD(A)}
 \le\tfrac{T}{N}\left(M\fkK\norm{\fkv}{V}
 +\fkK\norm{\fkv}{V} \right)\sup\limits_{1\le j\le M}\norm{\widetilde\Phi_j}{\rmD(A)}\notag\\
 &\hspace{2em}\le\tfrac{T}{N}(M+1)\dnorm{\widetilde \Phi}{}\fkK\norm{\fkv}{V},\qquad t\in
 \overline{I_{k}},\qquad\dnorm{\widetilde \Phi}{}\coloneqq\sup\limits_{1\le j\le M}\norm{\widetilde\Phi_j}{\rmD(A)}.
 \label{bound.relax}
\end{align}

 Next we show the continuity of the solution when the right-hand side control varies in the relaxation metric.
 Let us denote $d_N\coloneqq\fkY_k(\fkv,\clV_{[N]}(\fkv))-\fkY_k(\fkv,\clV^1(\fkv))$, and observe that~$d_N$ satisfies~\eqref{sys-y-fIk}, as
 \[
\dot d_N+Ad_N+A_{\rm rc} d_N=\clV_{[N]}(\fkv)-\clV^1(\fkv),\qquad d_N(kT)=0.
 \]
With~$z_N\coloneqq d_N-\clI_{[N]}$, we see that~$\dot z_N=\dot d_N-\dot \clI_{[N]}=-Ad_N-A_{\rm rc} d_N$, which implies
\begin{align}
&\dot z_N+Az_N+A_{\rm rc} z_N=-A\clI_{[N]}-A_{\rm rc} \clI_{[N]},\qquad z_N(kT)=0,\notag\\
\intertext{and also, by~\eqref{INvan.ext},}
&z_N(kT+n\tfrac{T}{N})=d_N(kT+n\tfrac{T}{N}),\qquad 0\le n\le N.\label{zN=dN.n}
 \end{align}

 Therefore, for $z_N=\fkY_k(0,-A\clI_{[N]}-A_{\rm rc} \clI_{[N]})$ we obtain, see~\eqref{cont.YL2},
\begin{align}
 \norm{z_N(t)}{V}^2\le D_Y\norm{-A\clI_{[N]}-A_{\rm rc} \clI_{[N]}}{L^2(I_k,H)}^2,\qquad t\in\overline{I_k}.\notag
\end{align}
By standard computations we find, for all~$t\in\overline{I_k}$,
\begin{align}
 \norm{z_N(t)}{V}^2&\le D_Y\left(\norm{A\clI_{[N]}}{L^2(I_k,H)}+\norm{A_{\rm rc} \clI_{[N]}}{L^2(I_k,H)}\right)^2\notag\\
 &\le D_Y\left(\norm{\clI_{[N]}}{L^2(I_k,\rmD(A))}+ C_{\rm rc} \norm{\clI_{[N]}}{L^2(I_k,V)}\right)^2\notag\\
 &= D_Y(1+C_{\rm rc}\norm{\Id}{\clL(\rmD(A),V)})^2T\norm{\clI_{[N]}}{L^\infty(I_k,\rmD(A))}^2,\notag
\end{align}
  and using~\eqref{bound.relax},
 \begin{align}
 \norm{z_N(t)}{V}&\le \frac{1}{N}D_Y^\frac12(1+C_{\rm rc}\norm{\Id}{\clL(\rmD(A),V)})
 T^\frac{3}{2}(M+1)\dnorm{\widetilde \Phi}{}\fkK\norm{\fkv}{V}.\notag
 \end{align}
  Now we can take~$N$ large enough, namely
 \begin{equation}\label{hatN}
  N=\widehat N\ge D_Y^\frac12(1+C_{\rm rc}\norm{\Id}{\clL(\rmD(A),V)})T^\frac32(M+1)\dnorm{\widetilde \Phi}{}\fkK
  \tfrac{10}{1-\theta},
 \end{equation}
 in order to obtain~$\norm{z_{\widehat N}(t)}{V}\le\tfrac{1-\theta}{10}\norm{\fkv}{V}$.
 Then, we set
 \begin{equation}\label{v2}
 \clV^2(\fkv)(t)\coloneqq \clV_{[\widehat N]}(\fkv)(t)=\sign(l_{k,n,j})\Sigma_{n}(\fkv)\widetilde\Phi_j,
\quad\mbox{if}\;\; t\in I_{k,n,j},\;\; 1\le j\le 2M,
 \end{equation}
 where the intervals~$I_{k,n,j}$ are defined as in~\eqref{pcc-VN.Sig.n0} and~\eqref{pcc-VN.Sig.0},
 \begin{equation}\label{v2-Int}
  I_{k,n,j} =I_{k,n,j}^{[\widehat N]}=[t_{k,n,j-1}^{[\widehat N]},t_{k,n,j}^{[\widehat N]})=[t_{k,n,j-1},t_{k,n,j}).
 \end{equation}
 We find, using~\eqref{bound.relax},
\begin{subequations}\label{diff_V21}
\begin{align}
 &\norm{\fkY_k(\fkv,\clV^2(\fkv))(t)-\fkY_k(\fkv,\clV^1(\fkv))(t)}{V}
 =\norm{d_{\widehat N}(t)}{V}\le \norm{z_{\widehat N}(t)}{V}+\norm{\clI_{[\widehat N]}(t)}{V}\notag\\
 &\hspace*{4em}\le \tfrac{1-\theta}{10}\norm{\fkv}{V}+\tfrac{T}{\widehat N}(M+1)\dnorm{\widetilde \Phi}{}
 \fkK\norm{\fkv}{V},
 \quad\mbox{for all}\quad t\in\overline{I_k},
 \intertext{and, using~\eqref{zN=dN.n},}
 &\norm{\fkY_k(\fkv,\clV^2(\fkv))(kT+T)-\fkY_k(\fkv,\clV^1(\fkv))(kT+T)}{V}
 =\norm{d_{\widehat N}(kT+T)}{V}\notag\\
 &\hspace*{4em}=\norm{z_{\widehat N}(kT+T)}{V}
 \le \tfrac{1-\theta}{10}\norm{\fkv}{V}.
\end{align}
\end{subequations}
\begin{remark}
The chosen cycle~\eqref{cycle} is not unique. For example, we could adapt the proof for the cycle $\widetilde\Phi_1\to\widetilde\Phi_2\to\widetilde\Phi_{M-1}\to\widetilde\Phi_M\to\widetilde\Phi_1$.
\end{remark}

\subsection{A piecewise constant static control taking values in $H$}\label{sS:mainproof.pccH}
To simplify the exposition we denote
\begin{equation}\label{hatPhi-ext}
 \widehat\Phi_{2M+1-j}\coloneqq\widehat\Phi_j,\qquad 1\le j\le M.
\end{equation}
Recall that~$\clV^2(\fkv)$ takes its values~$\clV^2(\fkv)(t)$ in the set~$\{\pm\Sigma_{n}(\fkv)\widetilde\Phi_j\}\subset\rmD(A)$,
for~$t\in\overline{I_k}$. We define a new
piecewise constant control~$\clV^3(\fkv)(t)$, taking its values in the
set~$\{\pm\Sigma_{n}(\fkv)\widehat\Phi_j\}\subset H$, by
\begin{equation}\label{v3}
\clV^3(\fkv)(t)=\sign(l_{k,n,j})\Sigma_{n}(\fkv)\widehat\Phi_j,
 \qquad\mbox{for}\quad \clV^2(\fkv)(t)=\sign(l_{k,n,j})\Sigma_{n}(\fkv)\widetilde\Phi_j.
\end{equation}
Using~\eqref{Sigma.bound.h}, we can see that the corresponding solutions satisfy
\begin{align}
 &\norm{\fkY_k(\fkv,\clV^3(\fkv))(t)-\fkY_k(\fkv,\clV^2(\fkv))(t)}{V}^2
 \le D_Y\norm{\clV^3(\fkv)-\clV^2(\fkv)}{L^2(I_k,H)}^2\notag\\
 &\hspace*{3em} \le D_Y T\sup_{1\le n\le N}\Sigma_n(\fkv)^2\max_{1\le j\le M}\norm{\widetilde\Phi_j-\widehat\Phi_j}{H}^2\notag\\
 &\hspace*{3em} \le D_Y TM^2\fkK^2\norm{\fkv}{V}^2\max_{1\le j\le M}\norm{\widetilde\Phi_j-\widehat\Phi_j}{H}^2\notag
 \end{align}
 and by~\eqref{choice-tildePhi},
 \begin{subequations}\label{diff_V32}
 \begin{align}
 &\norm{\fkY_k(\fkv,\clV^3(\fkv))(t)-\fkY_k(\fkv,\clV^2(\fkv))(t)}{{ V}}
 \le \tfrac{1-\theta}{10}\norm{\fkv}{{V}},\qquad t\in\overline{I_k}.
  \end{align}
Note also that
\begin{align}
&\norm{\clV^3(\fkv)(t)}{V}\le\sup_{1\le n\le N}\Sigma_n(\fkv)\le M\fkK\norm{\fkv}{V},\qquad t\in \overline{I_k},\\
&\clV^3(\fkv)(kT+(n-1)\tfrac{T}{N})\in\{\pm\Sigma_{n}(\fkv)\widehat\Phi_1\},\qquad 1\le n\le\widehat N.
\end{align}
\end{subequations}

Observe that~$\clV^2$ switches between the actuators~$\widetilde\Phi_j$ as described in~\eqref{cycle}, and hence~$\clV^3$
switches between the actuators~$\widehat\Phi_j$ as in the analogous cycle
\begin{equation}\label{cycle-hat}
 \widehat\Phi_1\to\widehat\Phi_2\to\dots\to\widehat\Phi_{M-1}\to\widehat\Phi_M
 \to\widehat\Phi_M\to\widehat\Phi_{M-1}\to\dots\to\widehat\Phi_2\to\widehat\Phi_1,
\end{equation}
which, due to~\eqref{tildePhi-ext}, results in
\begin{equation}\notag
 \widehat\Phi_1\to\widehat\Phi_2\to\dots\to\widehat\Phi_{M-1}\to\widehat\Phi_M
 \to\widehat\Phi_{M+1}\to\widehat\Phi_{M+2}\to\dots\to\widehat\Phi_{2M-1}\to\widehat\Phi_{2M}.
\end{equation}
\subsection{A control with no degenerate intervals of constancy}\label{sS:mainproof.nondeg}
By construction, the length~$\norm{l_{k,n,j}}{}$ vanishes if $v^1_j(\fkv)$ has zero average on~$I_{k,n}$, see~\eqref{lknj}.
 It will be convenient, also to simplify the exposition in Section~\ref{sS:mainproof.movH} below, that
all the intervals of constancy have a length larger than a suitable positive constant.

Recall the switching time instants~$t_{k,n,j}$ on each interval~$\overline{I_{k,n}}\subset \overline{I_k}$, see~\eqref{v2-Int},
\begin{subequations}\label{switch-order}
 \begin{align}
 t_{k,n,0}\le t_{k,n,1}\le t_{k,n,2}\le \dots\le  t_{k,n,2M-1}\le t_{k,n,2M},\qquad 1\le n\le\widehat N,
  \intertext{and recall that}
 \bft_{k,n-1}\coloneqq kT+(n-1)\tfrac{T}{\widehat N}= t_{k,n,0}\quad\mbox{and}\quad \bft_{k,n}=kT+n\tfrac{T}{\widehat N}= t_{k,n,2M}.
\end{align}
\end{subequations}
To guarantee nondegenerate intervals of constancy we will define new switching time instants satisfying
\begin{subequations}\label{switch-order-nondeg}
\begin{align}
t^\e_{k,n,0} < t^\e_{k,n,1}<  t^\e_{k,n,2}< \dots\le  t^\e_{k,n,2M-1}< t^\e_{k,n,2M}
\intertext{with}
 t^\e_{k,n,0}\coloneqq\bft_{k,n-1}\quad\mbox{and}\quad t^\e_{k,n,2M} \coloneqq \bft_{k,n}.
 \end{align}
\end{subequations}
To do so we fix a positive number~$\e>0$, and define
\begin{subequations}\label{vartheta}
\begin{align}
 t^\e_{k,n,j}&=\bft_{k,n-1} +\vartheta_\e(t_{k,n,j}-\bft_{k,n-1}+\tfrac{j+1}{2}j\e),\qquad 0\le j\le 2M,
 \intertext{with}
 \vartheta_\e&\coloneqq
 \tfrac{T}{T+\widehat N(2M+1)M\e}.
\end{align}
\end{subequations}
Now we show that, indeed,  the sequence~\eqref{vartheta} satisfies~\eqref{switch-order-nondeg}.
We find
\begin{subequations}\label{switch-order-nondeg-proof}
\begin{align}
t^\e_{k,n,0}&=\bft_{k,n-1}+\vartheta_\e 0=\bft_{k,n-1},\\
t^\e_{k,n,2M}&=\bft_{k,n-1}+\vartheta_\e(\tfrac{T}{\widehat N}+(2M+1)M\e)=\bft_{k,n-1}+\tfrac{T}{\widehat N}=\bft_{k,n},\\
t^\e_{k,n,j}-t^\e_{k,n,j-1}&=\vartheta_\e(t_{k,n,j}-t_{k,n,j-1}+(\tfrac{j+1}{2}j-\tfrac{j}{2}(j-1))\e),\notag\\
&\ge\vartheta_\e(\tfrac{j+1}{2}j-\tfrac{j}{2}(j-1))\e=\e\vartheta_\e j>0,\qquad 1\le j\le 2M.
\end{align}
\end{subequations}
From~\eqref{switch-order-nondeg-proof} we see that~\eqref{switch-order-nondeg} is satisfied.

Next we define the piecewise constant control, for time~$t\in I_k$, as follows
\begin{align}
\clV_\e(\fkv)(t)\coloneqq\sign(l_{k,n,j})\Sigma_n(\fkv)\widehat\Phi_j,
&\quad\mbox{if}\quad t\in [t^\e_{k,n,j-1},t^\e_{k,n,j}),\label{v.eps}\\
&\quad\mbox{for}\quad1\le n\le\widehat N,\quad 1\le j\le 2M.\notag
\end{align}
where the intervals of constancy have a positive minimum length, see~\eqref{switch-order-nondeg-proof},
\begin{equation}\label{ell.eps}
 \min_{1\le j\le 2M}\{t^\e_{k,n,j-1}-t^\e_{k,n,j-1}\}\ge\e\vartheta_\e  >0.
\end{equation}
Observe that, from~\eqref{v2} and~\eqref{v3}, we have that
\[
\clV^3(\fkv)(t)=\sign(l_{k,n,j})\Sigma_n(\fkv)\widehat\Phi_j,\quad\mbox{if}\quad t\in [t_{k,n,j-1},t_{k,n,j}).
\]

Note that as~$\e\to0$ we have~$t^\e_{k,n,j}\to t_{k,n,j}$.
Now we show that we also have~$\clV_\e(\fkv)(t)\to \clV^3(\fkv)(t)$ in~$L^2(I_k,H)$, as~$\e\to0$.

\begin{proposition}\label{P:cont.tswitch-L2}
 Let $[a,b]\in\bbR$ be a nonempty interval, $a<b$, let $X$ be a Banach space, and let~$K$ be positive integer.
 Let us be given a
 finite sequence in~$X$
 \begin{align}
&\phi_j\in X,\qquad 1\le j\le K,
\intertext{and two finite sequences in~$[a,b]$,}
  &a=\tau_0\le \tau_1\le \dots \le\tau_{K-1}\le\tau_{K}= b
  \quad\mbox{and}\quad a=\sigma_0\le \sigma_1\le \dots \le\sigma_{K-1}\le\sigma_{K}= b.\notag
\end{align}
Then, for the following two functions defined for~$t\in(a,b)$ by
\begin{align}
 f_\tau(t)\coloneqq  \phi_j\quad\mbox{if}\quad t\in [\tau_{j-1},\tau_j)
 \qquad\mbox{and}\qquad
 f_\sigma(t)\coloneqq  \phi_j\quad\mbox{if}\quad t\in [\sigma_{j-1},\sigma_j),\notag
\end{align}
we have the estimate
\[
 \norm{f_\tau-f_\sigma}{L^2((a,b),X)}\le K^\frac{1}{2}\clR^\frac{1}{2}\clX,
\]
with~$\clR\coloneqq\max\limits_{0\le j\le K}\norm{\tau_j-\sigma_j}{\bbR}$
and~$\clX\coloneqq\max\limits_{1\le i,j\le K}\norm{\phi_j-\phi_i}{X}$.
\end{proposition}

The proof of Proposition~\ref{P:cont.tswitch-L2} is given in Section~\ref{sS:proofP:cont.tswitch-L2}.

From Proposition~\ref{P:cont.tswitch-L2} it follows that
\begin{align}
&\norm{\clV_\e(\fkv)- \clV^3(\fkv)}{L^2(I_{k},H)}\notag\\
&\hspace*{3em}\le (2M\widehat N)^\frac12
 \max\limits_{\begin{subarray}{l}0\le j\le 2M\\
               1\le n\le \widehat N
              \end{subarray}
}\norm{t^\e_{k,n,j}-t_{k,n,j}}{\bbR}^\frac{1}{2}
 \max\limits_{\begin{subarray}{l}1\le i,j\le M\\
               1\le n\le \widehat N
              \end{subarray}
}\norm{\Sigma_n(\fkv)\widehat\Phi_j-\Sigma_n(\fkv)\widehat\Phi_i}{H}\notag\\
 &\hspace*{3em}\le 2(2M\widehat N)^\frac12\max\limits_{1\le n\le \widehat N}\Sigma_n(\fkv) \max\limits_{\begin{subarray}{l}0\le j\le 2M\\
               1\le n\le \widehat N
              \end{subarray}}\norm{t^\e_{k,n,j}-t_{k,n,j}}{\bbR}^\frac{1}{2}.\notag
 \end{align}

 Recalling~\eqref{Sigma.bound.h}, we arrive at
 \begin{equation}\label{Ve-V3-1}
  \norm{\clV_\e(\fkv)- \clV^3(\fkv)}{L^2(I_{k},H)}\le (2M)^\frac32\widehat N^\frac12 \fkK\norm{\fkv}{V}\max\limits_{\begin{subarray}{l}0\le j\le 2M\\
               1\le n\le \widehat N
              \end{subarray}}\norm{t^\e_{k,n,j}-t_{k,n,j}}{\bbR}^\frac{1}{2}.
 \end{equation}

Next, from~\eqref{vartheta} we find that
\begin{align}
 \norm{t^\e_{k,n,j}-t_{k,n,j}}{\bbR}&=\norm{\bft_{k,n-1} -t_{k,n,j}
 +\vartheta_\e(t_{k,n,j}-\bft_{k,n-1}+\tfrac{j+1}{2}j\e)}{\bbR}\notag\\
 &=\norm{(\vartheta_\e-1)(t_{k,n,j}-\bft_{k,n-1})+\tfrac{j+1}{2}j\e\vartheta_\e}{\bbR}\notag\\
 &\le(1-\vartheta_\e)\tfrac{T}{\widehat N}+(2M+1)M\e\vartheta_\e\eqqcolon\varTheta(\e).\label{te-t}
  \end{align}

 By combining~\eqref{Ve-V3-1} and~\eqref{te-t}, we obtain that
 \begin{equation}\label{Ve-V3-2}
  \norm{\clV_\e(\fkv)- \clV^3(\fkv)}{L^2(I_{k},H)}\le (2M)^\frac32\widehat N^\frac12 \fkK\norm{\fkv}{V}\varTheta(\e)^\frac12,
   \end{equation}
and by using~\eqref{cont.YL2} it follows that
 \begin{align}\label{diffVe3}
 \hspace*{-.5em}\norm{\fkY_k(\fkv,V_e(\fkv))(t)-\fkY_k(\fkv,\clV^3(\fkv))(t)}{V}
 \le (8M^3D_Y\widehat N)^\frac12\fkK\varTheta(\e)^\frac12\norm{\fkv}{V},\quad t\in \overline{I_k}.
  \end{align}

From~\eqref{vartheta} we also find
 \begin{align}
  1-\vartheta_\e&=
  \tfrac{\widehat N(2M+1)M\e}{T+\widehat N(2M+1)M\e},\qquad \e\vartheta_\e=\tfrac{T\e}{T+\widehat N(2M+1)M\e},\notag
 \end{align}
and
\[
 1-\vartheta_\e\to0,\quad\e\vartheta_\e\to0,\quad\mbox{and}\quad\varTheta(\e)\to0,\qquad\mbox{as}\quad\e\to0.
\]
Therefore, there exists $\widehat\e$ small enough, so that
\begin{equation}\label{hatepsil}
 \varTheta(\widehat\e)\le (8M^3D_Y\widehat N)^{-1}\fkK^{-2}(\tfrac{1-\theta}{10})^2.
\end{equation}

Now we set the control
\begin{subequations}\label{v4}
\begin{align}
&\clV^4(\fkv)=\clV_{\widehat\e}(\fkv),\qquad t\in I_{k},
\intertext{with nondegenerate intervals of constancy (cf.~\eqref{ell.eps}),}
  &\min_{1\le j\le 2M}\{t^{\widehat\e}_{k,n,j}-t^{\widehat\e}_{k,n,j-1}\}\ge\widehat\e\vartheta_{\widehat\e}  >0.
 \end{align}
 \end{subequations}
From~\eqref{diffVe3}, and~\eqref{hatepsil}, it follows  that
 \begin{equation}\label{diff_V43}
  \norm{\fkY_k(\fkv,\clV^4(\fkv))(t)-\fkY_k(\fkv,\clV^3(\fkv))(t)}{V}\le \tfrac{1-\theta}{10}\norm{\fkv}{V},
  \qquad t\in \overline{I_{k}}.
  \end{equation}

\subsection{A continuously moving control taking values in $H$}\label{sS:mainproof.movH}

We will travel in~$H$ between the static actuators~$\widehat\Phi_i$, following the cycle~\eqref{cycle-hat}.

For traveling we fix a set of roads, in the unit sphere~$\fkS_H$, connecting the static actuators, as follows:
\begin{subequations}\label{roads-pcc}
\begin{align}
 &\fkR_j\colon \clC^p([0,1],\fkS_H),&&\quad  1\le j\le M-1,\quad p\in\bbN,\\
 &\hspace*{1em}\mbox{with}\quad \fkR_j(0)=\widehat\Phi_j\quad\mbox{and}\quad\fkR_j(1)=\widehat\Phi_{j+1},\\
 &\fkR_{M}(s)=\widehat\Phi_M, &&\quad s\in[0,1], \\
 &\fkR_{M+j}(s)=\fkR_{M-j}(1-s),&& \quad 1\le j\le M-1.
\end{align}

Note that by~\eqref{hatPhi-ext}, we also have
\begin{align}
 &\fkR_{M+j}(0)=\fkR_{M-j}(1)=\widehat\Phi_{M-j+1}=\widehat\Phi_{M+j},\\
 &\fkR_{M+j}(1)=\fkR_{M-j}(0)=\widehat\Phi_{M-j}=\widehat\Phi_{M+j+1}.
\end{align}
\end{subequations}

We also introduce the scalar function
\begin{subequations}\label{roads-scalar}
\begin{align}
 &r_{k,n,j}^{\widehat\e,\xi}(t)\coloneqq
 \left(\tfrac{\xi+t_{k,n,j}^{\widehat\e}-t}{2\xi}\right)\sign(l_{k,n,j})
 +\left(\tfrac{\xi-t_{k,n,j}^{\widehat\e}+t}{2\xi}\right)\sign(l_{k,n,j+1}),
 \intertext{with, recall~\eqref{v4},}
 &\hspace*{6em}\xi\in(0,\tfrac{\widehat \e\vartheta_{\widehat \e}}{2}).
\end{align}
\end{subequations}

Then we define a moving control~$\clV_\xi(\fkv)$, for~$t\in \overline{I_k}$ as follows:
\begin{subequations}\label{pcc-VNxi}
\begin{align}
&\clV_\xi(\fkv)(t)\coloneqq
\sign(l_{k,n,1})\Sigma_{n}(\fkv)\widehat\Phi_1,\\
&\hspace*{2em} \mbox{if}\quad t\in [kT+(n-1)\tfrac{T}{\widehat N}, t_{k,n,1}^{\widehat\e}-\xi],\qquad 1\le n\le\widehat N.\notag\\
&\clV_\xi(\fkv)(t)\coloneqq
\sign(l_{k,n,2M})\Sigma_{n}(\fkv)\widehat\Phi_1,\\
&\hspace*{2em} \mbox{if}\quad t\in [t_{k,n,2M-1}^{\widehat\e}+\xi,kT+n\tfrac{T}{\widehat N}],\qquad 1\le n\le\widehat N.\notag\\
&\clV_\xi(\fkv)(t)\coloneqq
\sign(l_{k,n,j})\Sigma_{n}(\fkv)\widehat\Phi_j,\\
&\hspace*{2em} \mbox{if}\quad t\in [t^{\widehat\e}_{k,n,j-1}+\xi,t^{\widehat\e}_{k,n,j}-\xi ],
\qquad 1\le n\le\widehat N,\quad 2\le j\le 2M-1.\notag\\
&\clV_\xi(\fkv)(t)\coloneqq r_{k,n,j}^{\widehat\e,\xi}(t)\Sigma_n(\fkv) { \fkR_j}(\tfrac{\xi-t_{k,n,j}^{\widehat\e}+t}{2\xi}),\\
 &\hspace*{2em} \mbox{if}\quad t\in [t^{\widehat\e}_{k,n,j}-\xi,t^{\widehat\e}_{k,n,j}+\xi ],
 \qquad 1\le n\le\widehat N,\quad 1\le j\le 2M-1.\notag
\end{align}
\end{subequations}

Observe that $\clV_\xi$ differs from~$\clV^4$ only in the intervals~$(t^{\widehat\e}_{k,n,j}-\xi,t^{\widehat\e}_{k,n,j}+\xi )$,
$ 1\le n\le\widehat N,\quad 1\le j\le 2M-1$, when we travel from the static actuator~$\widehat\Phi_j$ to
the static actuator~$\widehat\Phi_{j+1}$.
These are exactly~$\widehat N(2M-1)$ intervals, where each has length~$2\xi$. Thus
\begin{align}
 \norm{\clV_\xi(\fkv)-\clV^4(\fkv)}{L^2(I_k,H)}^2
 &\le 2\xi\widehat N(2M-1)\max_{\begin{subarray}{l}
 1\le n\le\widehat N\notag\\
 1\le j\le 2M-1\notag
 \end{subarray}}
 \left\{\norm{r_{k,n,j}^{\widehat\e,\xi}(t)}{\bbR}^2\Sigma_n(\fkv)^2\norm{\widehat\Phi_j}{H}^2\right\}.
 \end{align}
 Since~$\norm{r_{k,n,j}^{\widehat\e,\xi}(t)}{\bbR}\le1$ and $\norm{\widehat\Phi_j}{H}=1$, using~\eqref{Sigma.bound.h}, we arrive at
 \begin{align}
 \norm{\clV_\xi(\fkv)-\clV^4(\fkv)}{L^2(I_k,H)}^2&\le 2\xi\widehat N(2M-1)M^2\fkK^2\norm{\fkv}{H}^2.\notag
  \end{align}

  Recalling~\eqref{cont.YL2}, we obtain
  \begin{align}
  &\norm{\fkY_k(\fkv,\clV_\xi(\fkv))(t)-\fkY_k(\fkv,\clV^4(\fkv))(t)}{V}^2\le D_Y\norm{\clV_\xi(\fkv)-\clV^4(\fkv)}{L^2(I_k,H)}^2\notag\\
  &\hspace*{4em}\le D_Y2\xi\widehat N(2M-1)M^2\fkK^2\norm{\fkv}{V}^2,\qquad t\in I_k.\label{diff_Vxi4}
  \end{align}

 Now choosing small enough~$\xi$, namely
\begin{equation}\label{hat-xi}
 \xi=\widehat\xi\coloneqq \min\left\{\frac{\widehat \e\vartheta_{\widehat \e}}{2},
 \left(D_Y2\widehat N(2M-1)M^2\fkK^2\right)^{-1}(\tfrac{1-\theta}{10})^2\right\},
\end{equation}
and setting
\begin{equation}\label{v5}
 \clV^5(\fkv))(t)\coloneqq \clV_{\widehat\xi}(\fkv))(t),\qquad t\in I_k,
\end{equation}
we find
\begin{equation}\label{diff_V54}
 \norm{\fkY_k(\fkv,\clV^5(\fkv))(t)-\fkY_k(\fkv,\clV^4(\fkv))(t)}{H}
 \le \tfrac{1-\theta}{10}\norm{\fkv}{H}.
\end{equation}

Finally, note that $\clV^5(\fkv)$ is a moving control of the form
\begin{subequations}\label{v5-prop}
 \begin{equation}
  \clV^5(\fkv)(t)\eqqcolon u(t)\Phi(t),\;\; t\in \overline{I_k},\;\;\mbox{with}\;\;\norm{u(t)}{H}\le M\fkK\norm{\fkv}{V},\;\;
  \Phi(t)\in\fkS_H,
\end{equation}
for suitable~$u\in L^\infty(\overline{I_k},\bbR)$ and~$\Phi\in\clC(\overline{I_k},\fkS_H)$. Furthermore,
by choosing~$p\ge0$ in~\eqref{roads-pcc} we can obtain a regular motion of the actuator. Namely, if
we have
\begin{align}
 \tfrac{\ed^q}{\ed s^q}\rest{s=0}\fkR_j=0=\tfrac{\ed^q}{\ed s^q}\rest{s=1}\fkR_j, \qquad 1\le q\le p,
 \end{align}
 then
\begin{align}
 &\Phi\in\clC^p(\overline{I_k},\fkS_H)
 \intertext{and}
 & \max_{\tau\in\overline{I_k}}\norm{\tfrac{\ed^p}{\ed t^p}\rest{t=\tau}\Phi}{H}\le(\tfrac{1}{2\widehat\xi})^p
 \max_{1\le j\le M}\max_{s_0\in[0,1]}\norm{\tfrac{\ed^p}{\ed s^p}\rest{s=s_0}\fkR_j}{H}.\notag
 \intertext{In particular, we have that}
 &\norm{\Phi}{\clC^p(\overline{I_k},\fkS_H)}\le(\tfrac{1}{2\widehat\xi})^p
 \max_{1\le j\le M}\norm{\fkR_j}{\clC^p([0,1],H)}.
  \end{align}
  \end{subequations}

\medskip\noindent
{\em Conclusion of the proof of Theorem~\ref{T:main}.}  By
using~\eqref{diff_V10}, \eqref{diff_V21}, \eqref{diff_V32}, \eqref{diff_V43}, and~\eqref{diff_V54},
together with the triangle inequality, we arrive at
\[
 \norm{\fkY_k(\fkv,\clV^5(\fkv))(kT+T)-\fkY_k(\fkv,\clV^0(\fkv))(kT+T)}{H}\le 5\tfrac{1-\theta}{10}\norm{\fkv}{H}=\tfrac{1-\theta}{2}\norm{\fkv}{H}.
\]
 Finally, note that the choice of~$\widehat\varepsilon$ in~\eqref{hatepsil} and that of~$\xi$ in~\eqref{hat-xi} are independent of~$y_0$.
This finishes the proof of Theorem~\ref{T:main}. \qed

\begin{remark}
Observe that the actuators~$\widetilde\Phi_j$ in~\eqref{choice-tildePhi}, the integer~$\widehat N$ in~\eqref{hatN},
the parameter~$\widehat\e$ in~\eqref{hatepsil}, and the parameter~$\widehat\xi$ in~\eqref{hat-xi}, were all  chosen independently
of~$k\in\bbN$.
Furthermore, from~\eqref{v5-prop} we can also see that~$\norm{\Phi}{\clC^p(\overline{I_k},\fkS_H)}$
is bounded by a constant independent of~$k\in\bbN$. Again from~\eqref{v5-prop}, by recalling Assumption~\ref{A:MstaticAct} we also have
$\norm{u}{L^\infty(I_k,H)}\le M\fkK\norm{\fkv}{V}=M\fkK\norm{y(kT)}{V}$ with the product~$M\fkK$ independent of~$k\in\bbN$.
 \end{remark}

\section{Proof of Theorem~\ref{T:main-Intro}}\label{S:proofT:main-Intro}
We start by writing~\eqref{sys-y-parab-Intro} as
\begin{align}\label{sys-y-parab}
 \dot y+A y+A_{\rm rc}y=u\widehat
 \indf_{\omega(c)},\quad y(0)=y_0,\quad t>0,
\end{align}
with~$A\coloneqq-\nu\Delta +\Id$  and~$A_{\rm rc}z=A_{\rm rc}(t)z\coloneqq(a(t,\Bigcdot)-1)z+b(t,\Bigcdot)\cdot\nabla z$.

It is not hard to check that Assumptions~\ref{A:A0sp},~\ref{A:A0cdc}, and~\ref{A:A1}, are satisfied by~$A\in\clL(V,V')$
and~$A_{\rm rc}\in L^\infty(\bbR_0,\clL(V,H))$,
namely with~$V=H^1_0(\Omega)$ in the case of Dirichlet boundary conditions and with~$V=H^1(\Omega)$
(see, e.g.,~\cite[Sect.5.1]{Rod20-eect}).

\subsection{Satisfiability of Assumption~\ref{A:MstaticAct}}
Assumption~\ref{A:MstaticAct} follows from the results in~\cite[Thm.~4.5]{Rod20-eect} (when applied to linear equations),
from which we know that
\begin{align}\label{sys-y-parab-obli}
 \dot y+A y+A_{\rm rc}y=P_{U_M}^{E_M^\perp}\left(A_{\rm rc}y-\lambda y\right),\quad y(0)=y_0,\quad t>0,
\end{align}
is a stable system for a suitable oblique projection~$P_{U_M}^{E_M^\perp}$, where~$\lambda>0$. Namely, its solution satisfies,
\begin{align}\label{sys-y-parab-obli-exp}
 \norm{y(t)}{ V}\le C\ex^{-\mu (t-s)}\norm{y(s)}{ V}, \qquad t\ge s\ge0,
\end{align}
with~$C\ge 1$ and~$\mu\ge 0$ independent of~$(t,s)$.
Actually in~\cite[Thm.~4.5]{Rod20-eect} only the case~$s=0$ is mentioned, however by a time shift argument~$t\eqqcolon s+\tau$,
$w(\tau)=y(s+\tau)$, $\widetilde A_{\rm rc}(\tau)=A_{\rm rc}(s+\tau)$, we can
rewrite~\eqref{sys-y-parab-obli} as
\begin{align}
 \tfrac{\ed}{\ed\tau} w+A w+\widetilde A_{\rm rc}w=P_{U_M}^{E_M^\perp}\left(\widetilde A_{\rm rc}w-\lambda w\right),
 \quad w(0)=y(s),\quad \tau>0,\notag
\end{align}
and the results in~\cite[Thm.~4.5]{Rod20-eect} give us
\begin{align}
 \norm{w(\tau)}{V}\le C_s\ex^{-\mu \tau}\norm{y(s)}{V}, \qquad \tau\ge0,\notag
\end{align}
which is equivalent to~\eqref{sys-y-parab-obli-exp}. The constant~$C_s$ is of the
form~$\ovlineC{\norm{\widetilde A_{\rm rc}}{L^\infty(\bbR_0,\clL(V,H))}}$, and so
$C_s\le C_0$, that is we can take~$C$ independent of~$s$ in~\eqref{sys-y-parab-obli-exp}.

This stability result in~\cite[Thm.~4.5]{Rod20-eect} holds for large enough~$M$,
where~$U_M=\linspan\{\indf_{\omega_j}\mid 1\le j\le M\}$
is the span of suitable indicator functions supported
in small rectangles~$\omega_j\subset\Omega$. The operator~$P_{U_M}^{E_M}$ is the
oblique projection in~$L^2(\Omega)$ onto~$U_M$ along an auxiliary space~$E_M^\perp$, where~$E_M$ is the span of a suitable
set of eigenfunctions of the diffusion~$A$ defined in~$L^2(\Omega)$, where~$\Omega$ is a bounded rectangular domain.
For precise definitions of suitable~$U_M$ and~$E_M$ we refer to~\cite[Sect.~4.8.1]{KunRod18-cocv} and~\cite[Sect.~2.2]{Rod20-eect}.
Furthermore, for such  choice we have
\begin{equation}\label{bound-OP}
 \sup_{M\ge1}\norm{P_{U_M}^{E_M^\perp}}{\clL(H)}\eqqcolon\dnorm{P}{}<+\infty.
\end{equation}

Observe that $U_M$ is the range of~$P_{U_M}^{E_M}$, hence our control is of the form
\[
 P_{U_M}^{E_M^\perp}\left(A_{\rm rc}y-\lambda y\right)\eqqcolon \sum_{j=1}^M\widetilde v_j(t)\indf_{\omega_j}
 =\sum_{j=1}^Mv_j(t)\widehat\Phi_j\eqqcolon\bfv(t)
\]

In particular, from~\eqref{sys-y-parab-obli-exp}, for any given~$\theta\in(0,1)$ we have that, for all~$k\in\bbN$,
\[
 \norm{y(kT+T)}{V}\le C\ex^{-\mu T}\norm{y(kT)}{V}\le \theta\norm{y(kT)}{V},
 \quad\mbox{if}\quad T\ge\mu^{-1}\log(\tfrac{C}{\theta}).
\]

To prove that Assumption~\ref{A:MstaticAct} is satisfied, it remains to show that the~$v_j(t)$ are appropriately
essentially bounded. From~\eqref{sys-y-parab-obli-exp} and~$P_{U_M}^{E_M^\perp}=P_{U_M}^{E_M^\perp}P_{E_M}$,
we find
\begin{align}
\norm{\bfv(t)}{H}&\le\dnorm{P}{}\norm{P_{E_M}\left(A_{\rm rc}y-\lambda y\right)}{H}\notag\\
 &\le
 \left(\norm{A_{\rm rc}}{L^\infty(\bbR_0,\clL(V,H))}+\lambda\norm{\Id}{\clL(V,H)}\right)\dnorm{P}{} \norm{y(t)}{V}\notag\\
 &\le
 \left(C_{\rm rc}+\lambda\norm{\Id}{\clL(V,H)}\right)\dnorm{P}{} C\ex^{-\mu (t-kT)}\norm{y(kT)}{V},\qquad t\ge kT.\label{control2.0}
\end{align}
which implies that
\begin{align}
\norm{\bfv(t)}{H}&\le
 \fkK_0\norm{y(kT)}{V},\quad t\in \overline{I_k}=[kT,kT+T].\label{control2.1}
\end{align}
with~$\fkK_0=\left(C_{\rm rc}+\lambda\norm{\Id}{\clL(V,H)}\right)\dnorm{P}{} C$ independent of~$k$.

Therefore,
Assumption~\ref{A:MstaticAct} holds for~$M$ large enough, and with $T=\mu^{-1}\log(\frac{C}{\theta})$ and~$\fkK=\fkK_0$ as above.

\subsection{Illustration of a path for the moving actuator}\label{sS:illus-react}
We consider the static actuators $\indf_{\omega_i}=\indf_{\omega(c^i)}$
with center~$c^i$ as in~\cite[Sect.~4.8.1]{KunRod18-cocv}, illustrated in figure~\ref{fig.suppsensors}.
Then we order the actuators and consider the corresponding cycle. For example,  as illustrated in the figure for the case~$S=3$,
we start at the first actuator in the
bottom-left corner, going up until the ~$M$th actuator (here, $M=9$) in the top-right corner and returning  back down to the bottom-left corner.

\setlength{\unitlength}{.002\textwidth}
\newsavebox{\Rectfwa}%
\savebox{\Rectfwa}(0,0){%
\linethickness{3pt}
{\color{black}\polygon(0,0)(120,0)(120,80)(0,80)(0,0)}%
}%
\newsavebox{\Rectfga}%
\savebox{\Rectfga}(0,0){%
{\color{lightgray}\polygon*(0,0)(120,0)(120,80)(0,80)(0,0)}%
}%

\newsavebox{\Rectrefa}%
\savebox{\Rectrefa}(0,0){%
{\color{white}\polygon*(0,0)(120,0)(120,80)(0,80)(0,0)}%
{\color{lightgray}\polygon*(45,30)(75,30)(75,50)(45,50)(45,30)}%
}%

%%%%%%%%%%%%%%%%%%%%%%%%%%

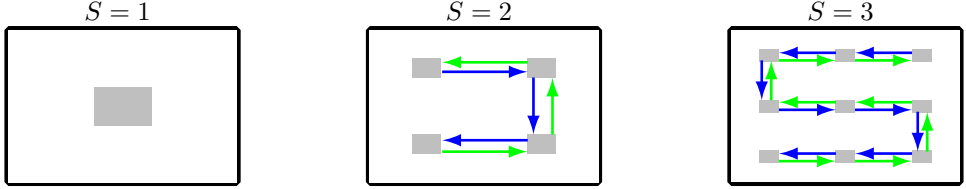
\begin{figure}[h!]
\begin{center}
\begin{picture}(500,100)%(0,0)

%Rectref
\put(0,0){\usebox{\Rectfwa}}%
\put(0,0){\usebox{\Rectrefa}}
% Rect2
 \put(190,0){\usebox{\Rectfwa}}
 \put(190,0){\scalebox{.5}{\usebox{\Rectrefa}}}
 \put(250,0){\scalebox{.5}{\usebox{\Rectrefa}}}
 \put(250,40){\scalebox{.5}{\usebox{\Rectrefa}}}
 \put(190,40){\scalebox{.5}{\usebox{\Rectrefa}}}
% Rect3
 \put(380,0){\usebox{\Rectfwa}}
 \put(380,0){\scalebox{.3333}{\usebox{\Rectrefa}}}
\put(420,0){\scalebox{.3333}{\usebox{\Rectrefa}}}
\put(460,0){\scalebox{.3333}{\usebox{\Rectrefa}}}
\put(380,26.6666){\scalebox{.3333}{\usebox{\Rectrefa}}}
\put(420,26.6666){\scalebox{.3333}{\usebox{\Rectrefa}}}
\put(460,26.6666){\scalebox{.3333}{\usebox{\Rectrefa}}}
\put(380,53.3333){\scalebox{.3333}{\usebox{\Rectrefa}}}
\put(420,53.3333){\scalebox{.3333}{\usebox{\Rectrefa}}}
\put(460,53.3333){\scalebox{.3333}{\usebox{\Rectrefa}}}
\put(40,85){$S=1$}
\put(230,85){$S=2$}
\put(420,85){$S=3$}

%%%%%Roads M=3
\linethickness{1pt}%
{\color{green}
\put(405,11){\vector(1,0){30}}
\put(445,11){\vector(1,0){30}}
\put(435,42){\vector(-1,0){30}}
\put(475,42){\vector(-1,0){30}}
\put(405,64){\vector(1,0){30}}
\put(445,64){\vector(1,0){30}}
\put(483,16){\vector(0,1){21}}
\put(401,43){\vector(0,1){21}}
}%

{\color{blue}
\put(435,15){\vector(-1,0){30}}
\put(475,15){\vector(-1,0){30}}
\put(405,38){\vector(1,0){30}}
\put(445,38){\vector(1,0){30}}
\put(435,68){\vector(-1,0){30}}
\put(475,68){\vector(-1,0){30}}
\put(478,37){\vector(0,-1){21}}
\put(396,64){\vector(0,-1){21}}
}%

%%%%%Roads M=2
{\color{green}
\put(228,16){\vector(1,0){45}}
\put(273,63){\vector(-1,0){45}}
\put(286,25){\vector(0,1){30}}
}%

{\color{blue}
\put(273,22){\vector(-1,0){45}}
\put(228,58){\vector(1,0){45}}
\put(276,55){\vector(0,-1){30}}
}%

\end{picture}
\end{center}

\vspace*{1em}
\caption{Supports of the static actuators. Case~$\Omega\subset\bbR^2$.} \label{fig.suppsensors}
\end{figure}

\begin{subequations}\label{roads-parab}
In this way we are considering roads, see~\eqref{v5-prop},
\begin{equation}
 \fkR_j(s)=\widehat\indf_{\omega(c_j(s))},\qquad\omega(c_j(s))\subset\Omega,\qquad c_j\in\clC^2([0,1],\bbR^d),
\end{equation}
and
\begin{equation}
 c_j(0)=c^j,\qquad c_j(1)=c^{j+1}, \qquad \dot c_j(0)=\dot c_j(1)=\ddot c_j(0)=\ddot c_j(1)=0.
\end{equation}
Note that for such roads, we may take
\begin{equation}
 c_j(s)=c_j+\phi(s)(c_{j+1}-c_j),\qquad1\le j\le M,
\end{equation}
where~$\phi\in \clC^2([0,1],[0,1])$ is increasing and satisfies the relations
$\phi(0)=0$, $\phi(1)=1$,
and~$\dot \phi(0)=\dot \phi(1)=\ddot \phi(0)=\ddot \phi(1)=0$.
\end{subequations}
Furthermore, we have
\begin{equation}\label{choice-xiphi-rect}
 \norm{c}{\clC^m(\bbR_0,\bbR^d}\le (\tfrac{1}{2\widehat\xi})^m\norm{\phi}{\clC^m([0,1],\bbR)},
 \qquad m\in\{0,1,2\}.
\end{equation}
Recall also that~$\widehat\xi$ can be chosen independent of~$y_0$.

\subsection{Conclusion of proof of Theorem~\ref{T:main-Intro}.}\label{sS:proofT:main-Intro}
Let us fix~$y_0\in H$ and~$c_0\in\bbR^M$ with~$\omega(c_0)\subset\Omega$, and let $c^1$ be the center of the static actuator~$\indf_{\omega_1}$. For
~$\widehat c(t)\coloneqq c_0+t^2(2-t)^2(c^1-c_0)$ we have that
\[
 \widehat c(0)=c_0,\quad \widehat c(1)=c^1,\qquad \dot {\widehat c}(0)=\dot {\widehat c}(1)=0,\qquad \ddot {\widehat c}\in L^\infty((0,1),\bbR^M).
\]
We proceed as in Corollary~\ref{C:main1} by taking the actuator path~$\Phi^*(t)=\widehat\indf_{\widehat c(t)}$, for time~$t\in[0,1]$, and the
path illustrated in section~\ref{sS:illus-react} for time~$t\ge1$ where we use Theorem~\ref{T:main0}.
Note that~$\Phi^*(1)=\widehat\indf_{\omega(\widehat c(1))}=\widehat\indf_{\omega_1}$ and~$\dot\Phi^*(1)=0$.

Observe also that
$\norm{\dot{\widehat c}}{W^{1,\infty}((0,1),\bbR^M)}\le
\norm{\dot\varphi}{W^{1,\infty}(0,1)}\norm{c^1-c_0}{\bbR^M}\le C_3$ with
$\varphi(t)\coloneqq t^2(2-t)^2$,  where~$C_3$ can be taken independent of~$c_0$ because~$\Omega$ is bounded.
Therefore, we have that
$\norm{\dot c}{W^{1,\infty}(\bbR_0,\bbR^M)}\le \max\{C_3,\norm{\cdot c}{W^{1,\infty}(\bbR_1,\bbR^M)}\}\le C_4$, with~$C_4$
independent of~$(y_0,c_0)$, because~$\norm{\dot c}{W^{1,\infty}(\bbR_1,\bbR^M)}$ is independent of~$y(1)$ (cf.~\eqref{choice-xiphi-rect}),
hence independent of~$y(0)$.
\qed

\subsection{A remark on Assumption~\ref{A:A1} and weak solutions.}\label{sS:RemarkArc}
Instead of the reaction-convection operator~$A_{\rm rc}\in L^\infty(\bbR_0,\clL(V,H))$, we can
also take~$A_{\rm rc}\in L^\infty(\bbR_0,\clL(H,V'))$ which is the case for a convection term as~$\nabla\cdot(by)$
under homogeneous Dirichlet boundary conditions, with~$b\in L^\infty(\bbR_0,\bbR^d)$. For the latter case we can repeat the procedure and prove
the stabilizability result in the~ {$H$-norm}. That is, we must work with weak
solutions~$y\in \clC(\overline\bbR_0,H)$ instead of strong solutions~$y\in \clC(\overline\bbR_0,V)$.
In particular we would just need to replace~$V$ by~$H$ in Assumption~\eqref{A:MstaticAct} and in~\eqref{cont.YL2}.
Recall that, with~$\widetilde C_{\rm rc}
=\norm{A_{\rm rc}}{L^\infty(\bbR_0,\clL(H,V'))}$ and~$\widetilde D_Y=\ovlineC{T,\widetilde C_{\rm rc}}$,
we will have (cf.~\cite[Lem.~2.2]{PhanRod18-mcss},
recalling that~$\clC(\overline{I_k},H)\xhookrightarrow{}W(I_k,V,V')$)
\begin{equation}\notag
\norm{\fkY_k(\fkv,f)(t)}{H}^2\le \widetilde D_Y\left(\norm{\fkv}{H}^2+\norm{f}{L^2(I_k,V')}^2\right),\quad t\in I_k=(kT,kT+T).
\end{equation}

Concerning parabolic equations we can see that instead of~\eqref{control2.0} we would obtain
\begin{align}
\norm{\bfv(t)}{V'}&\le\dnorm{P}{}
\left(
\norm{P_{E_M}}{\clL(V',H)}\norm{A_{\rm rc}y}{V'}+\lambda\norm{y}{H}\right)\notag\\
 &\le
 \left(\norm{P_{E_M}}{\clL(V',H)}\norm{A_{\rm rc}}{L^\infty(\bbR_0,\clL(H,V'))}+\lambda\right)\dnorm{P}{} \norm{y(t)}{H}\notag\\
 &\le
 \left(\norm{P_{E_M}}{\clL(V',H)}C_{\rm rc}+\lambda\right)\dnorm{P}{} C\ex^{-\mu (t-kT)}\norm{y(kT)}{H},\qquad t\ge kT.
 \label{control2.0a}
\end{align}
which implies that
\begin{align}
\norm{\bfv(t)}{V'}&\le
 \fkK_0\norm{y(kT)}{H},\quad t\in \overline{I_k}=[kT,kT+T].\label{control2.1a}
\end{align}
Such inequality implies that the control is essentially bounded as required in Assumption~\eqref{A:MstaticAct}.

Such weak solutions are also defined for~$A_{\rm rc}\in L^\infty(\bbR_0,\clL(V,H))$, but we cannot show that
the control remains essentially bounded in the case we only know that~$\norm{y(t)}{H}$ remains bounded. For that
we would need to bound~\eqref{control2.0} by~$\norm{y(kT)}{H}$ instead of~$\norm{y(kT)}{V}$, but this seems to be not possible
in general.

%%%%%%%%%%%%%%%%%%%%%%%%%%%%%%%%%%%%%%%%%%%%%%%%%%%%%%%%%%%%%%%%%%%%%%%%%%%
\section{Numerical simulations}\label{S:num_impl}
 According to the construction in Sections ~\ref{S:stability}  and ~\ref{S:proofT:main-Intro},
 once we have fixed a set of roads~$\fkR_j$, see~\eqref{roads-parab}, we could compute the
moving control~$\clV^5$ in~\eqref{v5} from the control~$\clV^0$ given
by~\eqref{sys-y-parab-obli} simply by setting ~$N=\widehat N$ large enough,  ~$\e=\widehat\e$  small enough
and~$\xi=\widehat\xi$, and by computing
the scalars~$l_{k,n,j}$ in~\eqref{pcc-VNxi}, from which we could also
compute~$\Sigma_n(y(kT))$, the switching times in~\eqref{vartheta}.
However, we would obtain an actuator~$\clV^5$ which would be moving very fast by visiting all initial
static actuators twice in a each interval of time~$i\frac{T}{N}$, $i\in\bbN$. In applications,
this is likely not the ``best''
motion for the actuator, sometimes it would be better to stay longer in a particular region or
it would be better to leave
the roads~$\fkR_j$ in order to cover other regions of~$\Omega$.
Therefore, we are going to compute the center~$c=c(t)$ of the moving actuator and the
control magnitude~$u=u(t)$ using tools from optimal control.

%%%%%%%%%%%%%%%%%%%%%%%%%%
\subsection{Computation of a stabilizing single actuator based receding horizon control}
We deal with system~\eqref{sys-y-parab-Intro-Ext},
 where now we will   consider~$(y,c)$ as the state of the system and~$(u,\eta)$ as the control.
 Note that~$\eta$ can be seen as a control on
 the acceleration of~$c$, which also makes sense from the applications point of view,
 where we cannot  change instantaneously the velocity
 of a device, but instead we can apply a force/acceleration to it.
 Then, to  compute the  the force $\eta$ and magnitude   $u$,
 we formulate the following infinite-horizon optimal control problem defined by
 minimizing the performance index function defined by
 \begin{align}\label{cost-funct}
 & J_{\infty}(u,\eta :(y_0,c_0,0)) \coloneqq \frac{1}{2}\int^{\infty}_{0} |\nabla y(t,\Bigcdot)|^2_{L^2(\Omega,\bbR^d)}+ \beta |u(t)|^2\,\rmd t.
\end{align}
 That is, we define the infinite-horizon optimization problem
 \begin{subequations}\label{IHOP}
 \begin{align}
 \label{infinite-OP}
&\inf_{(u,\eta)\in L^2(\bbR_0,\mathbb{R})\times L^2_{\rm loc}(\bbR_0,\mathbb{R}^{d}) } J_{\infty}(u,\eta :(y_0,c_0,0))
\intertext{subject to}
 \label{sys-y-parab-Num}
 & \begin{cases}
  \dot y-\nu\Delta y+ay+b\cdot\nabla y=u\indf_{\omega(c)},\quad y\rest\Gamma=0,\\
  \ddot c+\varsigma\dot c+\epsilon c=\eta,\\
  y(0,\Bigcdot)=y_0, \quad c(0)=c_0,\quad \dot c(0)=0,
  \end{cases}
\intertext{as well as to the constraints}
 \label{state-constraints}
 &\begin{cases}
  c  \in  \mathcal{C}\coloneqq\{ g \in C(\overline{\mathbb{R}}_0,\mathbb{R}^d) \mid  \omega(g(s))\subset\Omega,   \text{ for }  s \in \overline\bbR_0 \} , \\
 \eta \in \mathcal{X}  \coloneqq \{ \kappa \in  L^2_{\rm loc}(\bbR_0,\mathbb{R}^d) \mid
 \dnorm{\kappa(s)}{} \leq K \mbox{ for a.e. }s\in\bbR_0\} ,
 \end{cases}
\end{align}
\end{subequations}
where  $K=(K_1,K_2,\dots,K_d)\in\bbR^d$ is a vector with coordinates~$K_i>0$, for all $1\le i\le d$, and where by~$\dnorm{\kappa(s)}{} \leq K$ we mean that
$
|\kappa_i(s)| \leq K_i\mbox{ for all } 1\le i\le d.
$
For tackling this infinite-horizon problem we employ a receding horizon framework.
This framework relies on successively solving finite-horizon open-loop
problems on bounded time-intervals as follows. Let us fix~$T>0$ and let an initial vector of the
form $\mathcal{I}_{0} \coloneqq (t_0,y_0,c_0,c^1_0)\in {\overline\bbR_0}\times L^2(\Omega)\times \mathbb{R}^d\times
\mathbb{R}^{d}$ be given. We define the
time
interval~$I_{t_0}\coloneqq (t_0,t_0+T)$, and the finite-horizon cost functional
\[
 J_{T}(u,\eta :\mathcal{I}_{0})\coloneqq
\frac{1}{2}\int^{t_0+T}_{t_0}  |\nabla y(t,\Bigcdot)|^2_{L^2(\Omega,\bbR^d)} + \beta |u(t)|^2\,\rmd t,
\]
and introduce the finite-horizon optimization problem
\begin{subequations}\label{FHOP}
 \begin{align}
\label{finite-OP}
&\min_{(u,\eta)\in L^2(I_{t_0},\mathbb{R}^{1+d}) } J_{T}(u,\eta :\mathcal{I}_{0})
\intertext{subjected to the dynamical constraints}
 \label{finte-cont-sys}
 &\begin{cases}
\dot y-\nu\Delta y+ay+b\cdot\nabla y =u \indf_{\omega(c)},\quad y\rest\Gamma=0,\\
\ddot c+\varsigma\dot c+\epsilon c =\eta, \\
y(t_0,\Bigcdot)=y_0,\quad c(t_0)=c_0,  \quad  \dot c(t_0) = c^1_0, &
\end{cases}\\
\intertext{in the time interval~$I_{t_0}$,\black as well as to the constraints}
\label{finte-cont-constraint}
&\begin{cases}
c \in\mathcal{C}_{t_0,T} := \{ g \in  C(\overline I_{t_0}, \mathbb{R}^d) \mid \omega(g(s))\subset\Omega\,\text{ for } s\in
\overline I_{t_0}\}\\
\eta \in \mathcal{X}_{t_0,T}\coloneqq \{ \kappa \in  L^2(I_{t_0}, \mathbb{R}^d) \mid
\dnorm{\kappa(s)}{}\black\leq K \mbox{ for a.e. }s\in I_{t_0}\}.
\end{cases}
\end{align}
\end{subequations}

The steps of the RHC are described in Algorithm~\ref{RHA}, where
we use the subset
\[
 \bbR^d_{[\omega]}\coloneqq\{c\in\bbR^d\mid \omega(c)\in\Omega\}.
\]

\begin{algorithm}[htbp]
\caption{Receding Horizon Algorithm}\label{RHA}
\begin{algorithmic}[1]
\REQUIRE{The prediction horizon $T>0$, the sampling time $\delta<T$,
and an initial vector  $\mathcal{I}_{\infty}=(y_0,c_0) \in  H \times  \mathbb{R}^{d}_{[\omega]}$.}
\ENSURE{The suboptimal RHC pair~$(u_{rh},\eta_{rh})$.}
\STATE Set~$t_0=0$  and~$\mathcal{I}_{0}=(t_0,y_0,c_0,0)$;
\STATE Find the solution~$(y_T^*(\Bigcdot;\mathcal{I}_{0}) ,u^*_T(\Bigcdot;\mathcal{I}_{0}),
c^*_T(\Bigcdot;\mathcal{I}_{0}),{\eta}^*_T(\Bigcdot;\mathcal{I}_{0}))$
over the time horizon~$I_{t_0}$ by solving the  open-loop problem~\eqref{FHOP};
\STATE For all $\tau \in [t_0,t_0+\delta)$, set $u_{rh}(\tau)=u^*_T(\tau;\mathcal{I}_{0})$ and $\eta_{rh}(\tau)=\eta^*_T(\tau;\mathcal{I}_{0})$;
\STATE Update: $t_0 \leftarrow t_0 +\delta$;
\STATE Update: $\mathcal{I}_{0}\leftarrow (t_0,y^*_T(t_0;\mathcal{I}_{0}),c^*_T(t_0;\mathcal{I}_{0}),\dot c^*_T(t_0;\mathcal{I}_{0}))$;
\end{algorithmic}
\end{algorithm}

%%%%%%%%%%%%%%%%%%%%%%
\subsection{Numerical discretization and implementation}
Here we report on numerical experiments related
to Algorithm~\ref{RHA}. These experiments  confirm the capability of the moving control
computed by Algorithm~\ref{RHA}.  In all examples, we deal  with  one-dimensional
controlled systems of the form \eqref{sys-y-parab-Intro} defined on $\Omega :=(0,1)$
which are exponentially unstable without control. Moreover, we compare the performance
of  one single moving control with  finitely many static actuators.  Throughout, the spatial discretization was done by the
standard Galerkin method using  piecewise linear and continuous basis functions
with mesh-size $h = 0.0025$. Moreover, for temporal discretization  we used  the
Crank--Nicolson/Adams--Bashforth scheme \cite{He03} with step-size $t_{\rm step} =  0.001$.
In this scheme,   the  implicit Crank--Nicolson scheme is used except for the nonlinear term $u\indf_{\omega(c)}$ and convection term
$b \cdot \nabla y$ which are treated with the explicit Adams--Bashforth scheme.
To deal with open-loop problems ~\ref{finite-OP},
we considered the reduced formulation of the problem with respect
to the independent variables $(\eta,u)$. The state constraints $\mathcal{C}_{t_0,T}$
were treated using the  Moreau--Yosida \cite{HinterKunisch06} regularization
with parameter  $\mu  = 10^{-5}$. Moreover, the box constraints $|u(t)|\leq K$
were handled using projection. We used the projected Barzilai--Borwein gradient
method  \cite{AzmiKunisch20,DaiFlet05,BB} equipped with a nonmonotone  line search  strategy.
Further, we terminated the algorithm as the $L^2$-norm of the projected gradient for the
reduced problem was smaller than $10^{-4}$ times of the norm of the projected gradient for initial iterate.

 For the case with static actuators, we choose the indicator functions $\indf_{\omega_i}$ with the placements
\begin{equation}
\omega_i := \left( \tfrac{1}{2M}(2i-1)- \tfrac{r}{2}, \tfrac{1}{2M}(2i-1)+ \tfrac{r}{2} \right) \quad \text{ for } i=1,\dots,M,
\end{equation}
where $r>0$ and integer $M \in \mathbb{N}$.
 This is motivated by the stabilizability results given in \cite[Thm. 4.4]{RodSturm20} and    \cite[Sect.  4.8.1]{KunRod18-cocv}.
 Further, for every $t \geq 0$, the moving actuator $\indf_{\omega(c(t))}$ is described by
 \begin{equation}\notag
 \omega(c(t)):=\left( c(t)- \tfrac{r}{2}, c(t)+ \tfrac{r}{2} \right).
 \end{equation}

 For all actuators, namely, moving and fixed ones, we chose $r = 0.04$. Thus the support of every actuator covers
 only four percent of the whole of domain.  In the case of the  static actuators, we employed the receding horizon
 framework given in  \cite[Alg. 1]{AzmiKunisch19} for the choice of  $|\cdot|_{*}=|\cdot|_{\ell_2}$  with control
 cost parameter  $\beta$.  In all numerical experiments, we chose $T = 1.25$ and $\delta =0.5$.
\begin{example}
\label{exp1}
In this example, we set (cf.~\eqref{sys-y-parab-Num}--\eqref{state-constraints})
\begin{align}
 \nu &=0.1,\quad &\varsigma &= 1,\quad& \epsilon &= 0,\notag\\
 a(t,x)& =-3-2|\sin(t+x)|,\quad& b(t,x) &=|\cos(t+x)|,\quad& K &=500.\notag
\end{align}
Further, we chose the initial conditions
\[
 y_0(x):= \sin(\pi x),\quad (c_0,c^1_0) :=(0.5,0).
\]
Figures~\ref{Fig3} and~\ref{Fig4} correspond to the choices  $\beta = 0.1$ and $\beta = 0.5$, respectively.
Figures~\ref{Fig3a} and~\ref{Fig4a} illustrate  the evolution of the $L^2(\Omega)$-norm for the states
corresponding to  uncontrolled system,  one single moving actuator, and fixed actuators ($M=1,\dots,5$).
The black dotted line in both figures   corresponds to the uncontrolled state. It shows  that the uncontrolled
state is exponentially unstable.  For both  cases $\beta = 0.1$ and $\beta = 0.5$, we can see  that
the moving control obtained by Algorithm~\ref{RHA} is stabilizing and its stabilization rate is smaller
than the one corresponding to one single
static actuator ($M =1$),  and  comparable to the cases $M =2,3,4$. Further, by comparing Figures~\ref{Fig3a}
and~\ref{Fig4a}, we can infer that  $\beta =0.1 $ leads to a faster stabilization compared to the case $\beta =0.5$.
\begin{figure}[htbp]
    \centering
    \subfigure[$L^2(\Omega)$-norm for states]
    {
    \label{Fig3a}
       \includegraphics[height=.35\textwidth,width=.45\textwidth]{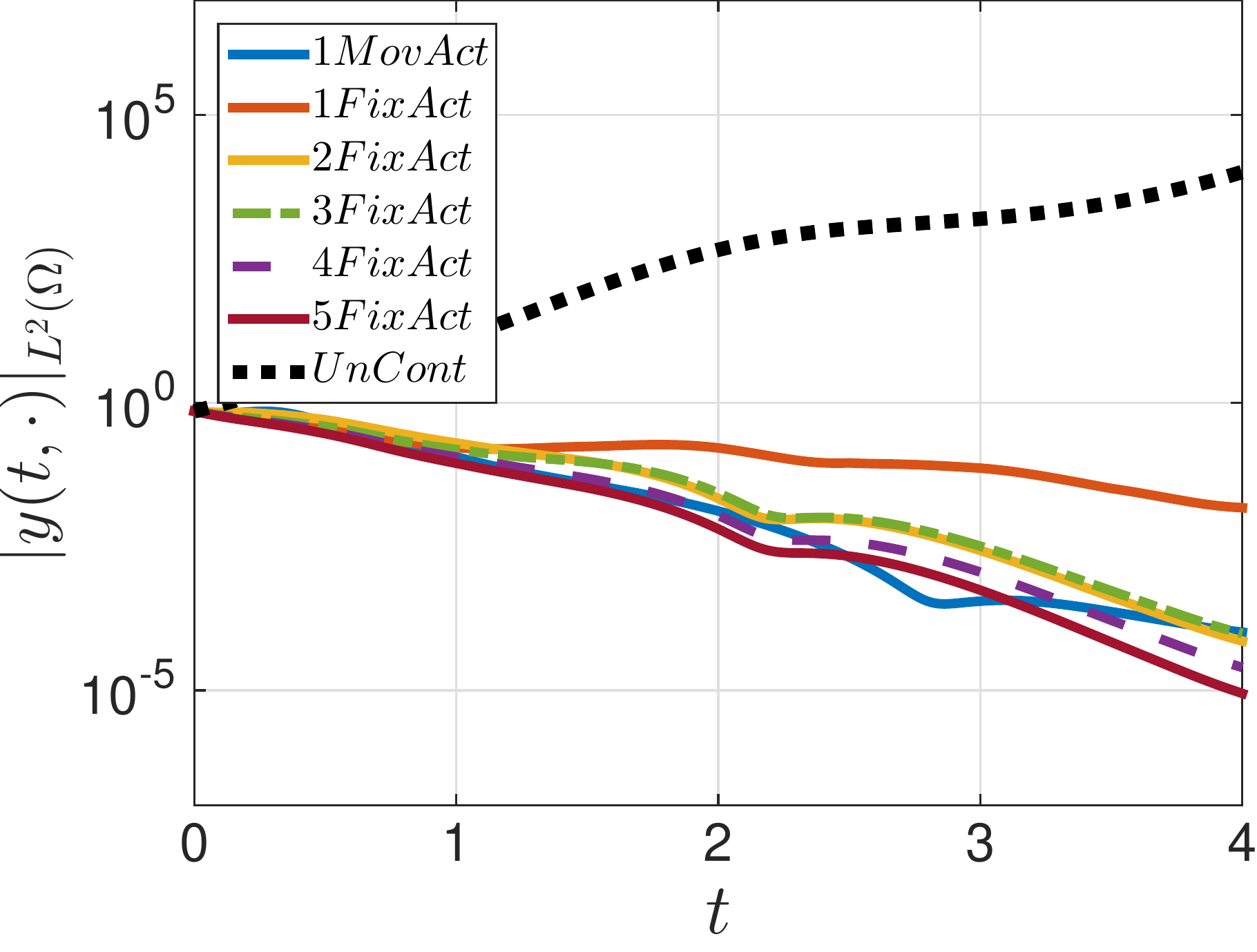}
    }
    \subfigure[ Control domain movements]
    {
    \label{Fig3b}
        \includegraphics[height=.35\textwidth,width=.45\textwidth]{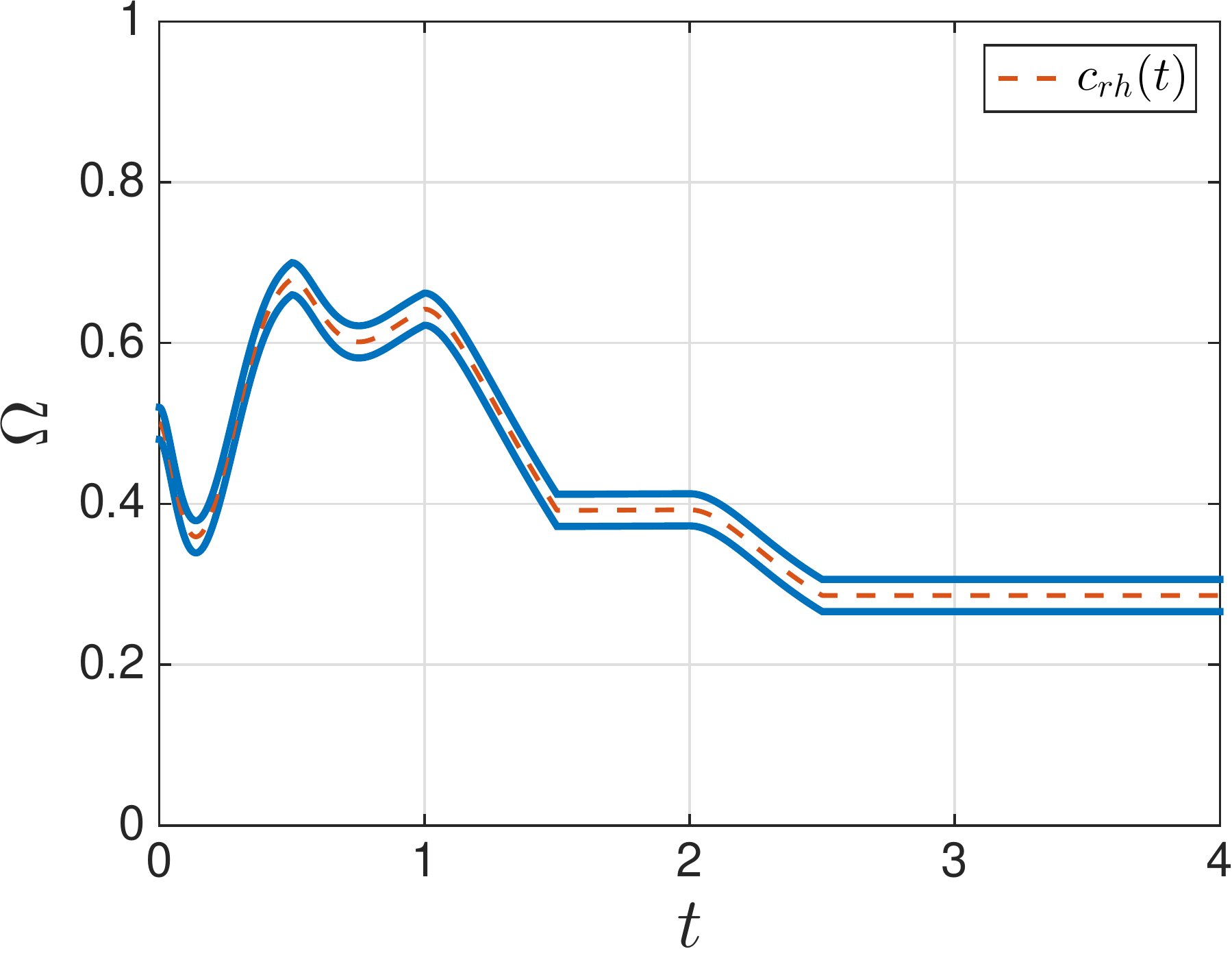}
    }
    \subfigure[ Absolute value of magnitude ]
    {
        \label{Fig3c}
        \includegraphics[height=.35\textwidth,width=.45\textwidth]{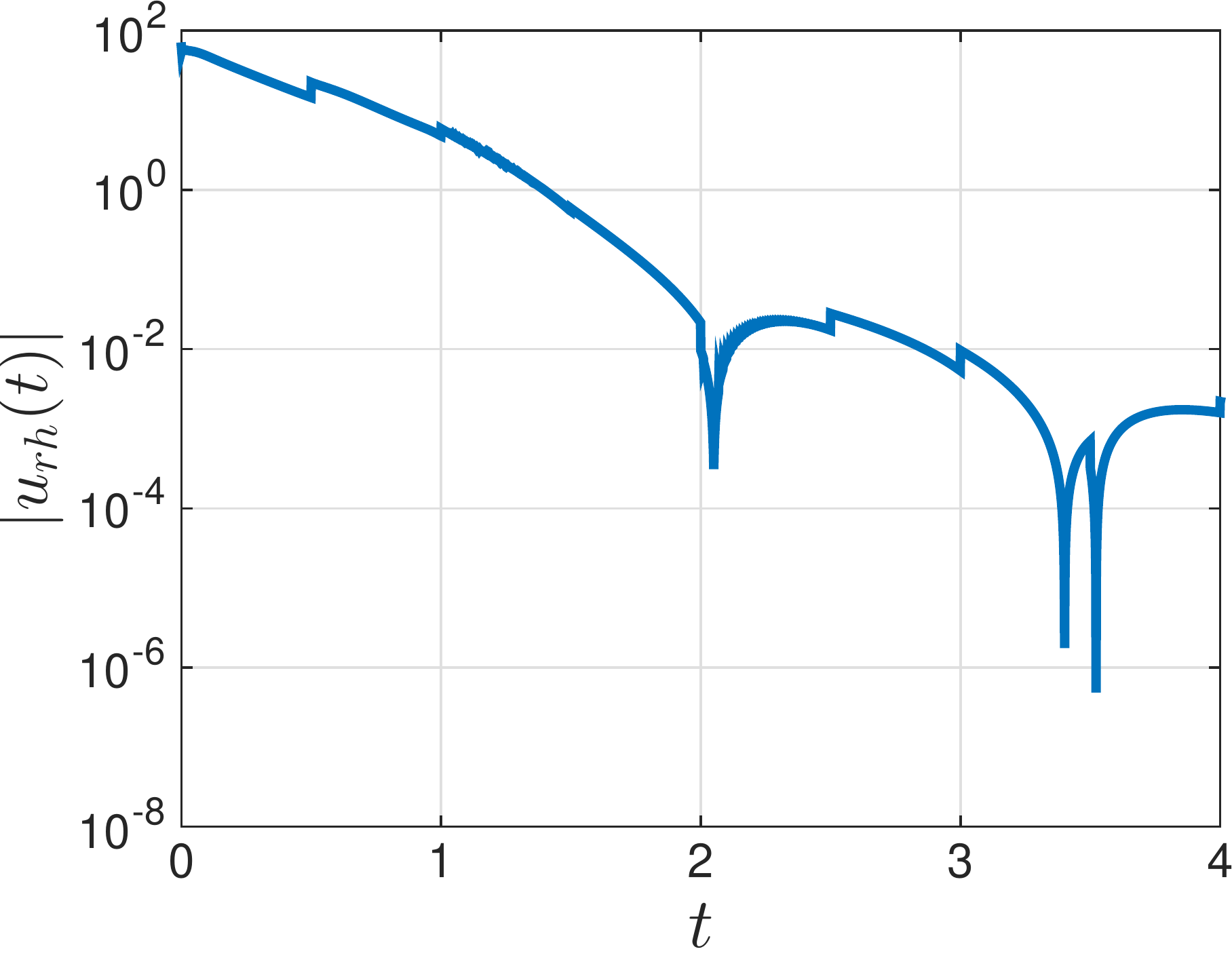}
    }
     \subfigure[Force]
    {
    \label{Fig3d}
        \includegraphics[height=.35\textwidth,width=.45\textwidth]{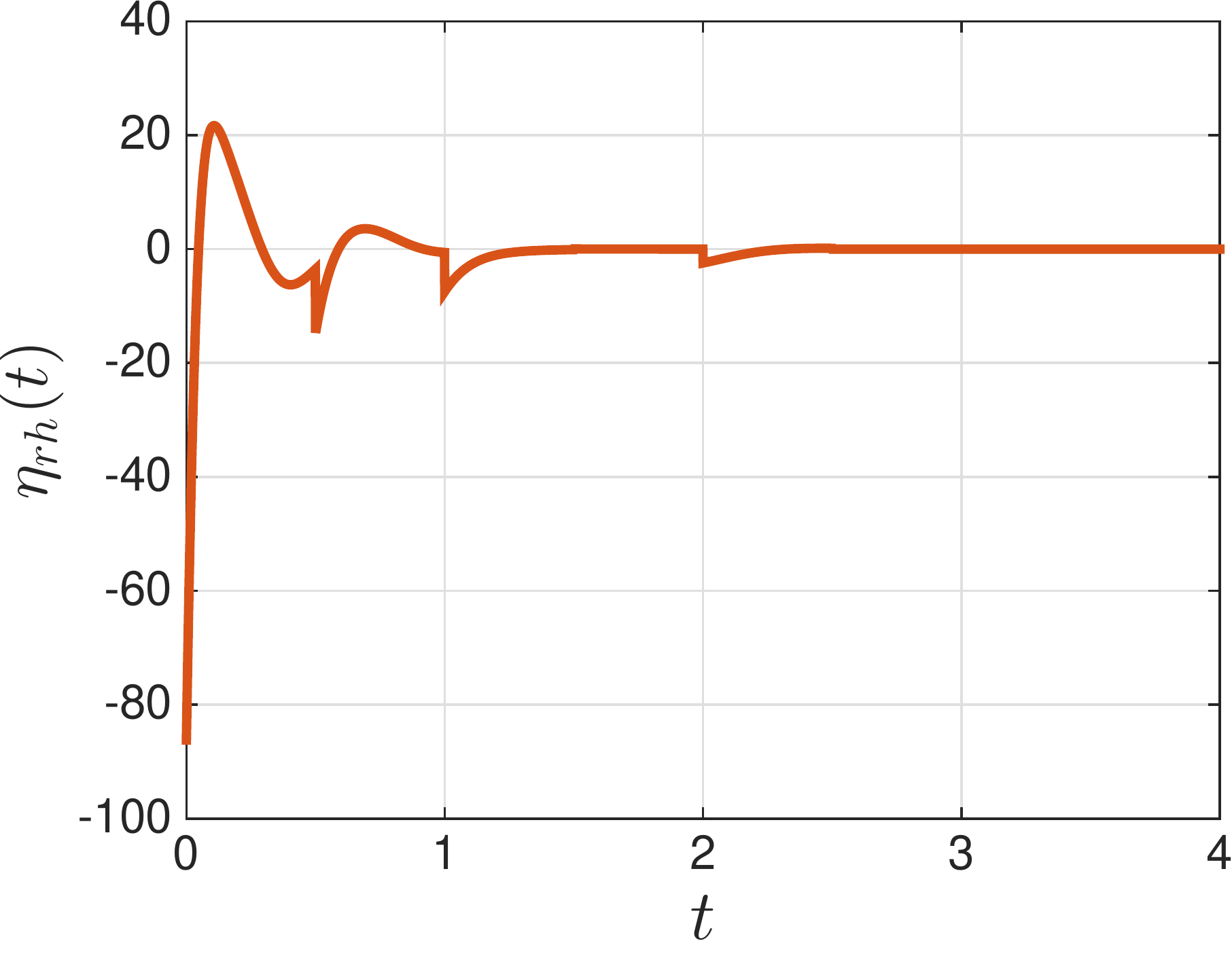}
    }
     \caption{Example~\ref{exp1}: Numerical results for $\beta = 0.1$.}
    \label{Fig3}
\end{figure}
 As can be seen from Figure~\ref{Fig4a}, it is not clear for  the case $\beta = 0.5$ that one fixed actuator is asymptotically stabilizing.
 Moreover, for $M=4,5$ we have better stabilization results compared to the single moving control.  Figures~\ref{Fig3b} and~\ref{Fig4b}
 illustrate the time  evolution of the control domain $\omega(c_{rh}(t) )$.
From  Figures~\ref{Fig3b},  we can observe that,  at some point ($t=2.5$), the actuator stops moving.
This corresponds to Figure $\ref{Fig3d}$, which demonstrates the evolution of the force.  In this case, the
receding horizon framework moved the actuator until some degree of
stabilization $(|y_{rh}(t,\Bigcdot)|_{L^2(\Omega)}\leq 10^{-3}) $ was reached and,
then, decided to steer the system with only a fixed actuator.  In this case, the $|y_{rh}(t,\Bigcdot)|_{L^2(\Omega)}$-norm
corresponding to $M= 4$ and $M=5$ is smaller than the one corresponding to  the single actuator which is free to move,
once $t\geq 3$, see Figure~\ref{Fig3a}.
For the case $\beta = 0.5$, we have a different scenario. In this case, the control
remains moving  throughout the whole  simulation (see Figure~\ref{Fig4b}). This fact
can also be seen from Figure ~\ref{Fig4d}  which shows that, here a stronger force was needed compared
to the case $\beta = 0.1$. Figures~\ref{Fig3c} and ~\ref{Fig4c} show the evolution of the absolute value of the
magnitude control $u_{rh}$.
\begin{figure}[htbp]
    \centering
    \subfigure[$L^2(\Omega)$-norm for states]
    {
    \label{Fig4a}
       \includegraphics[height=.35\textwidth,width=.45\textwidth]{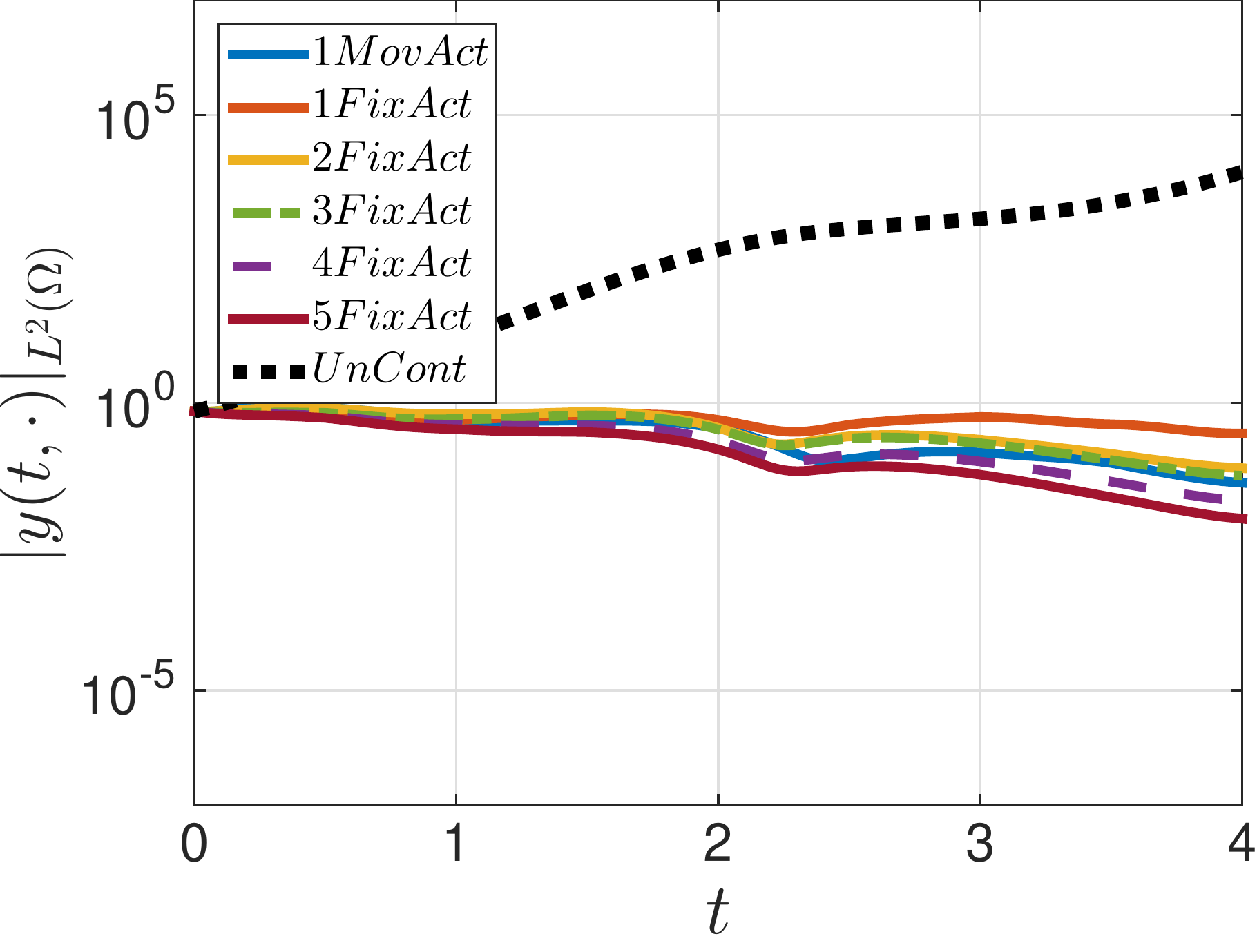}
    }
    \subfigure[Control domain movements]
    {
        \label{Fig4b}
        \includegraphics[height=.35\textwidth,width=.45\textwidth]{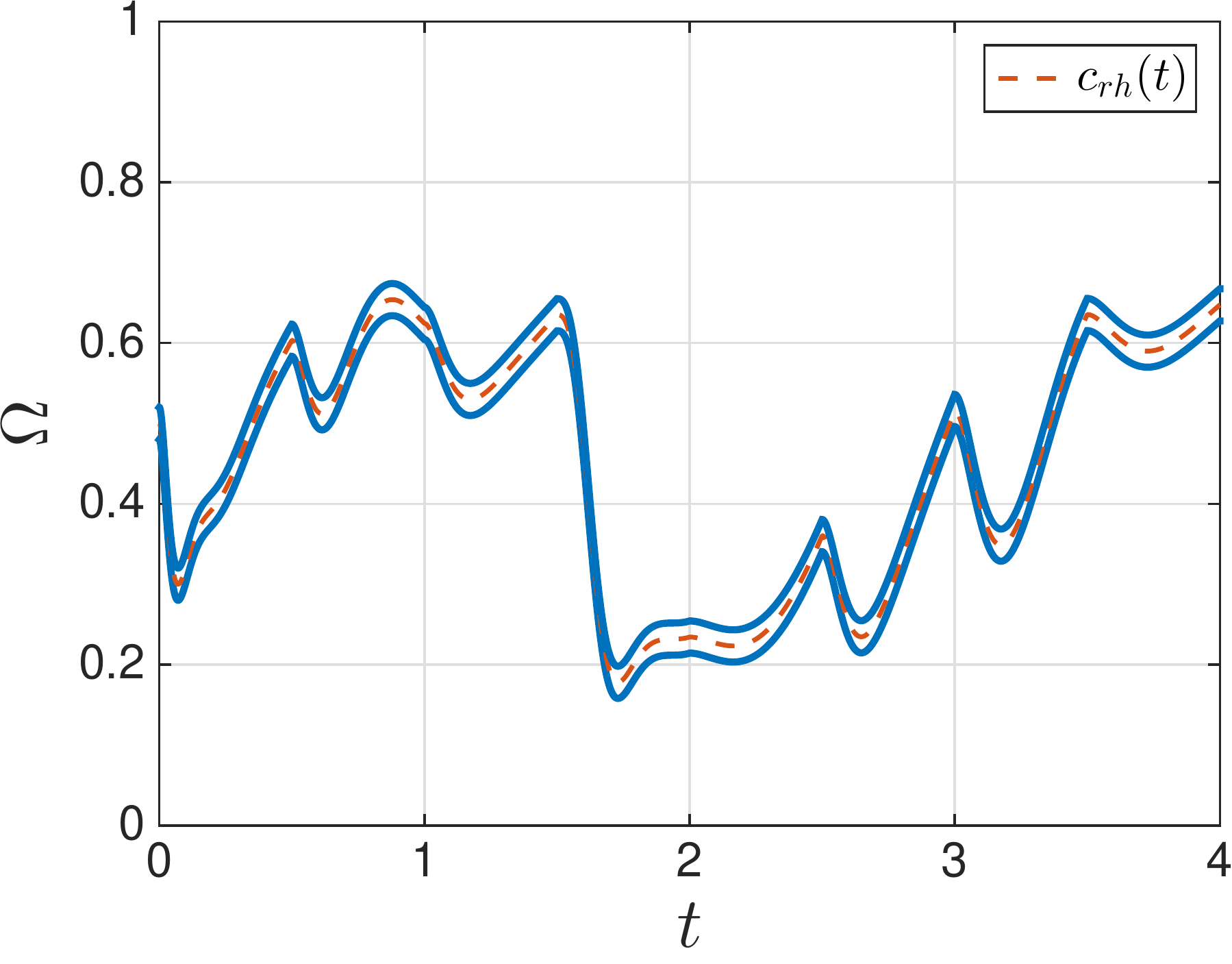}
    }
    \subfigure[Absolute value of magnitude ]
    {
       \label{Fig4c}
        \includegraphics[height=.35\textwidth,width=.45\textwidth]{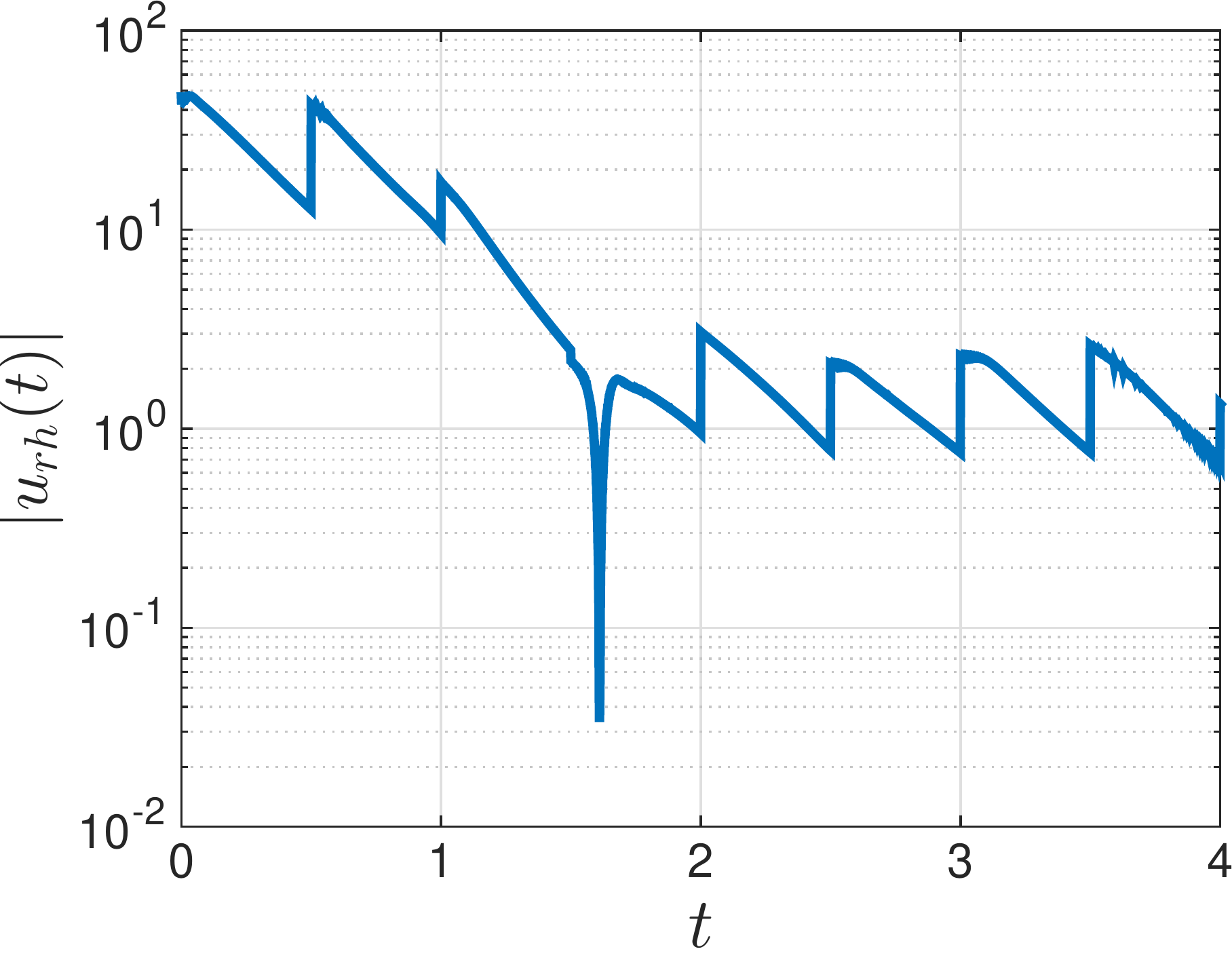}
    }
     \subfigure[ Force]
    {
     \label{Fig4d}
\includegraphics[height=.35\textwidth,width=.45\textwidth]{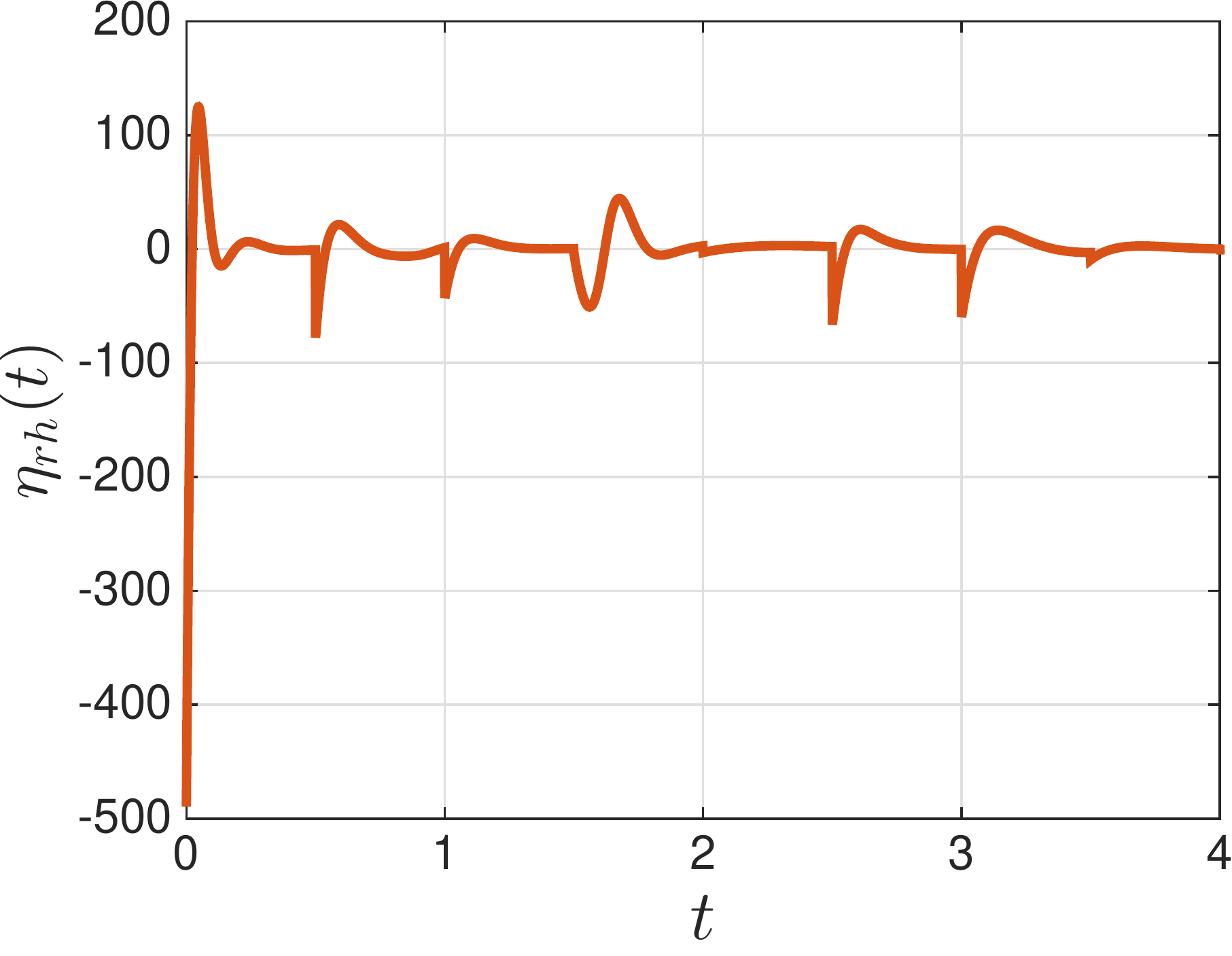}
    }
     \caption{Example~\ref{exp1}: Numerical results for $\beta = 0.5$ }
    \label{Fig4}
\end{figure}

\end{example}

\begin{example}
\label{exp2}
Note that Proposition~\ref{P:NOTstabil1Act} implies says that if $y_0$ is an eigenfunction associated to a negative eigenvalue of~$\clA=-\nu\Delta-5\Id$ and orthogonal to the fixed actuator, then the solution cannot be stabilized to zero. 
We present an example illustrating such a situation.
A moving control
steers, however, the system to zero. Here we used the same setting  as in the previous  example, except
that we put
\begin{subequations}\label{onestatic-fail}
\begin{align}
 a(t,x) &=-5,\quad &b(t,x) &=0,\quad& y_0(x) &= \sin( 2\pi x).
\end{align}

Finally, we take a fixed actuator~$\indf_\omega$ centered at~$0.5$,
\begin{equation}
\omega=(-\tfrac{r}2,\tfrac{r}2).
\end{equation}
\end{subequations}
Clearly we have that~$(y_0,\indf_\omega)_{L^2(\Omega)}=0$, and that
$ y_0$ is an eigenfunction of~$-\nu\Delta-5\Id$ with eigenvalue~$4\pi^2\nu-5<0$. Thus we cannot stabilize the corresponding solution to zero.
This is confirmed
in Figure~\ref{Fig5a}, where we can see that the single fixed actuator
is not able to stabilize the system.

Furthermore, 
we can see that the curves corresponding to the uncontrolled state and one single fixed actuator
are overlapping each other completely.
This is not surprising, because  $u^*=0$ is necessarily  the unique minimum for all finite horizon open-loop problems with $y_0\in\bbR\sin( 2\pi x)$. This is due to the fact that~$(y_0,\indf_\omega)_{L^2(\Omega)}=0$.

On the other hand, we can see from Figures~\ref{Fig5b} and~\ref{Fig5a}
that a single moving control is able to steer the system exponentially to zero by moving the actuator.
Figures~\ref{Fig5c} and~\ref{Fig5d} depict the evolution of  the absolute value of the magnitude
$u_{rh}$ and the evolution of the force $\eta_{rh}$, respectively.
\begin{figure}[htbp]
    \centering
    \subfigure[ $L^2(\Omega)$-norm for states]
    {
    \label{Fig5a}
       \includegraphics[height=.35\textwidth,width=.45\textwidth]{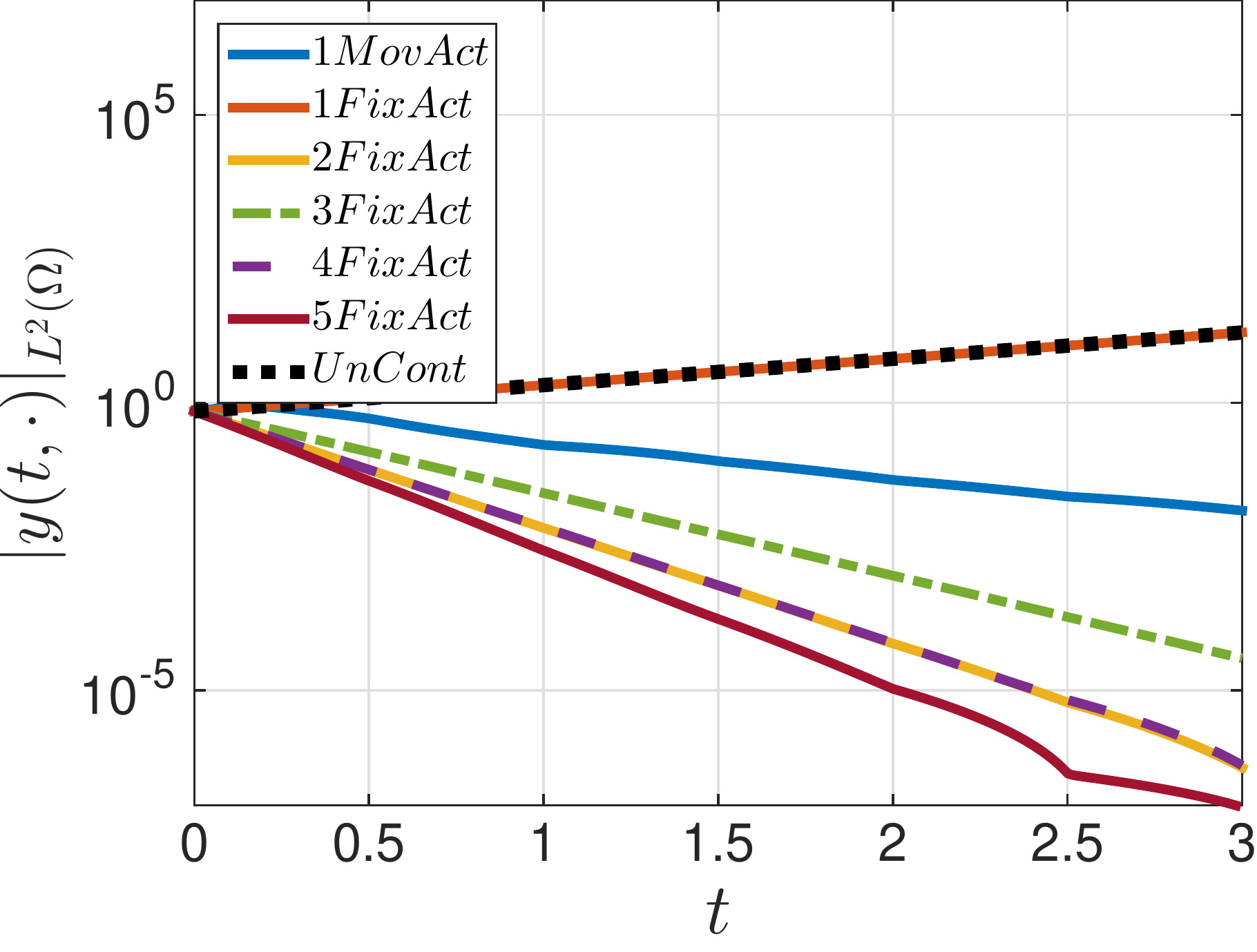}
    }
    \subfigure[Control domain movements]
    {
     \label{Fig5b}
        \includegraphics[height=.35\textwidth,width=.45\textwidth]{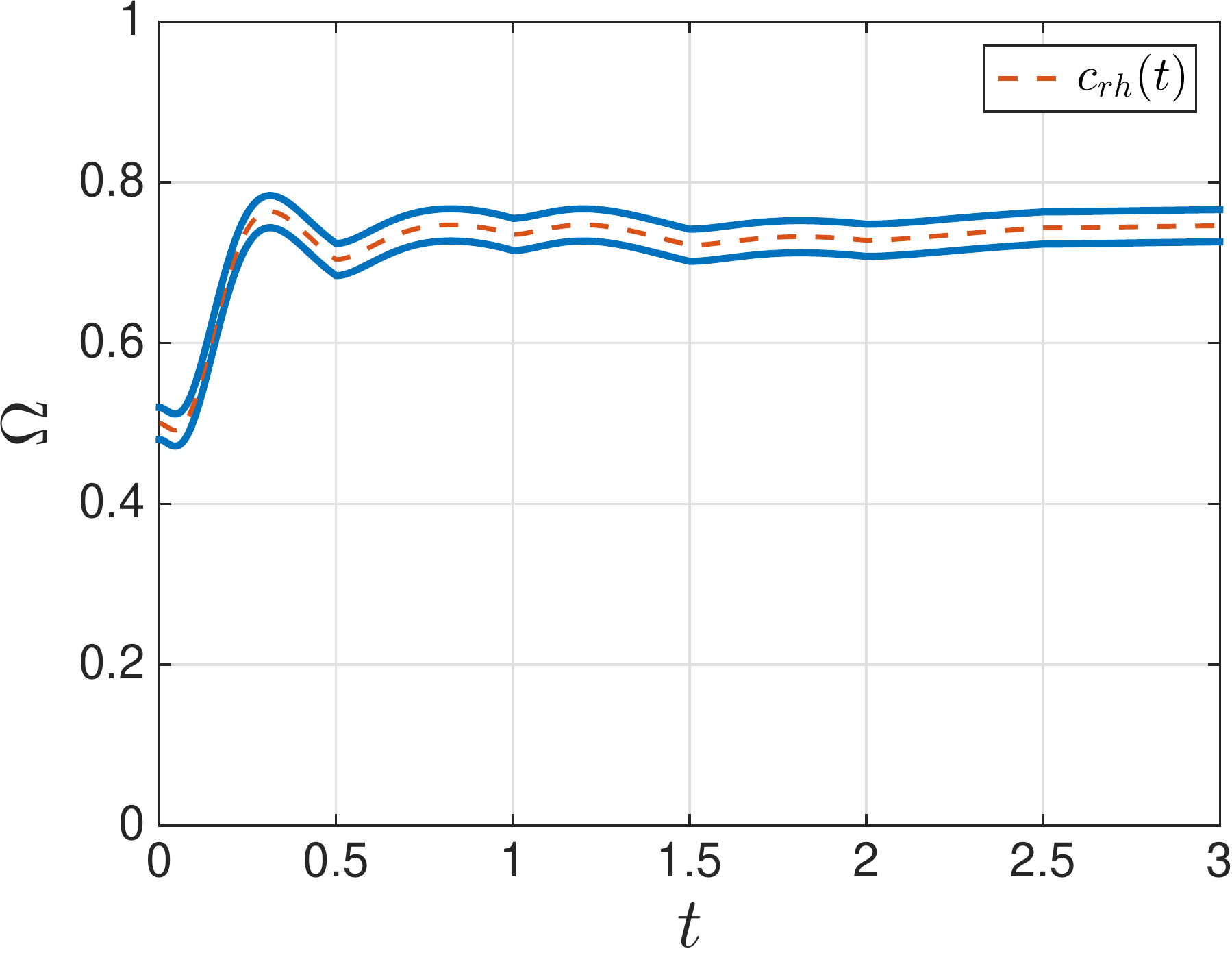}
    }
    \subfigure[Absolute value of magnitude ]
    {
     \label{Fig5c}
        \includegraphics[height=.35\textwidth,width=.45\textwidth]{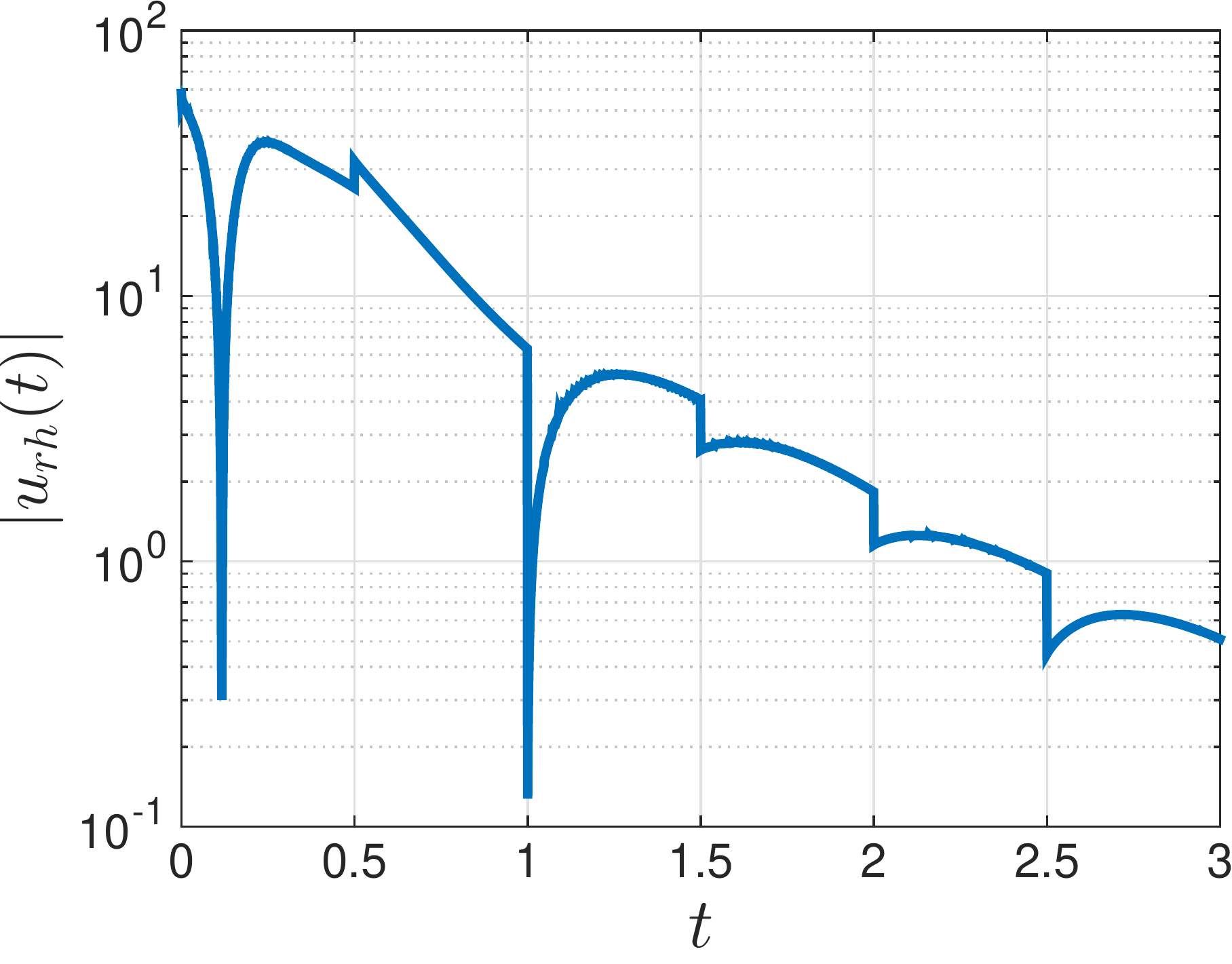}
    }
     \subfigure[Force]
    {
     \label{Fig5d}
        \includegraphics[height=.35\textwidth,width=.45\textwidth]{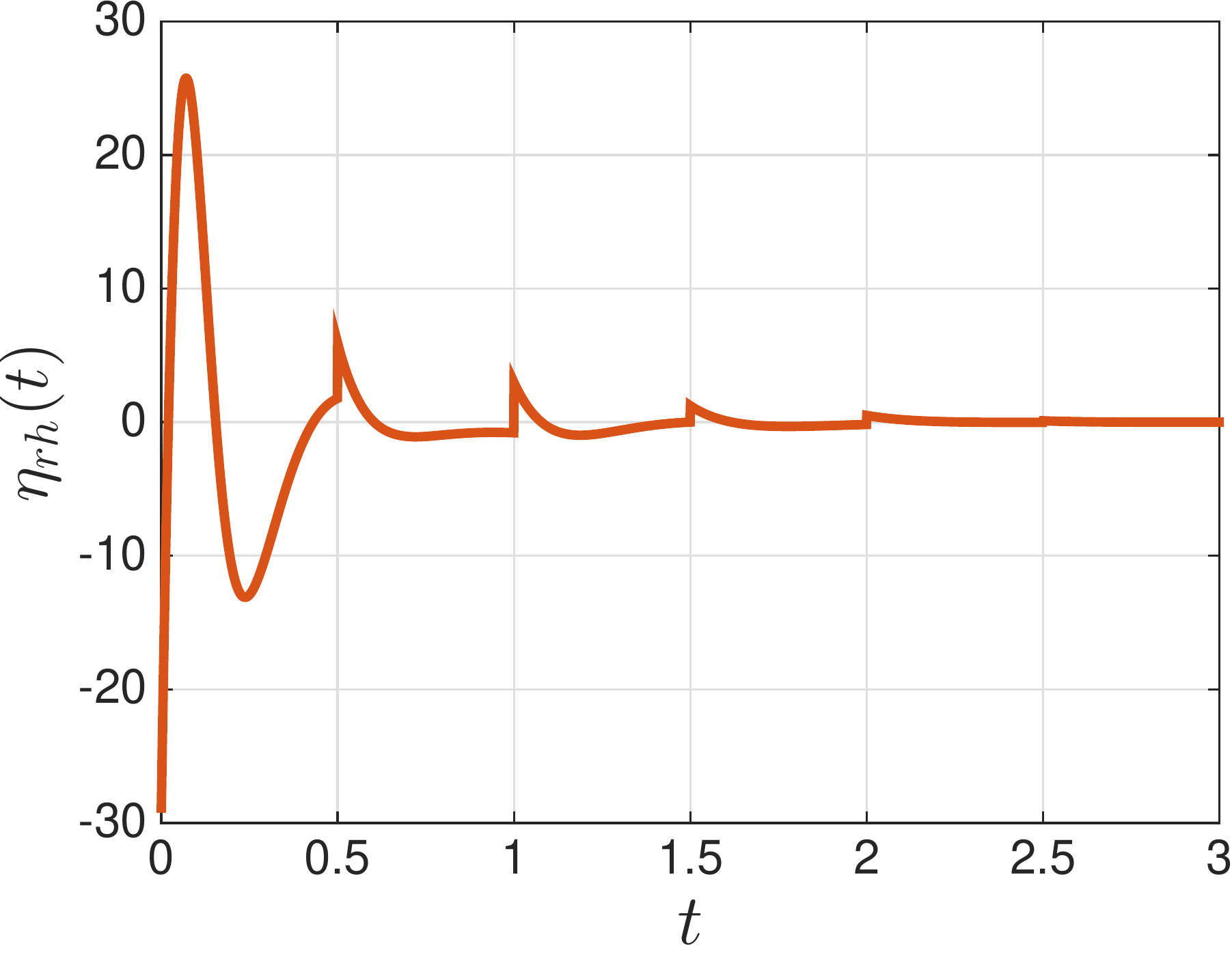}
    }
     \caption{Example~\ref{exp2}: Numerical results for $\beta = 0.01$ }
    \label{Fig5}
\end{figure}

\begin{remark}
Note that, though the setting~\eqref{onestatic-fail} serves to illustrate the orthogonality setting $(y_0,\indf_\omega)_{L^2(\Omega)}$ in Proposition~\ref{P:NOTstabil1Act},  we can show that, in our 1D setting there are choices of~$\omega$ which are able to stabilize the system. Indeed, all the eigenvalues are simple and for an actuator~$\indf_\omega$ with~$\omega=(a,b)\subseteq(0,1)$ 
we find that
\[
(\sin( j\pi x),\indf_\omega)_{L^2(\Omega)}=\tfrac{1}{j\pi}(\cos( j\pi a)-\cos( j\pi b)),\qquad j\in\bbN_0.
\]
Hence, since
\[
(\sin( j\pi x),\indf_\omega)_{L^2(\Omega)}=0\Longleftrightarrow j\pi a\pm j\pi b \in\pm2\pi\bbN\Longleftrightarrow j( a\pm  b) \in\pm2\bbN,
\]
we have that
\[
(\sin( j\pi x),\indf_\omega)_{L^2(\Omega)}\ne0,\quad\mbox{for all}\quad j\in\bbN_0,\quad\mbox{if}\quad \{a+ b,a-b\}\subset\bbR\setminus\bbQ,
\]
where~$\bbQ$ stands for the set of rational numbers. Therefore, by Proposition~\ref{P:stabil1Act},  we can stabilize the system if we replace~$\omega$ in~\eqref{onestatic-fail} by, for example,~$\widetilde\omega=(-\frac{r}2-\varrho_1,\frac{r}2+\varrho_2)\subset(0,1)$ with~$\{\varrho_2-\varrho_1,r+\varrho_2+\varrho_1\}\subset\bbR\setminus\bbQ$.
In particular taking small~$\varrho_1\ge0$ and small~$\varrho_2\ge0$ we see that with suitable arbitrary small perturbations~$\widetilde\omega$ of~$\omega$, the actuator~$\indf_{\widetilde\omega}$ allows us stabilize the system.

Finally, notice that in 2D (which we do not consider here), for the unit square as spatial domain~$\Omega$, the eigenvalues are not all simple and it will be always possible to find examples where a fixed actuator is not enough for stabilization (again, due to Proposition~\ref{P:NOTstabil1Act}).
\end{remark}
\end{example}

\subsection{Remarks on the computation of the moving actuator}
Summarizing,  we can assert that the single moving actuator obtained by Algorithm~\ref{RHA} is
able to stabilize the system to zero, confirming our theoretical findings.

For simplicity, and since the paper is already relatively long, we have restricted the numerical results to the 1D case.

The discretization of the optimal control problems in the 2D case is also more involved (due essentially to the adjoint equations). The 2D case will be addressed in a future work, where we plan to include also more details concerning the numerical realization. 
In particular, we plan to investigate the asymptotic stability of RHC and well-posedness of the associated finite-horizon open-loop subproblems.

It could also be of interest to investigate is the dependence on the stabilizability in the forcing bound parameter~$K$. Here in our simulations we have chosen~$K=500$ rather large, and the bound is never reached. It would be interesting to observe what happen when the bound is reached/active.

\appendix
\gdef\thesection{\Alph{section}}
\section*{Appendix}\normalsize
%\section*{}
%\addcontentsline{toc}{section}{Appendix}
% \begin{center}
% %{\sc\Large --- Appendix: proofs of auxiliary results ---}
% {\sc\Large --- Appendix ---}
% \end{center}
%to have the effect of a new section we have to reset the counters
\setcounter{section}{1}%to restart with A (~1 in Alph)
\setcounter{theorem}{0} \setcounter{equation}{0}
\numberwithin{equation}{section}

\subsection{Proofs of Propositions~\ref{P:NOTstabil1Act} and~\ref{P:stabil1Act}}\label{Apx:proofP:NOTYESstabil1Act}
Recall system~\eqref{sys-y-axy},
\begin{subequations}\label{sys-y-axy.Apx}
 \begin{align}
 &\tfrac{\p}{\p t} y(t,x)-\nu\Delta y(t,x)+a(x)y(t,x)=u(t)\Psi,\\
 &y(0,\Bigcdot)=y_0,\qquad \clG y\rest\Gamma=0,
 \end{align}
 \end{subequations}
and the system of eigenfunctions~$\widetilde e_i$ and increasing sequence of
eigenvalues~$\widetilde \alpha_i$ of the operator~$\clA=-\nu\Delta +a(x)\Id$,
$\clA\widetilde e_i=\widetilde \alpha_i\widetilde e_i$.

We start with the following auxiliary result.
 \begin{lemma}\label{L:act-eigf}
If there exists a nonsimple eigenvalue~$\widetilde \alpha_j$, then there exists
one associated eigenfunction~$\overline e_j$
such that  $(\overline e_j,\Psi)_{L^2(\Omega}=0$.
\end{lemma}
\begin{proof}
If~$\widetilde \alpha_j$ is a nonsimple eigenvalue, we can assume
that~$\widetilde \alpha_j=\widetilde \alpha_{j+1}$. Then,
in case~$(\widetilde e_j,\Psi)_{L^2(\Omega)}=0$
or~$(\widetilde e_{j+1},\Psi)_{L^2(\Omega)}=0$ the proof is finished.
It remains to consider the case~$(\widetilde e_j,\Psi)_{L^2(\Omega}\eqqcolon\beta_j\ne0
\ne\beta_{j+1}\coloneqq(\widetilde e_{j+1},\Psi)_{L^2(\Omega}$. In this case we simply
take the eigenfunction~$\overline e_j\coloneqq \beta_{j+1}\widetilde e_j-\beta_j\widetilde e_{j+1}$ which satisfies
$(\overline e_j,\Psi)_{L^2(\Omega)}=\beta_{j+1}\beta_j-\beta_j\beta_{j+1}=0$.
\end{proof}

\begin{proof}[Proof of Proposition~\ref{P:NOTstabil1Act}]
 We take the eigenfunction given by Lemma~\ref{L:act-eigf} as initial condition, $y_0\coloneqq\overline e_j$,
 $(\overline e_j,\Psi)_{L^2(\Omega}=0$.
By decomposing the solution into orthogonal components~$y=q+Q$,
 with~$q\in\linspan\{\overline e_j\}$ and~$Q\in\{\overline e_j\}^\perp$, we obtain
 \begin{subequations}\label{sys-y-axy-decfail}
 \begin{align}
 &\tfrac{\p}{\p t} q(t,x)+\clA q(t,x)=0,&&q(0)=\overline e_j,\\
 &\tfrac{\p}{\p t} Q(t,x)+\clA Q(t,x)=u(t)\Psi,&&Q(0)=0,\\
 &\clG q\rest\Gamma=0=\clG Q\rest\Gamma.
 \end{align}
 \end{subequations}
  Observe that the dynamics of the component~$q$ is independent of~$u$, and such component is then given by
 $q(t,\Bigcdot)=\rme^{-\widetilde \alpha_j t}\overline e_j$, $t>0$.
 Now, for the norm of the entire state we obtain
 \[
  \norm{y(t,\Bigcdot)}{L^2(\Omega)}^2=\norm{q(t,\Bigcdot)}{L^2(\Omega)}^2+\norm{Q(t,\Bigcdot)}{L^2(\Omega)}^2\ge
 \rme^{-2\widetilde \alpha_jt} \norm{\overline e_j}{L^2(\Omega)}^2,\qquad t>0,
 \]
for every magnitude control~$u$. 
\end{proof}

\begin{proof}[Proof of Proposition~\ref{P:stabil1Act}]
Let $j_0\coloneqq\min\{j\in\bbN\mid \widetilde \alpha_j>0\}$. If~$j_0=1$ then
 the free dynamics is exponentially stable.
 If $j_0>1$, then we consider the dynamics onto the linear span of the
 first~$j_0-1$ eigenfunctions
 \[
  \clE_{j_0-1}\coloneqq\linspan\{\widetilde e_j\mid 1\le j\le j_0-1\},\qquad q(t)
  \coloneqq P_{\clE_{j_0-1}}y(t)\in\clE_{j_0-1}
 \]
where~$P_{\clE_{j_0-1}}\in\clL(L^2(\Omega),\clE_{j_0-1})$ denotes the
orthogonal projection in~$L^2(\Omega)$ onto the subspace~$\clE_{j_0-1}$.
We decompose the system as
 \begin{subequations}\label{sys-y-axy-decsucc}
 \begin{align}
 &\tfrac{\p}{\p t} q(t,x)+\clA q(t,x)
 =u(t)P_{\clE_{j_0-1}}\Psi,&&q(0)=P_{\clE_{j_0-1}}y_0,\label{sys-y-axy-decsucc-q}\\
 &\tfrac{\p}{\p t} Q(t,x)+\clA Q(t,x)
 =u(t)P_{\clE_{j_0-1}^\perp}\Psi,&&Q(0)=P_{\clE_{j_0-1}^\perp}y_0,\\
 &\clG q\rest\Gamma=0=\clG Q\rest\Gamma,
 \end{align}
 \end{subequations}
 with~$Q=y-q=P_{\clE_{j_0-1}^\perp}y$.
 Next we prove that the finite dimensional system~\eqref{sys-y-axy-decsucc-q} is null controllable. Writing
 \[
  q=\textstyle\sum\limits_{k=1}^{j_0-1}q_k\widetilde e_k,\qquad \overline q=\begin{bmatrix}
                                                                  q_1\\q_2\\\vdots\\q_{j_0-1}
                                                                 \end{bmatrix}
 \]
we obtain the system
\begin{equation}\label{sys-Kalman}
 \dot{\overline q}=A\overline q+Bu,
\end{equation}
with~$A\in\bbR^{(j_0-1)\times(j_0-1)}$ and~$B\in\bbR^{(j_0-1)\times1}$ as follows
\begin{align}
A={\rm diag}\left(\widetilde \alpha_1,\widetilde \alpha_2,\dots,\widetilde \alpha_{j_0-1}
                   \right)
 \quad\mbox{and}\quad
                   B=\begin{bmatrix}
                   (\widetilde e_1,\Psi)\\(\widetilde e_2,\Psi)\\\vdots\\(\widetilde e_{j_0-1},\Psi)
                   \end{bmatrix}.\notag
\end{align}
The matrix~$A$ is diagonal with entries~$A_{(i,i)}=\widetilde \alpha_i$.

For any given~$T>0$, we have that system~\eqref{sys-Kalman} is controllable at time~$T$. Indeed,
this follows from Kalman rank condition (see, e.g., \cite[Sect~1.3, Thm.~1.2]{Zabczyk92}), because we have that
\[
 \det(\begin{bmatrix}A\mid B \end{bmatrix})=\left(\bigtimes_{k=1}^{j_0-1}(\widetilde e_k,\Psi)\right)
 \det \clV
\]
where
\[\begin{bmatrix}A\mid B \end{bmatrix}\coloneqq\begin{bmatrix}B &AB& \dots& A^{j_0-2}B\end{bmatrix}\]
and
$\clV$ is the Vandermonde matrix whose entries are
\[
 \clV_{(i,j)}=\widetilde \alpha_i^{j-1}.
\]
Hence
\[\det(\begin{bmatrix}A\mid B \end{bmatrix})
=\left(\bigtimes_{k=1}^{j_0-1}(\widetilde e_k,\Psi)\right)
\left(\bigtimes_{1\le i<j\le j_0-1}(\widetilde \alpha_j-\widetilde \alpha_i)\right)\ne0.
\]

Therefore, we can choose a control~$u$ such that~$\overline q(T,\Bigcdot)=0$, which implies $q(T,\Bigcdot)=0$. Then, we take
the concatenated
control defined as:
$u^c(t)$ if~$t\in[0,T]$, and~$u^c(t)=0$ for $t>T$. For time~$t\ge T$ we have that~$q(t,\Bigcdot)=0$ and
\[
  \norm{y(t,\Bigcdot)}{L^2(\Omega)}^2=\norm{Q(t,\Bigcdot)}{L^2(\Omega)}^2=
 \rme^{-\widetilde \alpha_{j_0}(t-T)} \norm{Q(T,\Bigcdot)}{L^2(\Omega)},\qquad t\ge T.
 \]
 Since, by definition of~$j_0$, we have that~$\widetilde \alpha_{j_0}>0$, it follows that~$\norm{y(t,\Bigcdot)}{L^2(\Omega)}^2$
 converges exponentially to zero, as~$t\to+\infty$. That is, $u^c$ is a stabilizing (open-loop) control.
\end{proof}

 The proofs of Propositions~\ref{P:NOTstabil1Act} and~\ref{P:stabil1Act}
 above are quite intuitive for the considered concrete system. Shorter proofs would follow from applications of the Fattorini criterium for stability that we find for example in~\cite{BadTakah14}.
At this point we would like to refer the reader also to the work by Hautus in~\cite[Sect.~3]{Hautus70} for further comments on controllability and stabilizability of more general finite-dimensional autonomous systems (including both continuous-time and discrete-time cases). See also~\cite[Sect.~6.5]{TucsnakWeiss09} for a version of the Hautus test for the observability of linear  abstract  infinite-dimensional systems (here recall also the duality between observability and controllability).

\subsection{Proof of Proposition~\ref{P:denseBall}.}\label{sS:proofP:denseBall}
Let us fix an arbitrary~$\delta>0$ and let~$h\in \fkS_H$ be in the unit ball of~$H$.
Since~$X$ is dense in~$H$, we can choose
 $\overline h\in X\setminus\{0\}$ such that~$\norm{\overline h-h}{H}\le\frac{\delta}{2}$. Now,
 for~$\widetilde h\coloneqq\norm{\overline h}{H}^{-1}\overline h\in \fkS_H$ we find
 \begin{align}
 \norm{\widetilde h-h}{H}
 &\le\norm{\norm{\overline h}{H}^{-1}\overline h-\overline h}{H}+\norm{\overline h-h}{H}
 \le\norm{\norm{\overline h}{H}^{-1}-1}{\bbR}\norm{\overline h}{H}+\tfrac{\delta}{2}\notag\\
 &=\norm{\overline h}{H}^{-1}\norm{1-\norm{\overline h}{H}}{\bbR}\norm{\overline h}{H}+\tfrac{\delta}{2}
 =\norm{\norm{h}{H}-\norm{\overline h}{H}}{\bbR}+\tfrac{\delta}{2}
\le\norm{h-\overline h}{H}+\tfrac{\delta}{2}=\delta.\notag
 \end{align}
Hence we can conclude that~$X\bigcap \fkS_H$ is dense in~$\fkS_H$.
\qed

\subsection{Proof of Proposition~\ref{P:cont.tswitch-L2}}\label{sS:proofP:cont.tswitch-L2}
We start by defining
\[
 \overline t_j\coloneqq\max\{\tau_j,\sigma_j\}, \qquad \underline t_j\coloneqq\min\{\tau_j,\sigma_j\},\qquad 0\le j\le K,
\]
and by writing, with~$g\coloneqq f_\tau-f_\sigma$,
\begin{equation}\label{swiL2-dec}
\norm{f_\tau-f_\sigma}{L^2((a,b),X)}^2=\sum_{j=1}^{K}\int_{\overline t_{j-1}}^{\overline t_{j}} \norm{f_\tau(t)-f_\sigma(t)}{ X}^2\,\ed t
=\sum_{j=1}^{K}\int_{\overline t_{j-1}}^{\overline t_{j}} \norm{g(t)}{ X}^2\,\ed t.
\end{equation}

We proceed by Induction.  Firstly, we find that
\begin{align}
 \int_{a}^{\overline t_{1}}\norm{g(t)}{X}^2\,\ed t
&=\int_{\overline t_{0}}^{\overline t_{1}}\norm{g(t)}{X}^2\,\ed t
=\int_{\underline t_{1}}^{\overline t_{1}} \norm{g(t)}{X}^2\,\ed t
\le\clR\clX^2, \label{swiL2-ind-base}
\end{align}
where in the last inequality we used the fact that  $g(t)=f_\tau(t)-f_\sigma(t)=\phi_{1}-\phi_{1}=0$
for~$ t \in [\overline   t_{0},\underline t_{1})$.

Next, we assume that for a given~$i\in\{1,2,\dots,K-1\}$ we have
\begin{equation}\label{swiL2-ind-hypo}
\int_{a}^{\overline t_{i}}\norm{g(t)}{X}^2\,\ed t\le i\clR\clX^2. \tag{$\fkH$}
\end{equation}
Then we obtain
\begin{align}
 \int_{a}^{\overline t_{i+1}}\norm{g(t)}{X}^2\,\ed t
 &\le i\clR\clX^2
 +\int_{\overline t_{i}}^{\overline t_{i+1}}\norm{g(t)}{X}^2\,\ed t,\notag
 \end{align}
 which implies that
 \begin{align}
 \int_{a}^{\overline t_{i+1}}\norm{g(t)}{X}^2\,\ed t
 &\le i\clR\clX^2
 +\int_{\overline t_{i}}^{\underline t_{i+1}}\norm{g(t)}{X}^2\,\ed t
 +\int_{\underline t_{i+1}}^{\overline t_{i+1}}\norm{g(t)}{X}^2\,\ed t,\quad\mbox{if}\quad
 \overline t_{i}<\underline t_{i+1}\le\overline t_{i+1},\notag
 \intertext{and}
 \int_{a}^{\overline t_{i+1}}\norm{g(t)}{X}^2\,\ed t
 &\le i\clR\clX^2
  +\int_{\underline t_{i+1}}^{\overline t_{i+1}}\norm{g(t)}{X}^2\,\ed t,\quad\mbox{if}\quad
 \underline t_{i+1}\le\overline t_{i}\le\overline t_{i+1},\notag
 \end{align}

Observe that, if~$\overline t_{i}<\underline t_{i+1}$, then $g(t)=f_\tau(t)-f_\sigma(t)=\phi_{i+1}-\phi_{i+1}=0$,
for $t\in(\overline t_{i},\underline t_{i+1})\subseteq(\tau_i,\tau_{i+1})\bigcap(\sigma_i,\sigma_{i+1})$. Therefore, in either case we have
\begin{align}\label{swiL2-ind-thesis}
 \int_{a}^{\overline t_{i+1}}\norm{g(t)}{X}^2\,\ed t
 &\le i\clR\clX^2
  +\int_{\underline t_{i+1}}^{\overline t_{i+1}}\norm{g(t)}{X}^2\,\ed t\le i\clR\clX^2+\clR\clX^2=(i+1)\clR\clX^2.\tag{$\fkT$}
 \end{align}
Hence, assumption~\eqref{swiL2-ind-hypo} implies~\eqref{swiL2-ind-thesis}, which together with~\eqref{swiL2-ind-base} imply,
by Induction,
\begin{equation}\label{switch-L2-end}
\int_{a}^{\overline t_{j}}\norm{g(t)}{X}^2\,\ed t\le j\clR\clX^2,\quad\mbox{for all}\quad j\in\{1,2,\dots,K\}.
\end{equation}
In particular, since~$\overline t_{K}=b$, for~$j=K$ we obtain~$\norm{f_\tau-f_\sigma}{L^2((a,b),X)}\le K^\frac{1}{2}\clR^\frac{1}{2}\clX$.
\qed

\bigskip\noindent
\textbf{Acknowledgments.} K. Kunisch was supported in part by the ERC advanced grant 668998 (OCLOC) under the EU’s H2020 research program. S. Rodrigues gratefully acknowledges partial support from the Austrian Science Fund (FWF): P 33432-NBL.

\input{ActuatorMoving_arx.bbl}
% \bibliographystyle{plainurl}
% \bibliography{ActuatorMoving}

\end{document}

%% file: Mathcommands.tex
\newcommand{\linspan}{\mathop{\rm span}\nolimits}

\newcommand{\sign}{\mathop{\rm sign}\nolimits}

\newcommand{\rest}{\left.\kern-2\nulldelimiterspace\right|_}
\newcommand{\norm}[2]{\left|#1\right|_{#2}}
\newcommand{\dnorm}[2]{\left\|#1\right\|_{#2}}

\newcommand{\Id}{{\mathbf1}}
\newcommand{\indf}{1}

\newcommand{\ex}{\mathrm{e}}
\newcommand{\p}{\partial}
\newcommand{\ed}{\mathrm d}

\newcommand{\e}{\varepsilon}

\newcommand*{\Bigcdot}{\raisebox{-.25ex}{\scalebox{1.25}{$\cdot$}}}

\newcommand{\clA}{{\mathcal A}}

\newcommand{\clC}{{\mathcal C}}

\newcommand{\clE}{{\mathcal E}}

\newcommand{\clG}{{\mathcal G}}

\newcommand{\clI}{{\mathcal I}}

\newcommand{\clL}{{\mathcal L}}

\newcommand{\clR}{{\mathcal R}}
\newcommand{\clS}{{\mathcal S}}

\newcommand{\clV}{{\mathcal V}}

\newcommand{\clX}{{\mathcal X}}

\newcommand{\bbN}{{\mathbb N}}

\newcommand{\bbQ}{{\mathbb Q}}
\newcommand{\bbR}{{\mathbb R}}

\newcommand{\bbT}{{\mathbb T}}

\newcommand{\fkB}{{\mathfrak B}}

\newcommand{\fkD}{{\mathfrak D}}

\newcommand{\fkH}{{\mathfrak H}}

\newcommand{\fkK}{{\mathfrak K}}

\newcommand{\fkN}{{\mathfrak N}}

\newcommand{\fkR}{{\mathfrak R}}
\newcommand{\fkS}{{\mathfrak S}}
\newcommand{\fkT}{{\mathfrak T}}

\newcommand{\fkY}{{\mathfrak Y}}

\newcommand{\rmD}{{\mathrm D}}

\newcommand{\bfn}{{\mathbf n}}

\newcommand{\bft}{{\mathbf t}}

\newcommand{\bfv}{{\mathbf v}}

\newcommand{\rmd}{{\mathrm d}}
\newcommand{\rme}{{\mathrm e}}

\newcommand{\fkv}{{\mathfrak v}}

\newcommand{\ovlineC}[1]{\overline C_{\left[#1\right]}}

\definecolor{DarkBlue}{rgb}{0,0.08,0.45}
\definecolor{DarkRed}{rgb}{.65,0,0}
\definecolor{applegreen}{rgb}{0.55, 0.71, 0.0}

\newcounter{mymac@matlab}
\newcommand{\matlab}{MATLAB%
   \ifnum\value{mymac@matlab}<1%
   \textregistered%
   \setcounter{mymac@matlab}{1}%
   \fi%
  }

\newcommand{\black}{ \color{black} }

%% file: ActuatorMoving_arx.bbl
\begin{thebibliography}{10}

\bibitem{AgraSary05}
A.~A. Agrachev and A.~V. Sarychev.
\newblock {N}avier--{S}tokes equations: Controllability by means of low modes
  forcing.
\newblock {\em J. Math. Fluid Mech.}, 7(1):108--152, 2005.
\newblock \href {http://dx.doi.org/10.1007/s00021-004-0110-1}
  {\path{doi:10.1007/s00021-004-0110-1}}.

\bibitem{AgraSary06}
A.~A. Agrachev and A.~V. Sarychev.
\newblock Controllability of 2{D} {E}uler and {N}avier--{S}tokes equations by
  degenerate forcing.
\newblock {\em Commun. Math. Phys.}, 265(3):673--697, 2006.
\newblock \href {http://dx.doi.org/10.1007/s00220-006-0002-8}
  {\path{doi:10.1007/s00220-006-0002-8}}.

\bibitem{AzmiKunisch19}
B.~Azmi and K.~Kunisch.
\newblock A hybrid finite-dimensional {RHC} for stabilization of time-varying
  parabolic equations.
\newblock {\em SIAM J. Control Optim.}, 57(5):3496--3526, 2019.
\newblock \href {http://dx.doi.org/10.1137/19M1239787}
  {\path{doi:10.1137/19M1239787}}.

\bibitem{AzmiKunisch20}
B.~Azmi and K.~Kunisch.
\newblock Analysis of the {B}arzilai-{B}orwein step-sizes for problems in
  {H}ilbert spaces.
\newblock {\em J. Optim. Theory Appl.}, 185(3):819--844, 2020.
\newblock \href {http://dx.doi.org/10.1007/s10957-020-01677-y}
  {\path{doi:10.1007/s10957-020-01677-y}}.

\bibitem{Badra09-cocv}
M.~Badra.
\newblock Feedback stabilization of the {2-D} and {3-D} {N}avier--{S}tokes
  equations based on an extended system.
\newblock {\em ESAIM Control Optim. Calc. Var.}, 15(4):934--968, 2009.
\newblock \href {http://dx.doi.org/10.1051/cocv:2008059}
  {\path{doi:10.1051/cocv:2008059}}.

\bibitem{BadTakah14}
M.~Badra and T.~Takahashi.
\newblock On the {F}attorini criterion for approximate controllability and
  stabilizability of parabolic systems.
\newblock {\em ESAIM Control Optim. Calc. Var.}, 20:924--956, 2014.
\newblock \href {http://dx.doi.org/10.1051/cocv/2014002}
  {\path{doi:10.1051/cocv/2014002}}.

\bibitem{BB}
J.~Barzilai and J.~M. Borwein.
\newblock Two-point step size gradient methods.
\newblock {\em IMA J. Numer. Anal.}, 8(1):141--148, 1988.
\newblock \href {http://dx.doi.org/10.1093/imanum/8.1.141}
  {\path{doi:10.1093/imanum/8.1.141}}.

\bibitem{BreKunRod17}
T.~Breiten, K.~Kunisch, and S.~S. Rodrigues.
\newblock Feedback stabilization to nonstationary solutions of a class of
  reaction diffusion equations of {F}itz{H}ugh--{N}agumo type.
\newblock {\em SIAM J. Control Optim.}, 55(4):2684--2713, 2017.
\newblock \href {http://dx.doi.org/10.1137/15M1038165}
  {\path{doi:10.1137/15M1038165}}.

\bibitem{CasCinMun14}
C.~Castro, N.~C\^{\i}ndea, and A.~M\"{u}nch.
\newblock Controllability of the linear one-dimensional wave equation with
  inner moving forces.
\newblock {\em SIAM J. Control Optim.}, 52(6):4027--4056, 2014.
\newblock \href {http://dx.doi.org/10.1137/140956129}
  {\path{doi:10.1137/140956129}}.

\bibitem{CastroZuazua05}
C.~Castro and E.~Zuazua.
\newblock Unique continuation and control for the heat equation from an
  oscillating lower dimensional manifold.
\newblock {\em SIAM J. Control Optim.}, 43(4):1400--1434, 2005.
\newblock \href {http://dx.doi.org/10.1137/S0363012903430317}
  {\path{doi:10.1137/S0363012903430317}}.

\bibitem{ChSilvaRosierZuazua14}
F.~W. Chaves-Silva, L.~Rosier, and E.~Zuazua.
\newblock Null controllability of a system of viscoelasticity with a moving
  control.
\newblock {\em J. Math. Pures Appl.}, 101(2):198--222, 2014.
\newblock \href {http://dx.doi.org/10.1016/j.matpur.2013.05.009}
  {\path{doi:10.1016/j.matpur.2013.05.009}}.

\bibitem{DaiFlet05}
Yu-H. Dai and R.~Fletcher.
\newblock New algorithms for singly linearly constrained quadratic programs
  subject to lower and upper bounds.
\newblock {\em Math. Program.}, 106(3, Ser. A):403--421, 2006.
\newblock \href {http://dx.doi.org/10.1007/s10107-005-0595-2}
  {\path{doi:10.1007/s10107-005-0595-2}}.

\bibitem{Demetriou10}
M.~A. Demetriou.
\newblock Guidance of mobile actuator-plus-sensor networks for improved control
  and estimation of distributed parameter systems.
\newblock {\em IEEE Trans. Automat. Control}, 55(7):1570--1584, 2010.
\newblock \href {http://dx.doi.org/10.1109/TAC.2010.2042229}
  {\path{doi:10.1109/TAC.2010.2042229}}.

\bibitem{Gamk78}
R.~V. Gamkrelidze.
\newblock {\em Principles of Optimal Control Theory}.
\newblock Plenum Press, 1978.
\newblock \href {http://dx.doi.org/10.1007/978-1-4684-7398-8}
  {\path{doi:10.1007/978-1-4684-7398-8}}.

\bibitem{Hautus70}
M.~L.~J. Hautus.
\newblock Stabilization controllability and observability of linear autonomous
  systems.
\newblock {\em Indag. Math. (Proceedings)}, 73:448--455, 1970.
\newblock \href {http://dx.doi.org/10.1016/S1385-7258(70)80049-X}
  {\path{doi:10.1016/S1385-7258(70)80049-X}}.

\bibitem{He03}
Y.~He.
\newblock Two-level method based on finite element and {C}rank-{N}icolson
  extrapolation for the time-dependent {N}avier-{S}tokes equations.
\newblock {\em SIAM J. Numer. Anal.}, 41(4):1263--1285, 2003.
\newblock \href {http://dx.doi.org/10.1137/S0036142901385659}
  {\path{doi:10.1137/S0036142901385659}}.

\bibitem{HinterKunisch06}
M.~Hinterm\"{u}ller and K.~Kunisch.
\newblock Path-following methods for a class of constrained minimization
  problems in function space.
\newblock {\em SIAM J. Optim.}, 17(1):159--187, 2006.
\newblock \href {http://dx.doi.org/10.1137/040611598}
  {\path{doi:10.1137/040611598}}.

\bibitem{Khapalov94}
A.~Khapalov.
\newblock {$L^\infty$}-exact observability of the heat equation with scanning
  pointwise sensor.
\newblock {\em SIAM J. Control Optim.}, 32(4):1037--1051, 1994.
\newblock \href {http://dx.doi.org/10.1137/S036301299222737X}
  {\path{doi:10.1137/S036301299222737X}}.

\bibitem{Khapalov99}
A.~Khapalov.
\newblock Approximate controllability and its well-posedness for the semilinear
  reaction-diffusion equation with internal lumped controls.
\newblock {\em ESAIM Control, Optim. Calc. Var.}, 4:83--98, 1999.
\newblock \href {http://dx.doi.org/10.1051/cocv:1999104}
  {\path{doi:10.1051/cocv:1999104}}.

\bibitem{Khapalov01}
A.~Khapalov.
\newblock Mobile point controls versus locally distributed ones for the
  controllability of the semilinear parabolic equation.
\newblock {\em SIAM J. Control Optim.}, 40(1):231--252, 2001.
\newblock \href {http://dx.doi.org/10.1137/S0363012999358038}
  {\path{doi:10.1137/S0363012999358038}}.

\bibitem{KunRod18-cocv}
K.~Kunisch and S.~S. Rodrigues.
\newblock Explicit exponential stabilization of nonautonomous linear
  parabolic-like systems by a finite number of internal actuators.
\newblock {\em ESAIM Control Optim. Calc. Var.}, 2018.
\newblock \href {http://dx.doi.org/10.1051/cocv/2018054}
  {\path{doi:10.1051/cocv/2018054}}.

\bibitem{KunischSouza18}
K.~Kunisch and D.~A Souza.
\newblock On the one-dimensional nonlinear monodomain equations with moving
  controls.
\newblock {\em J. Math. Pures Appl.}, 117:94--122, 2018.
\newblock \href {http://dx.doi.org/10.1016/j.matpur.2018.05.003}
  {\path{doi:10.1016/j.matpur.2018.05.003}}.

\bibitem{MartinRosierRouchon13}
Ph. Martin, L.~Rosier, and P.~Rouchon.
\newblock Null controllability of the structurally damped wave equation with
  moving control.
\newblock {\em SIAM J. Control Optim.}, 51(1):660--684, 2013.
\newblock \href {http://dx.doi.org/10.1137/110856150}
  {\path{doi:10.1137/110856150}}.

\bibitem{PhanRod18-mcss}
D.~Phan and S.~S. Rodrigues.
\newblock Stabilization to trajectories for parabolic equations.
\newblock {\em Math. Control Signals Syst.}, 30(2):11, 2018.
\newblock \href {http://dx.doi.org/10.1007/s00498-018-0218-0}
  {\path{doi:10.1007/s00498-018-0218-0}}.

\bibitem{Rod06}
S.~S. Rodrigues.
\newblock {N}avier--{S}tokes equation on the {R}ectangle: Controllability by
  means of low modes forcing.
\newblock {\em J. Dyn. Control Syst.}, 12(4):517--562, 2006.
\newblock \href {http://dx.doi.org/10.1007/s10883-006-0004-z}
  {\path{doi:10.1007/s10883-006-0004-z}}.

\bibitem{Rod-Thesis08}
S.~S. Rodrigues.
\newblock {\em Methods of Geometric Control Theory in Problems of Mathematical
  Physics}.
\newblock PhD Thesis. Universidade de Aveiro, Portugal, 2008.
\newblock URL: \url{http://hdl.handle.net/10773/2931}.

\bibitem{Rod18}
S.~S. Rodrigues.
\newblock Feedback boundary stabilization to trajectories for {3D}
  {N}avier--{S}tokes equations.
\newblock {\em Appl. Math. Optim.}, 2018.
\newblock \href {http://dx.doi.org/10.1007/s00245-017-9474-5}
  {\path{doi:10.1007/s00245-017-9474-5}}.

\bibitem{Rod20-eect}
S.~S. Rodrigues.
\newblock Semiglobal exponential stabilization of nonautonomous semilinear
  parabolic-like systems.
\newblock {\em Evol. Equ. Control Theory}, 9(3):635--672, 2020.
\newblock \href {http://dx.doi.org/10.3934/eect.2020027}
  {\path{doi:10.3934/eect.2020027}}.

\bibitem{RodSturm20}
S.~S. Rodrigues and K.~Sturm.
\newblock On the explicit feedback stabilization of one-dimensional linear
  nonautonomous parabolic equations via oblique projections.
\newblock {\em IMA J. Math. Control Inform.}, 37(1):175--207, 2020.
\newblock \href {http://dx.doi.org/10.1093/imamci/dny045}
  {\path{doi:10.1093/imamci/dny045}}.

\bibitem{Shirikyan07}
A.~Shirikyan.
\newblock Exact controllability in projections for three-dimensional
  {N}avier--{S}tokes equations.
\newblock {\em Ann. Inst. H. Poincar\'e Anal. Non Lin\'eaire}, 24(4):521--537,
  2007.
\newblock \href {http://dx.doi.org/10.1016/j.anihpc.2006.04.002}
  {\path{doi:10.1016/j.anihpc.2006.04.002}}.

\bibitem{TucsnakWeiss09}
M.~Tucsnak and G.~Weiss.
\newblock {\em Observation and Control for Operator Semigroups}.
\newblock Birkh\"auser Basel, Basel, 2009.
\newblock \href {http://dx.doi.org/10.1007/978-3-7643-8994-9}
  {\path{doi:10.1007/978-3-7643-8994-9}}.

\bibitem{Zabczyk92}
J.~Zabczyk.
\newblock {\em Mathematical Control Theory: an Introduction}.
\newblock Systems Control Found. Appl. Birkh\"auser, Boston, 1992.
\newblock \href {http://dx.doi.org/10.1007/978-0-8176-4733-9}
  {\path{doi:10.1007/978-0-8176-4733-9}}.

\end{thebibliography}
